\documentclass[10pt]{article}
\usepackage[T1]{fontenc}
\usepackage[margin=1in]{geometry}
\usepackage[utf8]{inputenc}

\usepackage{microtype}
\usepackage{graphicx}
\usepackage{subfig}
\usepackage{caption}
\usepackage{booktabs} %
\usepackage{mathtools}
\usepackage{dsfont}
\usepackage{amsmath}
\usepackage{amsmath,amsfonts,amssymb}
\usepackage[colorlinks=true,linkcolor=blue,citecolor=red,hyperfootnotes=false]{hyperref}
\usepackage[capitalise]{cleveref}

\title{Learning Decentralized Linear Quadratic Regulators\\ with $\sqrt{T}$ Regret}

\author{%
Lintao Ye%
\thanks{School of Artificial Intelligence and Automation at the Huazhong University of Science and Technology, Wuhan, China; \texttt{yelintao93@hust.edu.cn,chiming1189@gmail.com}.}
\and
Ming Chi%
\footnotemark[1]
\and
Ruiquan Liao
\thanks{College of Petroleum Engineering, Yangtze University, Wuhan, China; \texttt{liaoruiquan@263.net}.}
\and
Vijay Gupta%
\thanks{Elmore Family School of Electrical and Computer Engineering at Purdue University, West Lafayette, IN, USA; \texttt{gupta869@purdue.edu}}
}

\RequirePackage{mathtools}
\RequirePackage{dsfont}
\RequirePackage{delimset}

\usepackage{amsthm}
\usepackage{amssymb}
\usepackage{thmtools}

\DeclareMathOperator*{\argmin}{arg\,min}

\newtheorem{claim}{Claim}

\newtheorem{theorem}{Theorem}

\newtheorem{lemma}{Lemma}
\newtheorem{remark}{Remark}
\newtheorem{definition}{Definition}
\newtheorem{proposition}{Proposition}
\newtheorem{example}{Example}
\newtheorem{assumption}{Assumption}

\newcommand{\W}{\mathcal{W}}
\newcommand{\ind}{\mathds{1}}

\newcommand{\K}{\mathcal{K}}
\newcommand{\BS}{\mathbb{S}}

\newcommand{\T}{\mathcal{T}}
\newcommand{\CL}{\mathcal{L}}
\newcommand{\A}{\mathcal{A}}
\newcommand{\CN}{\mathcal{N}}
\newcommand{\D}{\mathcal{D}}
\newcommand{\CO}{\mathcal{O}}
\newcommand{\R}{\mathbb{R}}
\newcommand{\CR}{\mathcal{R}}

\newcommand{\G}{\mathcal{G}}
\newcommand{\Z}{\mathbb{Z}}
\newcommand{\E}{\mathbb{E}}
\newcommand{\F}{\mathcal{F}}
\newcommand{\CE}{\mathcal{E}}
\newcommand{\CH}{\mathcal{H}}
\newcommand{\Prob}{\mathbb{P}}
\newcommand{\CP}{\mathcal{P}}
\newcommand{\U}{\mathcal{U}}
\newcommand{\V}{\mathcal{V}}

\newcommand{\I}{\mathcal{I}}

\newcommand{\tr}{\text{Tr}}

\DeclarePairedDelimiter{\nm}{\lVert}{\rVert}

\usepackage{enumitem}
\usepackage[numbers]{natbib}
\usepackage{crossreftools}
\usepackage{algorithm}
\usepackage[noend]{algpseudocode}
\algnewcommand{\LineComment}[1]{\State \(\triangleright\) #1}

\algnewcommand{\IfThenElse}[3]{%
  \State \algorithmicif\ #1\ \algorithmicthen\ #2\ \algorithmicelse\ #3}

\begin{document}

\maketitle

\begin{abstract}
We propose an online learning algorithm that adaptively designs decentralized linear quadratic regulators when the system model is unknown a priori and new data samples from a single system trajectory become progressively available. The algorithm  uses a disturbance-feedback representation of state-feedback controllers coupled with online convex optimization with memory and delayed feedback. Under the assumption that the system is stable or given a known stabilizing controller, we show that our controller enjoys an expected regret that scales as $\sqrt{T}$ with the time horizon $T$ for the case of partially nested information pattern. For more general information patterns, the optimal controller is unknown even if the system model is known. In this case, the regret of our controller is shown with respect to a linear sub-optimal controller. We validate our theoretical findings using numerical experiments.
\end{abstract}

\section{Introduction}
A fundamental challenge in decentralized control is that a local controller at a subsystem of the networked system may have access only to a subset of the global state information (e.g.,~\cite{shah2013cal,lamperski2012dynamic}). Control design under specified information patterns has been widely studied (e.g., \cite{ho1972team,rotkowitz2011nearest,lamperski2015optimal,feng2019exponential,shin2023near}). Finding an optimal controller under general information constraints is NP-hard (e.g., \cite{witsenhausen1968counterexample,blondel2000survey}).  Much work has been done to identify special information constraints~\cite{ho1972team,lamperski2012dynamic,lamperski2015optimal,rotkowitz2005characterization,rotkowitz2011nearest} that yield tractable optimal solutions. In particular, the partially nested information constraint~\cite{ho1972team,lamperski2012dynamic,lamperski2015optimal} assumes that the information propagates at least as fast as dynamics in networked systems consisting of multiple interconnected subsystems \cite{ho1972team}.

However, almost all existing work assumes the knowledge of the system model when designing the control policy.  In this paper, we study controller design for decentralized Linear Quadratic Regulator (LQR) when the system model is unknown. In this problem, the controller needs to be designed in an online manner utilizing  data samples that become available from a single system trajectory. We propose a model-based learning algorithm that regulates the system while learning the system matrices (e.g., \cite{dean2018regret,tu2019gap,cassel2020logarithmic}). As is standard in online learning, we measure the performance of our controller design algorithm using the notion of regret (e.g., \cite{bubeck2011introduction}), which compares the cost incurred by the controller designed using an online algorithm to that incurred by the optimal decentralized controller. Perhaps surprisingly, despite the existence of the information constraint, our regret bound matches with the one provided in \cite{abbasi2011regret,cohen2019learning} for centralized LQR, which has also been shown to be the best regret guarantee that can be achieved by any online learning algorithms (up to logarithmic factors in the time horizon $T$ and other problem constants) \cite{simchowitz2020naive,cassel2021online,ziemann2022regret}.

\subsection*{Related Work}
For control design in centralized LQR, many approaches have now appeared in the literature. Offline learning algorithms solve the case when data samples from the system trajectories can be collected offline before the control policy is designed (e.g., \cite{dean2020sample,zheng2021sample}). The metric of interest here is sample complexity which relates the gap between the performances of the proposed control policy and the optimal control policy to the number of data samples from the system trajectories that are collected offline. Online learning algorithms solve the case when data samples from a single system trajectory become available in an online manner and the control policy needs to be designed simultaneously (e.g., \cite{abbasi2011regret,cohen2019learning,simchowitz2020improper,cassel2021online,chen2021black,lale2020logarithmic,ouyang2017learning}). The metric of interest here is the regret of the adaptive approach. These approaches can also be classified into model-based learning algorithms in which a system model is first estimated from data samples from the system trajectories and then a control policy is designed based on the estimated system model (e.g., \cite{cassel2021online,sarkar2021finite,dean2018regret,mania2019certainty,xin2022identifying}), or model-free learning algorithms in which a control policy is directly obtained from the system trajectories (e.g., \cite{fazel2018global,malik2020derivative,bu2019lqr,gravell2020learning,zhang2020policy,mohammadi2021convergence}) using zeroth-order optimization methods (e.g., \cite{ghadimi2013stochastic}). We focus on model-based online learning of LQR controllers, but in a decentralized control problem. 

As compared to the centralized case, there are only a few results available for solving decentralized linear quadratic control problems with information constraints and unknown system models. In \cite{furieri2020learning}, the authors assumed a quadratic invariance condition on the information constraint in a decentralized output-feedback linear quadratic control problem over a finite horizon and proposed a model-free offline learning algorithm along with a sample complexity analysis. In \cite{li2019distributed}, the authors proposed a model-free offline learning algorithm for multi-agent decentralized LQR over an infinite horizon, where each agent (i.e., controller) has access to a subset of the global state without delay. They showed that their (consensus-based) algorithm converges to a control policy that is a stationary point of the objective function in the LQR problem. In \cite{fattahi2020efficient}, the authors studied model-based offline learning for LQR with subspace constraints on the closed-loop responses, which may not lead to controllers that satisfy the desired information constraints (e.g., \cite{zheng2020equivalence}). Finally, in \cite{ye2021sample}, the authors considered model-based offline learning for decentralized LQR with a partially nested information constraint using certainty equivalence approach and analyzed the sample complexity of the learning algorithm.

\subsection*{Contributions}
Our approach and contributions can be summarized as follows.
\begin{itemize}[leftmargin=*]
\item We begin by assuming that the information pattern is partially nested, for which an optimal decentralized controller can be designed if the model is known~\cite{lamperski2015optimal}. In Section~\ref{sec:DFC}, we show that this optimal controller can be cast into a Disturbance-Feedback Controller (DFC) which has been utilized in learning centralized LQR~\cite{anava2015online,simchowitz2020improper,li2021safe}. Adapting the DFC structure to a decentralized controller requires more care since the resulting DFC needs to respect the prescribed information pattern.
	
\item In Section~\ref{sec:decentralized online control alg}, we present our online decentralized control algorithm, which first identifies a system model based on a single system trajectory, and then adaptively designs a control policy (with the DFC structure and the partially nested information pattern) based on the estimated system model. The control policy utilizes a novel Online Convex Optimization (OCO) algorithm with memory and delayed feedback. 
	
\item In Section~\ref{sec:regret analysis}, we prove that the expected regret of our online decentralized control algorithm scales as $\sqrt{T}$ under the assumption that the system is stable, where $T$ is the length of the time horizon. 
We first analyze the regret bounds for the general OCO algorithm with memory and delayed feedback (which works for general OCO problems and is of independent interest), and then specialize the results to our setting. Surprisingly, our regret bound matches with the one provided in \cite{abbasi2011regret,cohen2019learning} for learning centralized LQR in terms of $T$, which has been shown to be the best regret guarantee (up to logarithmic factors in $T$ and other constants) \cite{simchowitz2020naive,cassel2021online,chen2021black}. Our regret result is in stark contrast to the sample complexity result in \cite{ye2021sample} which shows a degradation for decentralized LQR compared to the centralized case \cite{mania2019certainty}.
	
\item 
In Section~\ref{sec:extensions}, we show that all of our results can be extended to a general information pattern and stabilizable systems. Since the optimal decentralized controller with a general information pattern is unknown even if the model is known, the regret analysis compares our design to a particular sub-optimal controller. 
\end{itemize}

\paragraph*{Notation and Terminology}
The sets of integers and real numbers are denoted as $\mathbb{Z}$ and $\mathbb{R}$, respectively. The set of integers (resp., real numbers) that are greater than or equal to $a\in\mathbb{R}$ is denoted as $\mathbb{Z}_{\ge a}$ (resp., $\R_{\ge a}$). The space of $m$-dimensional real vectors is denoted by $\mathbb{R}^{m}$ and the space of $m\times n$ real matrices is denoted by $\mathbb{R}^{m\times n}$. For a matrix $P\in\R^{n\times n}$, let $P^{\top}$, $\tr(P)$, and $\{\sigma_i(P):i\in\{1,\dots,n\}\}$ be its transpose, trace, and the set of singular values, respectively. Without loss of generality, let the singular values of $P$ be ordered as $\sigma_1(P)\ge\cdots\ge\sigma_n(P)$. Let $\norm{\cdot}$ denote the Euclidean norm for a vector or spectral norm for a matrix, i.e., $\norm{P}=\sqrt{\sigma_1(P)}$ for $P\in\R^{n\times n}$ and $\norm{x}=\sqrt{x^{\top}x}$ for $x\in\R^n$. Let $\norm{P}_F=\sqrt{\tr(PP^{\top})}$ denote the Frobenius norm of $P\in\R^{n\times m}$. A positive semidefinite matrix $P$ is denoted by $P\succeq0$, and $P\succeq Q$ if and only if $P-Q\succeq0$. Let $\BS_+^n$ (resp., $\BS_{++}^n$) denote the set of $n\times n$ positive semidefinite (resp., positive definite) matrices. Let $I_n$ denote an $n\times n$ identity matrix; the subscript is omitted if the dimension can be inferred from the context. Given any integer $n\ge1$, we define $[n]=\{1,\dots,n\}$. The cardinality of a finite set $\mathcal{A}$ is denoted by $|\mathcal{A}|$. Let $\CN(0,\Sigma)$ denote a Gaussian distribution with mean $0$ and covariance $\Sigma\succeq0$. For a vector $x$, let ${\tt dim}(x)$ denote its dimension. Let $\sigma(\cdot)$ denote the sigma field generated by the corresponding random vectors. For $[n]$ and $P_i\in\R^{n\times m_i}$ for all $i\in[n]$, denote $[P_i]_{i\in[n]}=\begin{bmatrix}P_1&\cdots&P_n\end{bmatrix}$.

\section{Problem Formulation and Preliminary Results}\label{sec:preliminaries and problem formulation}

\subsection{Decentralized LQR with Sparsity and Delay Constraints}\label{sec:dist LQR known matrices}
Consider a networked system with $p\in\Z_{\ge1}$ dynamically coupled linear-time-invariant (LTI) subsystems. Let $\V=[p]$ be a set that contains the indices of all the $p$ subsystems. Specifically, denoting the state, input and disturbance of subsystem $i\in[p]$ at time $t$ as $x_{t,i}\in\R^{n_i}$, $u_{t,i}\in\R^{m_i}$, and $w_{t,i}\in\R^{n_i}$, respectively, the dynamics of subsystem $i$ is given by
\begin{equation}
\label{eqn:system for node i}
x_{t+1,i} = \Big(\!\sum_{j\in\CN_i}\!A_{ij}x_{t,j}+B_{ij}u_{t,j}\Big)+w_{t,i},\ \forall i\in\V,
\end{equation}
where $\CN_i\subseteq[p]$ denotes the set of subsystems whose states and inputs {\it directly} affect the state of subsystem $j$, $A_{ij}\in\R^{n_i\times n_j}$ and $B_{ij}\in\R^{n_i\times m_j}$ are given coupling matrices, and $w_{t,i}\sim\CN(0,\sigma_w^2I_{n_i})$ is a white Gaussian noise process for $t\in\Z_{\ge0}$ with $\sigma_w\in\R_{>0}$. We assume that $w_{t_1,i}$ and $w_{t_2,j}$ are independent for all $i,j\in\V$ with $i\neq j$ and for all $t_1,t_2\in\Z_{\ge0}$, which implies that $w_t\sim\CN(0,\sigma_w^2I_n)$ is a white Gaussian noise process for $t\in\Z_{\ge0}$.\footnote{The analysis can be extended to the case when $w_i(t)$ is assumed to be a zero-mean white Gaussian noise process with covariance $W\in\BS_{++}^{n_i}$. In that case, our analysis will depend on $\max_{i\in\V}\sigma_1(W_i)$ and $\min_{i\in\V}\sigma_n(W_i)$.} For simplicity, we assume throughout that $n_i\ge m_i$ for all $i\in\V$. We can rewrite Eq.~\eqref{eqn:system for node i} as
\begin{equation}
\label{eqn:dynamics for x_i(t)}
x_{t+1,i} = A_ix_{t,\CN_i} + B_iu_{t,\CN_i} + w_{t,i},\ \forall i\in\V,
\end{equation}
where $A_i\triangleq\begin{bmatrix}A_{ij}\end{bmatrix}_{j\in\CN_i}$, $B_i\triangleq\begin{bmatrix}B_{ij}\end{bmatrix}_{j\in\CN_i}$, $x_{t,\CN_i}\triangleq\begin{bmatrix}x_{t,j}^{\top}\end{bmatrix}^{\top}_{j\in\CN_i}$, and $u_{t,\CN_i}\triangleq\begin{bmatrix}u_{t,j}^{\top}\end{bmatrix}^{\top}_{j\in\CN_i}$, with $\CN_i=\{j_1,\dots,j_{|\CN_i|}\}$. Furthermore, letting $n=\sum_{i\in\V}n_i$ and $m=\sum_{i\in\V}m_i$, and defining $x_t=\begin{bmatrix}x_{t,i}^{\top}\end{bmatrix}^{\top}_{i\in\V}$, $u_t=\begin{bmatrix}u_{t,i}^{\top}\end{bmatrix}^{\top}_{i\in\V}$ and $w_t=\begin{bmatrix}w_{t,i}^{\top}\end{bmatrix}^{\top}_{i\in\V}$, we can write Eq.~\eqref{eqn:system for node i} into the following matrix form:
\begin{equation}\label{eqn:overall system}
    x_{t+1}=Ax_t + Bu_t + w_t,
\end{equation}
where the $(i,j)$th block of $A\in\R^{n\times n}$ (resp., $B\in\R^{n\times m}$), i.e., $A_{ij}$ (resp., $B_{ij}$) satisfies $A_{ij}=0$ (resp., $B_{ij}=0$) if $j\notin\CN_i$. Without loss of generality, we assume that $x_0=0$.

{\bf Information Structure.} A key difficulty in decentralized control is that the control input design must satisfy the constraints of a prescribed information flow among the subsystems in $[p]$. We can use a directed graph $\G(\V,\A)$ (with $\V=[p]$) to characterize the information flow under sparsity and delay constraints on the communication among the subsystems, where each node in $\G(\V,\A)$ represents a subsystem in $[p]$, and we assume that $\G(\V,\A)$ does not have self loops. Moreover, we associate any edge $(i,j)\in\A$ with a delay of either $0$ or $1$, further denoted as $i\xrightarrow[]{0}j$ or $i\xrightarrow[]{1}j$, respectively.\footnote{The framework described in this paper can also be used to handle more general delay values; see \cite{lamperski2015optimal} for a detailed discussion.} Then, we define the delay matrix corresponding to $\G(\V,\A)$ as $D\in\R^{p\times p}$ such that: (i) If $i\neq j$ and there is a directed path from $j$ to $i$ in $\G(\V,\A)$, then $D_{ij}$ is equal to the sum of delays along the directed path from node $j$ to node $i$ with the smallest accumulative delay; (ii) If $i\neq j$ and there is no directed path from $j$ to $i$ in $\G(\V,\A)$, then $D_{ij}=+\infty$; (iii) $D_{ii}=0$ for all $i\in\V$. For the remainder of this paper, we focus on the scenario where the information (e.g., state information) corresponding to subsystem $j\in\V$ can propagate to subsystem $i\in\V$ with a delay of $D_{ij}$ (in time), if and only if there exists a directed path from $j$ to $i$ with an accumulative delay of $D_{ij}$. As argued in \cite{lamperski2015optimal}, we assume that there is no directed cycle with zero accumulative delay; otherwise, one can first collapse all the nodes in such a directed cycle into a single node, and equivalently consider the resulting directed graph in the framework described above. 

Thus, considering any $i\in\V$ and any $t\in\Z_{\ge0}$, the state information that is available to the controller corresponding to $i\in\V$ is given by 
\begin{equation}\label{eqn:info set}
    \I_{t,i} = \{x_{k,j}:j\in\V_i,0\le k\le t-D_{ij}\},
\end{equation}
where $\V_i\triangleq\{j\in\V:D_{ij}\neq+\infty\}$. In the sequel, we refer to $\I_{t,i}$ as the {\it information set} of controller $i\in\V$ at time $t\in\Z_{\ge0}$. Note that $\I_{t,i}$ contains the states of the subsystems in $\V$ such that there is sufficient time for these state values to reach subsystem $i\in\V$ at time $t\in\Z_{\ge0}$, in accordance with the sparsity and delay constraints described above. Now, based on the information set $\I_{t,i}$, we define $\pi_i(\I_{t,i})$ to be the set that consists of all the policies $u_{t,i}$ that map the states in $\I_{t,i}$ to a control input at node $i$. 

Similarly to \cite{lamperski2015optimal,ye2021sample,yu2022online}, we make the following assumption on the information structure associated with system~\eqref{eqn:system for node i}. Later in Section~\ref{sec:extensions}, we will show how to extend our analysis to the setting when Assumption~\ref{ass:info structure} does not hold. 
\begin{assumption}\label{ass:info structure}
For any $i\in\V$, it holds that $D_{ij}\le1$ for all $j\in\CN_i$, where $\CN_i$ is given in Eq.~\eqref{eqn:system for node i}.
\end{assumption}
The above assumption ensures that the state of subsystem $i\in\V$ is affected by the state and input of subsystem $j\in\V$, if and only if there is a communication link with a delay of at most $1$ from subsystem $j$ to $i$ in $\G(\V,\A)$. This assumption ensures that the information structure associated with the system given in Eq.~\eqref{eqn:system for node i} is {\em partially nested} \cite{ho1972team}, which is a condition for tractability of decentralized (or distributed) control problems that is frequently used in the literature (e.g., \cite{lamperski2015optimal,shah2013cal} and the references therein). This assumption is also satisfied in networked systems where information propagates at least as fast as dynamics. To illustrate the above arguments, we introduce Example~\ref{exp:running example}.
\begin{figure}[htbp]
    \centering
    \includegraphics[width=0.3\linewidth]{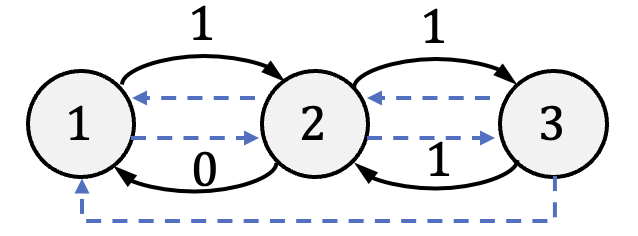} 
    \caption{ The directed graph of Example~\ref{exp:running example}. Node $i\in\V$ represents a subsystem with state $x_i(t)$, a solid edge $(i,j)\in\A$ is labeled with the information propagation delay from  $i$ to $j$, and the dotted edges represent the coupling of the dynamics among the nodes in $\V$.}
\label{fig:directed graph}
\end{figure}

\begin{example}
\label{exp:running example}
Consider a directed graph $\G(\V,\A)$ given in Fig.~\ref{fig:directed graph}, where $\V=\{1,2,3\}$ and each directed edge is associated with a delay of $0$ or $1$. The corresponding LTI system is then given by 
\begin{equation}
\label{eqn:LTI in exp}
\begin{bmatrix}
x_1(t+1)\\x_2(t+1)\\x_3(t+1)
\end{bmatrix}=
\begin{bmatrix}
A_{11} & A_{12} & A_{13}\\
A_{21} & A_{22} & A_{23}\\
0 & A_{32} & A_{33}
\end{bmatrix}
\begin{bmatrix}
x_1(t)\\x_2(t)\\x_3(t)
\end{bmatrix}
+\begin{bmatrix}
B_{11} & B_{12} & B_{13}\\
B_{21} & B_{22} & B_{23}\\
0 & B_{32} & B_{33}
\end{bmatrix}
\begin{bmatrix}
u_1(t)\\u_2(t)\\u_3(t)
\end{bmatrix}
+\begin{bmatrix}
w_1(t)\\w_2(t)\\w_3(t)
\end{bmatrix}.
\end{equation}
\end{example}

{\bf Decentralized LQR and Its Solution.} The decentralized LQR problem can then be posed as follows: 
\begin{equation}
\label{eqn:dis LQR obj}
\begin{split}
&\min_{u_0,u_1,\dots}\lim_{T\to\infty}\E\Big[\frac{1}{T}\sum_{t=0}^{T-1}(x_t^{\top}Qx_t+u_t^{\top}Ru_t)\Big]\\
s.t.\ &x_{t+1}=Ax_t+Bu_t+w_t,\\
&u_{t,i}\in\pi_i(\I_{t,i}),\ \forall i\in\V,\forall t\in\Z_{\ge0},
\end{split}
\end{equation}
with cost matrices $Q\in\BS_+^n$ and $R\in\BS_{++}^m$ and the expectation taken with respect to $w_t$ for all $t\in\Z_{\ge0}$. 

Following the steps in, e.g.,  \cite{lamperski2015optimal}, for solving \eqref{eqn:dis LQR obj}, we first construct an information graph $\CP(\U,\CH)$. Considering any directed graph $\G(\V,\A)$ with $\V=[p]$, and the delay matrix $D\in\R^{p\times p}$ as we described above, let us first define $s_{k,j}$ to be the set of nodes in $\G(\V,\A)$ that are reachable from node $j$ within $k$ time steps, i.e., $s_{k,j}=\{i\in\V:D_{ij}\le k\}$. The information graph $\CP(\U,\CH)$ is then constructed as
\begin{equation}
\label{eqn:def of info graph}
\begin{split}
\U&=\{s_{k,j}:k\ge0,j\in\V\},\\
\CH&=\{(s_{k,j},s_{k+1,j}):k\ge0, j\in\V\}.
\end{split}
\end{equation}
We see from \eqref{eqn:def of info graph} that each node $s\in\U$ corresponds to a set of nodes from $\V=[p]$ in the original directed graph $\G(\V,\A)$. As discussed in \cite{lamperski2015optimal}, the nodes in $\U$ specify a partition of the noise history in \eqref{eqn:overall system} (i.e., $\{w_t\}_{t\ge0}$) with respect to the information constraint, and an edge $(s_{k,j},s_{k+1,j})\in\CH$ indicates that the set of noise terms corresponding to $s_{k,j}$ belongs to that corresponding to $s_{k+1,j}$. If there is an edge from $s$ to $r$ in $\CP(\U,\CH)$, we also denote the edge as $s\to r$. Additionally, considering any $s_{0,i}\in\U$, we write $w_{i}\xrightarrow[]{}s_{0,i}$ to indicate that the noise $w_{t,i}$ is injected to node $i\in\V$ for $t\in\Z_{\ge0}$.\footnote{Since we have assumed that there is no directed cycle with zero accumulative delay in $\CP(\U,\CH)$, one can show that $w_{t,i}$ is the unique noise term such that $w_{i}\rightarrow s_{0,i}$} The information graph $\CP(\U,\CH)$ constructed from the directed graph $\G(\V,\A)$ in Fig.~\ref{fig:directed graph} is given in Fig.~\ref{fig:info graph}. 

\begin{figure}[htbp]
    \centering
    \includegraphics[width=0.26\linewidth]{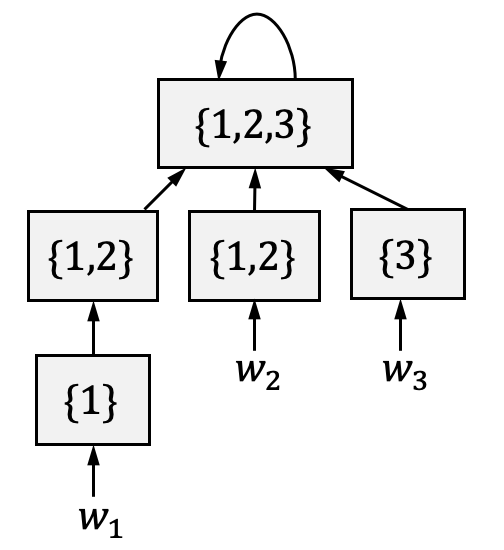} 
    \caption{The information graph of Example~\ref{exp:running example}. Each node in the information graph is a subset of the nodes in the directed graph given in Fig.~\ref{fig:directed graph}.}
\label{fig:info graph}
\end{figure}

Throughout this paper and without loss of generality, we assume that the elements in $\V=[p]$ are ordered in an increasing manner, and that the elements in $s$ are also ordered in an increasing manner for all $s\in\U$. Moreover, for any $\mathcal{X},\mathcal{Y}\subseteq\V$, we use $A_{\mathcal{X}\mathcal{Y}}$ (or $A_{\mathcal{X},\mathcal{Y}}$) to denote the submatrix of $A$ that corresponds to the nodes of the directed graph $\G(\V,\A)$ contained in $\mathcal{X}$ and $\mathcal{Y}$. For example, $A_{\{1\},\{1,2\}}=\begin{bmatrix}A_{11} & A_{12}\end{bmatrix}$. In the sequel, we will also use similar notations to denote submatrices of other matrices, including $B$, $Q$, $R$ and the identity matrix $I$. We will make the following standard assumption (see, e.g., \cite{lamperski2015optimal}).
\begin{assumption}
\label{ass:pairs}
For any $s\in\U$ that has a self loop, the pair $(A_{ss},B_{ss})$ is stabilizable and the pair $(A_{ss},C_{ss})$ is detectable, where $C_{ss}$ is such that $Q_{ss}=C_{ss}^{\top}C_{ss}$.
\end{assumption}

Leveraging the partial nestedness from Assumption~\ref{ass:info structure} and the information graph constructed above, \cite{lamperski2015optimal} gives a closed-form expression for the optimal solution to \eqref{eqn:dis LQR obj}.
\begin{lemma}\label{lemma:opt solution}
\cite[Corollary~4]{lamperski2015optimal} Suppose Assumptions~\ref{ass:info structure}-\ref{ass:pairs} hold. Consider the problem given in \eqref{eqn:dis LQR obj}, and let $\CP(\U,\CH)$ be the associated information graph. Suppose Assumption \ref{ass:pairs} holds. For all $r\in\U$, define matrices $P_r$ and $K_r$ recursively as
\begin{align}
K_r &= -(R_{rr}+B_{sr}^{\top} P_s B_{sr}^{\top})^{-1} B_{sr}^{\top} P_s A_{sr},\label{eqn:set of DARES K}\\
P_r &= Q_{rr}+K_r^{\top}R_{rr}K_r+(A_{sr}+B_{sr}K_r)^{\top}P_s(A_{sr}+B_{sr}K_r),\label{eqn:set of DARES P}
\end{align}
where for each $r\in\U$, $s\in\U$ is the unique node such that $r\rightarrow s$. In particular, for any $s\in\U$ that has a self loop, the matrix $P_s$ is the unique positive semidefinite solution to the Riccati equation given by Eq.~\eqref{eqn:set of DARES P} , and the matrix $A_{ss}+B_{ss}K_s$ is stable. The optimal solution to \eqref{eqn:dis LQR obj} is then given by
\begin{align}
\zeta_{t+1,s} &= \sum_{r\rightarrow s}(A_{sr}+B_{sr}K_r)\zeta_{t,r} + \sum_{w_i\rightarrow s}I_{s,\{i\}}w_{t,i},\label{eqn:dynamics of zeta}\\
u_{t,i}^{\star} &= \sum_{r\ni i}I_{\{i\},r}K_r\zeta_{t,r},\label{eqn:exp for u^star}
\end{align}
for all $t\in\Z_{\ge0}$, where $\zeta_{t,s}$ is an internal state initialized with $\zeta_{0,s}=\sum_{w_i\rightarrow s}I_{s,\{i\}}x_{0,i}=0$ for all $s\in\U$. The corresponding optimal cost of \eqref{eqn:dis LQR obj}, denoted as $J_{\star}$, is given by
\begin{equation}
\label{eqn:opt J}
J_{\star}=\sigma_w^2\sum_{\substack{i\in\V\\ w_i\rightarrow s}}\tr\big(I_{\{i\},s}P_sI_{s,\{i\}}\big).
\end{equation}
\end{lemma}
Note that Eqs.~\eqref{eqn:set of DARES K}-\eqref{eqn:set of DARES P} need to be computed for all $s\in\U$ to obtain $u_t^{\star}$, however, from \cite[Proposition~1]{lamperski2015optimal}, $|\U|\le p^2-p+1$, where $p$ is the number of nodes in $\G(\V,\A)$.

\subsection{Problem Considered in this Paper and Our Result}\label{sec:problem formulation}
We consider the problem of online learning for decentralized LQR with the information structure described in Section~\ref{sec:dist LQR known matrices} when the system model is unknown a priori.  Following \cite{shah2013cal,lamperski2015optimal,ye2021sample,yu2022online}, we assume that the cost matrices $Q,R$, the graph $\G(\V,\A)$ and the delay matrix $D$ are known. For any $t\in\Z_{\ge0}$ and any $i\in\V$, a decentralized online control algorithm first chooses $u_{t,i}^{\tt alg}\in\pi_i(\I_{t,i})$  based only on $x_0^{\tt alg},\dots,x_t^{\tt alg}$ observed so far and then incurs a cost $c(x_t^{\tt alg},u_t^{\tt alg})$, where $c(x,u)\triangleq x^{\top}Qx+u^{\top}Ru$ and $x_t^{\tt alg}$ is the state of system \eqref{eqn:overall system} when the inputs $u_0^{\tt alg},\dots,u_{t-1}^{\tt alg}$ designed by the algorithm were applied. We wish to find a decentralized online control algorithm that minimizes the regret defined as 
\begin{equation}
\label{eqn:regret}
{\tt Regret}=\sum_{t=0}^{T-1}\big(c(x_t^{\tt alg},u_t^{\tt alg}) - J_{\star}\big),
\end{equation}
which compares the cost incurred by the algorithm against the optimal cost $J_{\star}$ of \eqref{eqn:dis LQR obj} (given in Lemma~\ref{lemma:opt solution}) over the horizon $T\in\Z_{\ge1}$. The following standard result is standard, which will be useful in our later analysis.
\begin{lemma}
\label{lemma:lipschitz c}
Consider $Q\in\BS_{+}^m$ and $R\in\BS_{+}^n$. The quadratic function $c(x,u)=x^{\top}Qx+u^{\top}Ru$ is Lipschitz continuous such that 
\begin{equation*}
|c(x_1,u_1)-c(x_2,u_2)|\le2(R_x+R_u)\max\{\sigma_1(Q),\sigma_2(R)\}\big(\norm{x_1-x_2}+\norm{u_1-u_2}\big),
\end{equation*}
for all $x_1,x_2\in\R^n$ and all $u_1,u_2\in\R^m$ such that $\norm{x_1}\le R_x,\norm{x_2}\le R_x,\norm{u_1}\le R_u,\norm{u_2}\le R_u$.
\end{lemma}

{\bf Main result (informal).} {\it We design an algorithm for learning the decentralized LQR under the information constraint that achieves ${\tt Regret}=\tilde{\CO}(\sqrt{T})$, where $\tilde{\CO}(\cdot)$ hides polynomial factors in $\log T$ and other problem parameters.}

As we discussed in the introduction, $\tilde{\CO}(\sqrt{T})$ matches with the regret lower bound for learning centralized LQR which naturally holds for the decentralized setting.

\subsection{Summary of Main Symbols}\label{sec:summary of symbols}
To ease our presentation in the remainder of this paper, we summarize the main symbols used throughout. $i\in\V$ represents a subsystem from $\V=[p]$ or the corresponding controller of the subsystem. $\CP(\U,\CH)$ is the information graph constructed by \eqref{eqn:def of info graph} from the directed graph $\G(\V,\A)$ and $D_{\max}$ is the maximum accumulative delay along any path in $\G(\V,\A)$. $\D$ and $\D_0$ represent sets of DFC parameterized by some $M$. $u$, $x$ and $\zeta$, $w$ and $\eta$ refer to control input, system state, and internal state (see Lemma~\ref{lemma:opt solution}), disturbance and concatenate disturbance vector, respectively. Subscript $i\in\V$  (resp., $s\in\U$) denotes the node index in $\V$ (resp., $\U$). Superscript $M$ (resp., ${\tt alg}$) refers to the quantities corresponding to a DFC (resp., the proposed algorithm). For example, $x^{\tt alg}_t$ represents the state vector of system \eqref{eqn:overall system} when the control inputs $u_0^{\tt alg},\dots,u_{t-1}^{\tt alg}$ were applied, and $x_{t,i}^{\tt alg}$ refers to the states in $x_t^{\tt alg}\in\R^n$ that correspond to subsystem $i\in\V$. $R_{\phi}$ typically denotes an upper bound on $\phi\in\{u,x,w\}$. Using $\star$ in a symbol means the symbol represents some optimal quantity, e.g., $J_{\star}$ is the optimal cost of \eqref{eqn:dis LQR obj}. Finally, $\hat{\cdot}$ is used for estimated quantities, and $\varepsilon$  denotes estimation error.

\subsection{Disturbance-Feedback Controller}\label{sec:DFC}
For our analyses in the remaining of this paper, we cast the optimal control policy given by Lemma~\ref{lemma:opt solution} into a Disturbance-Feedback Controller (DFC) introduced in, e.g., \cite{zames1981feedback,lamperski2015optimal,simchowitz2020improper}). As we will see in later sections, it is crucial to leverage the DFC structure in our decentralized control policy design in order to achieve the desired $\tilde{\CO}(\sqrt{T})$ regret. To proceed, recall the information graph $\CP(\U,\CH)$ constructed in \eqref{eqn:def of info graph}. Let $\CL$ denote the set of all the leaf nodes in $\CP(\U,\CH)$, i.e.,
\begin{equation}
\label{eqn:set of leaf nodes}
\CL=\{s_{0,i}\in\U:i\in\V\}.
\end{equation}
For any $s\in\U$, we denote
\begin{equation}
\label{eqn:set of leaf nodes of s}
\CL_s=\{v\in\CL:v\rightsquigarrow s\},
\end{equation}
where we write $v\rightsquigarrow s$ if and only if there is a unique directed path from node $v$ to node $s$ in $\CP(\U,\CH)$. In other words, $\CL_s$ is the set of leaf nodes in $\CP(\U,\CH)$ that can reach $s$. Moreover, for any $v,s\in\U$ such that $v\rightsquigarrow s$, let $l_{vs}$ denote the length of the unique directed path from $v$ to $s$ in $\CP(\U,\CH)$; let $l_{vs}=0$ if $v=s$. Furthermore, let $\U_1$ (resp., $\U_2$) be the set of all the nodes in $\U$ that have (resp., do not have) a self loop. As shown in \cite{ye2021sample}, for any $s\in\U_1$, one can rewrite Eq.~\eqref{eqn:dynamics of zeta} as
\begin{equation}
\label{eqn:dynamics of zeta self loop}
\zeta_{t+1,s}=(A_{ss}+B_{ss}K_s)\zeta_{t,s}+\sum_{v\in\CL_s}H_{v,s}I_{v,\{j_v\}}w_{t-l_{vs},j_v},
\end{equation}
with $w_{j_v}\to v$, where
\begin{equation}
\label{eqn:H(v,s)}
H_{v,s} \triangleq (A_{sr_1}+B_{sr_1}K_{r_1})\cdots(A_{r_{l_{vs}-1}v}+B_{r_{l_{vs}-1}v}K_v),
\end{equation}
with $H_{v,s}=I$ if $v=s$, $K_r$ is given by Eq.~\eqref{eqn:set of DARES K} for all $r\in\U$, and $v,r_{l_{vs}-1},\dots,r_1,s$ are the nodes along the directed path from $v$ to $s$ in $\CP(\U,\CH)$. Note that we let $w_t=0$ for all $t<0$. Unrolling Eq.~\eqref{eqn:dynamics of zeta self loop}, and recalling that $\zeta_{0,s}=0$ for all $s\in\U$, we have 
\begin{align}\nonumber
\zeta_{t,s} &= \sum_{k=0}^{t-1}(A_{ss}+B_{ss}K_s)^{t-(k+1)}\sum_{v\in\CL_s}H_{v,s}I_{v,\{j_v\}}w_{k-l_{vs},j_v},\\
&=\sum_{k=0}^{t-1}(A_{ss}+B_{ss}K_s)^{t-(k+1)}\begin{bmatrix}H_{v,s}I_{v,\{j_v\}}\end{bmatrix}_{v\in\CL_s}\eta_{k,s},\ \forall s\in\U_1,\label{eqn:dynamics of zeta self loop unroll}
\end{align}
where 
\begin{equation}
\label{eqn:eta_k,s}
\eta_{k,s}=\begin{bmatrix}w_{k-l_{vs},j_v}^{\top}\end{bmatrix}^{\top}_{v\in\CL_s}.
\end{equation}
Similarly, for any $s\in\U_2$, one can rewrite Eq.~\eqref{eqn:dynamics of zeta} as
\begin{align}\nonumber
\zeta_{t,s} &= \sum_{v\in\CL_s}H_{v,s}I_{v,\{j_v\}}w_{t-1-l_{vs},j_v}\\
&=\begin{bmatrix}H_{v,s}I_{v,\{j_v\}}\end{bmatrix}_{v\in\CL_s}\eta_{t-1,s},\label{eqn:dynamics of zeta no self loop}
\end{align}
Using Eqs.~\eqref{eqn:dynamics of zeta self loop unroll} and \eqref{eqn:dynamics of zeta no self loop}, one can then show that $u^{\star}_t=\sum_{s\in\U}I_{\V,s}K_s\zeta_{t,s}$ given by Eq.~\eqref{eqn:dynamics of zeta} can be expressed as the following DFC:
\begin{align}
u_t^{\star} &= \sum_{s\in\U}I_{\V,s}\sum_{k=1}^{t}M_{\star,s}^{[k]}\eta_{t-k,s},\ \forall t\in\Z_{\ge0}
	\label{eqn:DFC of u_star}\\
	M_{\star,s}^{[k]}&=
\begin{cases}
		K_s(A_{ss}+B_{ss}K_s)^{k-1}\begin{bmatrix}H_{v,s}I_{v,\{j_v\}}\end{bmatrix}_{v\in\CL_s}\ \text{if}\ s\in\U_1 \textrm{ and } k\in[t], \\
		K_s\begin{bmatrix}H_{v,s}I_{v,\{j_v\}}\end{bmatrix}_{v\in\CL_s}\ \qquad\qquad\qquad\qquad\text{if}\ s\in\U_2 \textrm{ and }k=1,\\
		0\qquad\qquad\qquad\qquad\qquad\qquad\qquad \text{if}\ s\in\U_2 \textrm{ and }k\in\{2,\dots,t\}.
\end{cases}
\label{eqn:M_star self loop}
\end{align}
Note that $M_{\star,s}^{[k]}\in\R^{n_s\times n_{\CL_s}}$ for all $s\in\U$ and all $k\in[t]$, where $n_s=\sum_{i\in s}n_i$ and $n_{\CL_s}=\sum_{v\in\CL_s}n_v$. Hence, Eq.~\eqref{eqn:DFC of u_star} gives a disturbance-feedback representation of the optimal control policy $u_t^{\star}$. Note that $u_t^{\star}$ given in Eq.~\eqref{eqn:DFC of u_star} depends on {\it all} the past disturbances $w_0,\dots,w_{t-1}$. To mitigate the dependency on all the past disturbances, we further introduce a (truncated) DFC defined as follows (e.g., \cite{agarwal2019online}). We will later show in our theoretical analysis in Section~\ref{sec:regret analysis} that leveraging the DFC given by Definition~\ref{def:u_t M} in our algorithm design suffices to achieve the $\sqrt{T}$-regret.
\begin{definition}
\label{def:u_t M}
A DFC parameterized by $M=[M_s^{[k]}]_{k\in[h],s\in\U}$ is given by 
\begin{equation}
\label{eqn:u_t M}
u_t^{M} = \sum_{s\in\U}\sum_{k=1}^{h}I_{\V,s}M_{s}^{[k]}\eta_{t-k,s},\ \forall t\in\Z_{\ge0},
\end{equation}
where $h\in\Z_{\ge1}$, and $\eta_{k,s}$ is given by Eq.~\eqref{eqn:eta_k,s} and satisfies $\eta_{k,s}=0$ for all $k<0$ and all $s\in\U$. Moreover, let $x_t^M$ be the  state of system~\eqref{eqn:overall system} when the input sequence $u_0^M,\dots,u_t^M$ is applied.
\end{definition}
We assume that the following assumption holds for now and similar assumptions can be found in \cite{li2021safe,lale2020logarithmic}. Later in Section~\ref{sec:extensions}, we show that Assumption~\ref{ass:stable A} can be relaxed.
\begin{assumption}
\label{ass:stable A}
The system matrix $A\in\R^{n\times n}$ is stable, and $\norm{A^k}\le\kappa_0\gamma_0^k$ for all $k\in\Z_{\ge0}$, where $\kappa_0\ge1$ and $\rho(A)<\gamma_0<1$.
\end{assumption}
Now, recall from Lemma~\ref{lemma:opt solution} that for any $s\in\U_1$, the matrix $A_{ss}+B_{ss}K_s$ is stable, where $K_s$ is given by Eq.~\eqref{eqn:set of DARES K}. We then have from the Gelfand formula (e.g., \cite{horn2012matrix}) that for any $s\in\U_1$, there exist $\kappa_s\in\R_{\ge1}$ and $\gamma_s\in\R$ with $\rho(A_{ss}+B_{ss}K_s)<\gamma_s<1$ such that $\norm{(A_{ss}+B_{ss}K_s)^k}\le\kappa_s\gamma_s^k$ for all $k\in\Z_{\ge0}$. To proceed, we denote
\begin{equation}
\label{eqn:kappa and gamma}
\begin{split}
&\Gamma=\max\big\{\norm{A},\norm{B},\max_{s\in\U}\norm{P_s},\max_{s\in\U}\norm{K_s},1\big\},\\
&\gamma=\max\{\max_{s\in\CR}\gamma_s,\gamma_0\},\ \kappa=\{\max_{s\in\CR}\kappa_s,\kappa_0\},
\end{split}
\end{equation}
where $\CR\subseteq\U$ denotes the set of root nodes in $\U$, and $P_s$ is given by Eq.~\eqref{eqn:set of DARES P} for all $s\in\U$. Similarly to the learning algorithms for LQR proposed in \cite{cassel2020logarithmic,li2019distributed,ye2021sample}, we assume (upper bounds on) the parameters in \eqref{eqn:kappa and gamma} are known for our algorithm design. Moreover, we denote
\begin{equation}
\label{eqn:depth of T_i}
D_{\max}=\max_{\substack{i,j\in\V\\ j\rightsquigarrow i}} D_{ij},
\end{equation}
where we write $j\rightsquigarrow i$ if and only if there is a directed path from node $j$ to node $i$ in $\G(\V,\A)$, and recall that $D_{ij}$ is the sum of delays along the directed path from $j$ to $i$ with the smallest accumulative delay. Note that $D_{\max}\ll p$ when the delay in the information flow among the nodes in $\V=[p]$ is small. We then have the following result, which shows that the optimal DFC given by \eqref{eqn:DFC of u_star} belongs to a class of DFCs with bounded norm.

\begin{lemma}
\label{lemma:norm bound on M_star}
For $M_{\star,s}$ given by \eqref{eqn:M_star self loop}, it holds that 
\begin{align*}
&\norm{M_{\star,s}^{[k]}}\le\kappa\gamma^{k-1}p\Gamma^{2D_{\max}+1},\\
&\norm{M_{\star,s}^{[k]}}_F\le\sqrt{n}\kappa\gamma^{k-1}p\Gamma^{2D_{\max+1}},
\end{align*}
for all $k\in[t]$ and all $s\in\U$, where $h\in\Z_{\ge1}$.
\end{lemma}
\begin{proof}
First, we have from Eq.~\eqref{eqn:M_star self loop} that for any $s\in\U_1$ and any $k\in[t]$,
\begin{align}\nonumber
\nm[\big]{M_{\star,s}^{[k]}}&\le\norm{K_s}\nm[\big]{(A_{ss}+B_{ss}K_s)^{k-1}}\nm[\big]{\begin{bmatrix}H_{v,s}I_{v,\{j_v\}}\end{bmatrix}_{v\in\CL_s}}\\\nonumber
&\le\Gamma\kappa_s\gamma_s^{k-1}\sum_{v\in\CL_s}\nm[\big]{H_{v,s}I_{v,\{j_v\}}}\\\nonumber
&\le\Gamma\kappa_s\gamma_s^{k-1}|\CL_s|\Gamma^{2D_{\max}}\le\kappa\gamma^{k-1} p\Gamma^{2D_{\max}+1},
\end{align}
where $\CL_s=\{v\in\CL:v\rightsquigarrow s\}$ is the set of leaf nodes in $\CP(\U,\CH)$ that can reach $s$, and the third inequality follows from the facts that $|\CL_s|\le|\V|=p$ and $\norm{H_{v,s}}\le\Gamma^{2D_{\max}}$ for all $v,s\in\U$. Similarly, we have from \eqref{eqn:M_star self loop} that for any $s\in\U_2$ and any $k\in[t]$,
\begin{equation*}
\nm[\big]{M_{\star,s}^{[k]}}\le\kappa\gamma^{k-1}p\Gamma^{2D_{\max}+1}.
\end{equation*}
Finally, for any $s\in\U$ and any $k\in[t]$, we have from the above arguments that 
\begin{align}\nonumber
\nm[\big]{M_{\star,s}^{[k]}}_F&\le\sqrt{\min\{n_s,n_{\CL_s}\}}\nm[\big]{M_{\star,s}^{[k]}}\\\nonumber
&\le\sqrt{n}\kappa\gamma^{k-1}p\Gamma^{2D_{\max}+1},
\end{align}
where $M_{\star,s}^{[k]}\in\R^{n_s\times n_{\CL_s}}$, and we use the fact that $n_s\le n_{\CL_s}$ with $n_s\le n$.
\end{proof}

{\bf Candidate DFCs.} Based on Definition~\ref{def:u_t M} and Lemma~\ref{lemma:norm bound on M_star}, we define a class of DFCs parameterized by $M=[M_s^{[k]}]_{k\in[h],s\in\U}$:
\begin{align}
\label{eqn:the class of DFCs}
\D=\Big\{M=[M_s^{[k]}]_{k\in[h],s\in\U}:\nm[\big]{M_s^{[k]}}_F\le2\sqrt{n}\kappa p\Gamma^{2D_{\max}+1}\Big\}.
\end{align}
Our decentralized control policy design is then based on the DFCs from $\D$. Consider any $M\in\D$ and any $i\in\V$, the corresponding control input $u_{t,i}^M$ (i.e, the $i$th element of $u_t^M$ given by Definition~\ref{def:u_t M}) needs to be determined based on the state information in $\I_{t,i}$ defined in Eq.~\eqref{eqn:info set}. In particular, the past disturbances that are required to compute $u_{t,i}^M$ may be determined exactly via Eq.~\eqref{eqn:overall system}, using the state information in $\I_{t,i}$ and the system matrices (e.g., \cite{lamperski2015optimal,ye2021sample}). Since we consider the scenario with unknown system matrices $A$ and $B$, $u_t^M$ given by Definition~\ref{def:u_t M} cannot be directly implemented, which together with the information constraints create the major challenge when designing and analyzing our decentralized online control algorithm in the next section. 

\section{Algorithm Design}\label{sec:decentralized online control alg}
The decentralized online control algorithm that we propose contains two phases. The first phase is a pure exploration phase dedicated to identifying the system matrices $A$ and $B$. The second phase leverages an Online Convex Optimization (OCO) algorithm to find the decentralized control policy. 

\subsection{Phase I: System Identification}\label{sec:system id phase}
\begin{algorithm}
	\textbf{Input:} parameter $\lambda\in\R_{>0}$, time horizon $N$, graph $\G(\V,\A)$ with delay matrix $D$
	\caption{Least Squares Estimation of $A$ and $B$}
	\label{algorithm:least squares}
	\begin{algorithmic}[1]
		\For{$t=0,\dots,N-1$ and {\bf for} each $i\in\V$ in parallel}
		\State Play $u_{t,i}^{\tt alg}\overset{\text{i.i.d.}}{\sim}\CN(0,\sigma_u^2I_{m_i})$
		\EndFor
		\State Obtain $\tilde{\Phi}_N$ using~\eqref{eqn:least squares approach}
		\State Extract $\tilde{A}$ and $\tilde{B}$ from $\tilde{\Phi}_N$
		\State $\hat{A}\gets\tilde{A}$, $\hat{B}\gets\tilde{B}$
		\State Set $\hat{A}_{ij}=0$ and $\hat{B}_{ij}=0$ if $D_{ij}=\infty$
		\State {\bf Return} $\hat{A}$ and $\hat{B}$
	\end{algorithmic} 
\end{algorithm}

During the first phase, the algorithm uses a least squares method to obtain estimates of $A$ and $B$, denoted as $\hat{A}$ and $\hat{B}$, respectively, using a single system trajectory consisting of the control input sequence $\{u_0^{\tt alg},\dots,u_{N-1}^{\tt alg}\}$ and the corresponding system state sequence $\{x_0^{\tt alg},\dots,x_N^{\tt alg}\}$, where $N\in\Z_{\ge1}$. Here, the inputs $u_0^{\tt alg},\dots,u_{N-1}^{\tt alg}$ are drawn independently from a Gaussian distribution $\CN(0,\sigma^2_uI_m)$, where $\sigma_u\in\R_{>0}$. In other words, we have $u_t^{\tt alg}\overset{\text{i.i.d.}}{\sim}\CN(0,\sigma_u^2I_m)$ for all $t\in\{0,\dots,N-1\}$. Moreover, we assume that the input $u_t$ and the disturbance $w_t$ are independent for all $t\in\{0,\dots,N-1\}$. Specifically, denote
\begin{equation}
\label{eqn:Theta_i and z_Ni}
\Phi=\begin{bmatrix}A & B\end{bmatrix},\ z_t=\begin{bmatrix}x_t^{{\tt alg}\top} & u_t^{{\tt alg}\top}\end{bmatrix}^{\top},
\end{equation}
where $\Phi\in\R^{n\times(n+m)}$ and $z_t\in\R^{n+m}$. Given the sequences $\{z_0,\dots,z_{N-1}\}$ and $\{x_1^{\tt alg},\dots,x_N^{\tt alg}\}$,  the algorithm uses a regularized least squares method to obtain an estimate of $\Phi$, denoted as $\tilde{\Phi}_N$, i.e.,
\begin{equation}
\label{eqn:least squares approach}
\tilde{\Phi}_N=\argmin_{Y\in\R^{n\times (n+m)}}\Big\{\lambda\norm{Y}_F^2+\sum_{t=0}^{N-1}\norm{x_{t+1}^{\tt alg} - Y z_t}^2\Big\},
\end{equation}
where $\lambda\in\R_{>0}$ is the regularization parameter. The least squares method is summarized in Algorithm~\ref{algorithm:least squares}. Note that the last step in Algorithm~\ref{algorithm:least squares} sets certain elements in $\tilde{A}$ and $\tilde{B}$ (extracted from $\tilde{\Phi}_N$) to be zero, according to the directed graph $\G(\V,\A)$ and the corresponding delay matrix $D$ as we described in Section~\ref{sec:dist LQR known matrices}. Algorithm~\ref{algorithm:least squares} returns the final estimates of $A$ and $B$ as $\hat{A}$ and $\hat{B}$, respectively. Denote
\begin{align}
\tilde{\Delta}_N=\Phi-\tilde{\Phi}_N,\ \Delta_N=\Phi-\hat{\Phi}_N,\label{eqn:def of Delta}
\end{align}
where $\hat{\Phi}_N=\begin{bmatrix}\hat{A} & \hat{B}\end{bmatrix}$, with $\hat{A}$ and $\hat{B}$ returned by Algorithm~\ref{algorithm:least squares}.   To obtain $\hat{A}$ and $\hat{B}$, we set certain entries in $\tilde{A}$ and $\tilde{B}$ obtained from \eqref{eqn:least squares approach} to zero according to the delay matrix $D$ that is assumed to be known. In fact, there exist system identification methods (for sparse system identification) that return $\hat{A}$ and $\hat{B}$ with the same sparsity pattern as $A$ and $B$, under extra assumptions on $A$ and $B$ (e.g., \cite{fattahi2019learning}). However, the extra assumptions on $A$ and $B$ can be restrictive and hard to check in practice.

{\bf Gaussian Inputs.} Algorithm~\ref{algorithm:least squares} needs to use inputs drawn from a Gaussian distribution, which is typical in algorithms for learning (centralized) LQR (e.g., \cite{cohen2019learning,cassel2020logarithmic,simchowitz2020improper}). As we will see later, using the Gaussian inputs is crucial to provide upper bounds on the estimation error terms $\tilde{\Delta}_N$ and $\Delta_N$. We leave relaxing the Gaussian input assumption to future work.

{\bf Requirement on Global Knowledge.} Algorithm~\ref{algorithm:least squares} considers the scenario where each $i\in\V$ plays the control $u_{t,i}^{\tt alg}\overset{\text{i.i.d.}}{\sim}\CN(0,\sigma_u^2I_{m_i})$ in parallel for $t=0,\dots,N-1$, which implies $u_t^{\tt alg}\overset{\text{i.i.d.}}{\sim}\CN(0,\sigma_u^2I_{m})$, and sends the local $x_{t,i}^{\tt alg}$ and $u_{t,i}^{\tt alg}$ to a central agent for $t=0,\dots,N-1$. The central agent obtains $
\tilde{\Phi}_N$ using \eqref{eqn:least squares approach} based on a single system trajectory $\{(x_0^{\tt alg},u_0^{\tt alg}),\dots,(x_{N-1}^{\tt alg},u_{N-1}^{\tt alg})\}$, and then sends the estimates $\hat{A},\hat{B}$ back to all $i\in\V$. Nonetheless, 
Algorithm~\ref{algorithm:least squares} can be implemented without violating the information constraints given by Eq.~\eqref{eqn:info set}, since $u_t^{\tt alg}\overset{\text{i.i.d.}}{\sim}\CN(0,\sigma_u^2I_m)$ is not a function of the states in the information set defined in Eq.~\eqref{eqn:info set} for any $t\in\{0,\dots,N-1\}$. After receiving the global $\hat{A}$ and $\hat{B}$, each $i\in\V$ enters the second phase of the algorithm and computes $u_{t,i}^{\tt alg}$ in a fully decentralized manner (without access to the centralized agent). Even when the system model is known, existing algorithms for decentralized controller design typically require the global knowledge of the system model (e.g., \cite{lamperski2015optimal,lamperski2012dynamic,rotkowitz2005characterization,rotkowitz2011nearest}. Relaxing this requirement is an interesting avenue for future work.

\subsection{Phase II: Decentralized Online Control}\label{sec:decentralized online control phase} The second phase leverages a general OCO algorithm which we introduce first.
\subsubsection{OCO with Memory and Delayed Feedback}\label{sec:OCO with memory and delay}
{\color{black} We solve a general OCO problem with memory and delayed feedback, which is of independent interest.} 
At each time step {\color{black}$t\ge\tau$}, a decision maker first chooses $x_t\in\W\subseteq\R^d$ and incurs a cost $F_t(x_{t-h},\dots,x_t)$, and the function $F_{t-\tau}(\cdot)$ is revealed to the decision maker, with $x_k\in\W$ arbitrarily if $k<0$. Different from classic OCO frameworks, but unifying the ones in \cite{anava2015online} and \cite{langford2009slow},  the incurred cost at time step $t$ has a memory that depends on the choices $x_{t-h},\dots,x_t$, and the function $F_{t-\tau}(\cdot)$ is revealed at time step $t$. {\color{black}The learner wishes to minimize its regret $\sum_{t=k}^{T-1}(F_t(x_{t-h},\dots,x_t)-f_t(x^{\star}))$ for some $x^{\star}\in\W$, where 
$f_t(x)=F_t(x,\dots,x)$ for $x\in\W$. For this, we propose Algorithm~\ref{algorithm:OCO} based on an online projected gradient descent scheme, where $x_{t+1}$ is updated from $x_t$ using the delayed gradient $g_{t-\tau}$ from $\nabla f_{t-\tau}(x_{t-\tau})$ with an additive error $\varepsilon_{t-\tau}\in\R$. This error is useful when the gradient may not be exactly evaluated as in our decentralized LQR design.}

\begin{algorithm}
\textbf{Input:} time horizon $T$, step size $\eta_t\ \forall t\in\{\tau,\dots,T+\tau\}$, delay $\tau\in\Z_{\ge0}$, feasible set $\W$
\caption{OCO with memory and delayed feedback}
\label{algorithm:OCO}
\begin{algorithmic}[1]
	\State Initialize $x_0=\cdots=x_{\tau-1}\in\W$ arbitrarily
	\For{$t=\tau,\dots,T-1+\tau$}
	\State Obtain $g_{t-\tau}=\nabla f_{t-\tau}(x_{t-\tau})+\varepsilon_{t-\tau}$
	\State Update $x_{t+1}=\Pi_{\W}(x_t-\eta_tg_{t-\tau})$
	\EndFor
\end{algorithmic} 
\end{algorithm}

{\color{black}We analyze the regret of Algorithm~\ref{algorithm:OCO} under the following assumptions. 
\begin{assumption}
	\label{ass:cost F_t}
	For any $t\in\{0,\dots,T-1\}$, the cost function $F_t:\W^{h+1}\to\R$ is $L_c$-coordinatewise-Lipschitz.\footnote{$F_t(\cdot)$ is $L_c$-coordinatewise-Lipschitz if $|F_t(x_0,\dots,x_j,\dots,x_{h})-F_t(x_0,\dots,\tilde{x}_j,\dots,x_h)|\le L_c\norm{x_j-\tilde{x}_j}$, for all $j\in\{0,\dots,h\}$ and all $x_j,\tilde{x}_j\in\W$.}
\end{assumption}

\begin{assumption}
	\label{ass:unary f_t}
	Let $k\in\Z_{\ge0}$ with $\tau< k< T$ and $\tau\in\Z_{\ge0}$, and let $\{\F_t\}_{t\ge0}$ be a filtration. The induced unary function $f_t:\W\to\R$ is $L_f$-Lipschitz with $\max_{x\in\W}\norm{\nabla^2f_t(x)}_2\le\beta$ for all $t\in\{0,\dots,T-1\}$, and  $f_{t;k}(x)\triangleq\E[f_t(x)|\F_{t-k}]$ is $\alpha$-strongly convex $\forall t\in\{k,\dots,T-1\}$.\footnote{$f_t(\cdot)$ is $L_f$-Lipschitz if $|f_t(x_1)-f_t(x_2)|\le L_f\norm{x_1-x_2}$ for all $x_1,x_2\in\W$. $f_t(\cdot)$ is $\alpha$-strongly convex if and only if $\nabla^2f_t(x)\succeq\alpha I_d$ for all $x\in\W$ (e.g., \cite{boyd2004convex}).}
\end{assumption}

\begin{assumption}
	\label{ass:gradient bound}
	The set $\W\subseteq\R^d$ is assumed to be convex and satisfies that ${\rm Diam}(\W)\triangleq\sup_{x,y\in\W}\norm{x-y}\le G$. Suppose $\norm{g_t}\le L_g$ for all $t\in\{0,\dots,T-1\}$.   
	\end{assumption}}
	The following results are proved in Appendix~\ref{sec:OCO proofs}. 
	\begin{lemma}
\label{lemma:upper bound for conditional cost}
Let Assumptions~\ref{ass:cost F_t}-\ref{ass:gradient bound} hold. Set $\eta_t=\frac{3}{\alpha t}$ for $t\in\{\tau,\dots,T+\tau\}$ in Algorithm~\ref{algorithm:OCO} and define $\varepsilon_{t-\tau}^s\triangleq\nabla f_{t-\tau}(x_{t-\tau})-\nabla f_{t-\tau;k}(x_{t-\tau})$. Then, 
\begin{align}\nonumber
	\sum_{t=k}^{T-1}\big(f_{t;k}(x_t)-f_{t;k}(x_{\star})\big)&\le-\frac{\alpha}{6}\sum_{t=0}^{T-1}\norm{x_t-x_{\star}}^2+\frac{\alpha G^2(3k+5\tau+3)}{6}+\frac{3L_g^2(1+\tau)}{\alpha}\log T\\
	&\qquad\qquad\qquad+\sum_{t=k}^{T-1}\big(\frac{3}{2\alpha}\norm{\varepsilon_{t}}^2-\varepsilon_{t}^{s\top}(x_{t}-x_{\star})\big),\label{eqn:conditional f upper bound}\ \forall x_{\star}\in\W.
\end{align}
\end{lemma}

\begin{lemma}
\label{lemma:true f upper bound}
Let Assumptions~\ref{ass:cost F_t}-\ref{ass:gradient bound} hold, and set $\eta_t=\frac{3}{\alpha t}$ for $t\in\{\tau,\dots,T+\tau\}$ in Algorithm~\ref{algorithm:OCO}. With $X_t(x_{\star})\triangleq (f_{t:k}-f_t)(x_{t-k})-(f_{t;k}-f_t)(x_{\star})+\nabla(f_t-f_{t;k})(x_{t-k})^{\top}(x_{t-k}-x_{\star}),$
\begin{align}\nonumber
	\sum_{t=k}^{T-1}\big(f_{t}(x_t)-f_{t}(x_{\star})\big)&\le-\frac{\alpha}{6}\sum_{t=0}^{T-1}\norm{x_t-x_{\star}}^2+\frac{\alpha G^2(3k+5\tau+3)}{6}+\sum_{t=k}^{T-1}\frac{3}{2\alpha}\norm{\varepsilon_{t}}^2\\
	&\qquad\qquad+\frac{3L_g}{\alpha}\big(L_g(1+\tau)+(4\beta G+8L_f)k\big)\log T-\sum_{t=k}^{T-1}X_t(x_{\star}).\label{eqn:true f upper bound}
\end{align}
\end{lemma}

\begin{proposition}
\label{prop:true f with memory upper bound}
Let Assumptions~\ref{ass:cost F_t}-\ref{ass:gradient bound} hold. Let $k> h$ and $\eta_t=\frac{3}{\alpha t}$ for $t\in\{\tau,\dots,T+\tau\}$ in Algorithm~\ref{algorithm:OCO}. Then, for any $\delta>0$,  with probability at least $1-\delta$:
\begin{align}\nonumber
	&\sum_{t=k}^{T-1}\big(F_t(x_{t-h},\dots,x_{t})-f_t(x_{\star})\big)\\\nonumber
	&\le-\frac{\alpha}{12}\sum_{t=0}^{T-1}\norm{x_t-x_{\star}}^2+\CO(1)\bigg(\sum_{t=k}^{T-1}\frac{\norm{\varepsilon_{t}}^2}{\alpha}+\frac{kdL_f^2}{\alpha}\log\Big(\frac{T\big(1+\log_+(\alpha G^2)\big)}{\delta}\Big)\\
	&\qquad\qquad+\alpha G^2(k+\tau)+\frac{L_g}{\alpha}\big(L_g\tau+(\beta G+L_f)k+L_ch^2\big)\log T\bigg),\ \forall x_{\star}\in\W,\label{eqn:true f upper bound result}
\end{align}
where $\CO(1)$ denotes a universal constant, and $\log_+(x)\triangleq\log(\max\{1,x\})$ $\forall x\in\R_{>0}$.
\end{proposition}

\subsubsection{Decentralized Control Policy Design}\label{sec:description of Algorithm 3.2}
After obtaining $\hat{A},\hat{B}$ from the system identification phase, Algorithm~\ref{algorithm:control design} computes $u_{t,i}^{\tt alg}$ for each $i\in\V$ in a fully decentralized manner (lines~2-16).  Algorithm~\ref{algorithm:control design} chooses a DFC $u_{t}^{\tt alg}$ (from the convex set $\D$ of DFCs given by Eq.~\eqref{eqn:the class of DFCs}) parameterized by $M_t=[M_{t,s}^{[k]}]_{k\in[h],s\in\U}$ for all $t\in\{N,\dots,T-1\}$, based on estimates of the true disturbance $w_t$ in Eq.~\eqref{eqn:overall system}. Specifically, for any $j\in\V$, let $\hat{w}_{t,j}$ be an estimate of the disturbance $w_{t,j}$ in Eq.~\eqref{eqn:dynamics for x_i(t)} obtained as
\begin{equation}
\label{eqn:est w_i t}
\hat{w}_{t,j}=\begin{cases}
0,\ \text{if}\ t\le N-1,\\
x_{t+1,j}^{\tt alg} - \hat{A}_jx_{t,\CN_j}^{\tt alg} - \hat{B}_ju^{\tt alg}_{t,\CN_j},\ \text{if}\ t\ge N,
\end{cases}
\end{equation}
where we replace $A_j$ and $B_j$ in Eq.~\eqref{eqn:dynamics for x_i(t)} with the estimates $\hat{A}_j$ and $\hat{B}_j$ obtained from Algorithm~\ref{algorithm:least squares}, respectively, and $x^{\tt alg}_{t,\CN_j}=\begin{bmatrix}x_{t,{j_1}}^{{\tt alg}\top}\end{bmatrix}_{j_1\in\CN_j}^{\top}$ and $u^{\tt alg}_{t,\CN_j}=\begin{bmatrix}u_{t,{j_1}}^{{\tt alg}\top}\end{bmatrix}_{j_1\in\CN_j}^{\top}$ with $\CN_j$ given in Eq.~\eqref{eqn:overall system}. Denote
\begin{equation}
\label{eqn:eta_k,s hat}
\hat{\eta}_{k,s}=\begin{bmatrix}\hat{w}_{k-l_{vs},j_v}^{\top}\end{bmatrix}^{\top}_{v\in\CL_s},
\end{equation}
with $w_{j_v}\to v$ and $\CL_s$ given by Eq.~\eqref{eqn:set of leaf nodes of s}.

{\bf The OCO Subroutine in Algorithm~\ref{algorithm:control design}.} In line~14 of Algorithm~\ref{algorithm:control design}, the parameter $M_t=[M_{t,s}^{[k]}]_{k\in[h],s\in\U}\in\D$ of the DFC is updated using the OCO subroutine from Algorithm~\ref{algorithm:OCO}, where we denote $M_{t,s}=[M_{t,s}^{[k]}]_{k\in[h]}$. In line~14 of Algorithm~\ref{algorithm:control design}, $\Pi_{\D}(\cdot)$ denotes the projection onto the set $\D$, which yields efficient implementations (e.g., \cite{quattoni2009efficient,simchowitz2020improper}). Formally, we introduce the following definitions.
\begin{definition}{\bf (Counterfactual Dynamics and Cost)}
\label{def:counterfactual cost}
Let $M_k\in\D$ for all $k\ge0$, where $\D$ is given by Eq.~\eqref{eqn:the class of DFCs}. First, for any $k\ge0$, define
\begin{align}
u_k(M_k|\hat{w}_{0:k-1})=\sum_{s\in\U}\sum_{k^{\prime}=1}^{h}I_{\V,s}M_{k,s}^{[k^{\prime}]}\hat{\eta}_{k-k^{\prime},s},\label{eqn:alg control input}
\end{align}
where $\hat{w}_{0:t-1}$ denotes the sequence $\hat{w}_0,\dots,\hat{w}_{t-1}$. Let $u_{k,i}(M_k|\hat{w}_{0:k-1})$ denote the $i$th element of $u_k(M_k|\hat{w}_{0:k-1})$. Next, for any $t\ge h$, define
\begin{align}
&x_t(M_{t-h:t-1}|\hat{\Phi},\hat{w}_{0:t-1})=\sum_{k=t-h}^{t-1}\hat{A}^{t-(k+1)}\big(\hat{w}_k+\hat{B}u_k(M_k|\hat{w}_{0:k-1})\big),\label{eqn:counter state}\\
&F_t(M_{t-h:t}|\hat{\Phi},\hat{w}_{0:t-1})=c\Big(x_t(M_{t-h:t-1}|\hat{\Phi},\hat{w}_{0:t-1}),u_t(M_t|\hat{w}_{0:t-1})\Big),\label{eqn:counter cost}
\end{align}
where $\hat{\Phi}=\begin{bmatrix}\hat{A} & \hat{B}\end{bmatrix}$. Finally, for any $t\ge h$, define
\begin{align}
&x_t(M|\Phi,\hat{w}_{0:t-1}) = x_t(M_{t-h:t-1}|\hat{\Phi},\hat{w}_{0:t-1}),\\
&f_t(M|\hat{\Phi},\hat{w}_{0:t-1})=F_t(M_{t-h:t}|\hat{\Phi},\hat{w}_{0:t-1}),
\end{align}
if $M_{t-h}=\cdots=M_t=M$ with $M\in\D$. Similarly, define $u_k(M_k|w_{0:k-1})$, $x_t(M_{t-h:t}|\Phi,w_{0:t-1})$, $F_t(M_{t-h:t}|\Phi,w_{0:t-1})$, $x_t(M|\Phi,w_{0:t-1})$, and $f_t(M|\Phi,w_{0:t-1})$.
\end{definition}
{\color{black} Since the cost $F_t:\D^{h+1}\to\R$ is defined using the estimated system $\hat{A},\hat{B}$ rather than the real system $A,B$, we refer to $F_t(\cdot)$ as the {\it counterfactual cost}. Next, we define a cost corresponding to the true system dynamics $x_{t+1}^{\tt alg}=Ax_t^{\tt alg}+Bu_t^{\tt alg}+w_t$.}

\begin{algorithm}
\textbf{Input:} parameters $\lambda,N,R_x,R_u,\vartheta$, cost matrices $Q$ and $R$, directed graph $\G(\V,\A)$ with delay matrix $D$, time horizon length $T$, step sizes $\eta_t$ for $t\in\{N,\dots,T+D_{\max}-1\}$
\caption{Decentralized online control algorithm}\label{algorithm:control design}
\begin{algorithmic}[1]
\State Use Algorithm~\ref{algorithm:least squares} to obtain $\hat{A}$ and $\hat{B}$
\State For any $i\in\V$, initialize $\K_{i,1}=\bar{\K}_{i,1}$ and $\K_{i,2}=\bar{\K}_{i,2}$
\For{$t=N,\dots,N+D_{\max}-1$}
    \State Set $M_t=0$ and play $u^{\tt alg}_{t,i}=u_{t,i}(M_t|\hat{w}_{0:t-1})=0$
\EndFor
\For{$t=N+D_{\max},\dots,T+D_{\max}-1$ and {\bf for} each $i\in\V$ in parallel}
    \For{$s\in\CL(\T_i)$}
        \State Find $w_j$ s.t. $j\in\V$ and $s_{0,j}=s$
        \State Obtain $\hat{w}_{t-D_{ij}-1,j}$ from Eq.~\eqref{eqn:est w_i t}
        \State $\K_{i,1}\gets \K_{i,1}\cup\{\hat{w}_{t-D_{ij}-1,j}\}$
    \EndFor
    \If{$\nm[\big]{x_{t,i}^{\tt alg}}>R_x$ or $\nm[\big]{u_{t,i}(M_t|\hat{w}_{0:t-1})}>R_u$}
        \State Play $u_{t,i}^{\tt alg}=0$ until $t=T+D_{\max}-1$
    \EndIf
    \State Play $u^{\tt alg}_{t,i}=u_{t,i}(M_t|\hat{w}_{0:t-1})$
    \For{$s\in\T_i$}
        \State $M_{t+1,s}\gets\Pi_{\D}\big(M_{t,s}-\eta_t\frac{\partial f_{t-D_{\max}}(M_{t-D_{\max}}|\hat{\Phi},\hat{w}_{0:t-D_{\max}-1})}{\partial M_{t-D_{\max},s}}\big)$
        \State $\K_{i,2}\gets\K_{i,2}\cup\{M_{t+1,s}\}\setminus\{M_{t-D_{\max}-1,s}\}$
    \EndFor
    \State $\K_{i,1}\gets\K_i\setminus\{\hat{w}_{t-2D_{\max}-2h,j}:s\in\CL(\T_i),w_j\to s\}$
\EndFor
\end{algorithmic} 
\end{algorithm}

\begin{definition}
\label{def:true prediction cost}{\bf (Prediction Cost)}	Let $M_k\in\D$ for all $k\ge0$, where $\D$ is given by Eq.~\eqref{eqn:the class of DFCs}. For any $t\ge h$, define 
\begin{align}\nonumber
		&x^{\tt pred}_t(M_{t-h:t-1}) = \sum_{k=t-h}^{t-1}A^{t-(k+1)}\big(w_k+Bu_k(M_k|\hat{w}_{0:k-1})\big),\\\nonumber
		&F_t^{\tt pred}(M_{t-h:t})=c\Big(x_t^{\tt pred}(M_{t-h:t-1}),u_t(M_t|\hat{w}_{0:t-1})\Big),
\end{align}
where $u_k(M_k|\hat{w}_{0:k-1})$ is defined in Definition~\ref{def:counterfactual cost}. Also define the terms $x^{\tt pred}_t(M)$ and $f_t^{\tt pred}(M)$ analogous to $x_t(M|\Phi,\hat{w}_{0:t-1})$ and $f_t(M|\hat{\Phi},\hat{w}_{0:t-1})$ in Definition~\ref{def:counterfactual cost}.
\end{definition}
$x_t^{\tt pred}(M_{t-h:t-1})$ is a prediction of the true state $x_t^{\tt alg}$ based on the past control inputs and disturbances from a length-$h$ window. Thus, we refer to the corresponding cost $F_t^{\tt pred}(\cdot)$ as the {\it prediction cost.} 
	
Based on the above arguments, line~14 in Algorithm~\ref{algorithm:control design} can be viewed as applying the OCO subroutine given by Algorithm~\ref{algorithm:OCO} to the function sequence $F_t^{\tt pred}:\D^{h+1}\to\R$ and $f_t^{\tt pred}:\D\to\R$ for $t\in\{N+D_{\max},\dots,T+D_{\max}-1\}$, where $F_t^{\tt pred}(\cdot)$ and $f_t^{\tt pred}(\cdot)$ are defined in Definition~\ref{def:true prediction cost}, $D_{\max}$ is defined in Eq.~\eqref{eqn:kappa and gamma}, $N$ is an input parameter to Algorithm~\ref{algorithm:control design}, and we set the delay $\tau\in\Z_{\ge0}$ to be $\tau=D_{\max}$ in the OCO subroutine. Furthermore, we let the delayed gradient $g_{t-\tau}$ in the OCO subroutine to be $g_{t-\tau}=\nabla f_{t-\tau}(M_{t-\tau}|\hat{\Phi},\hat{w}_{0:t-\tau-1})$ and thus we have
\begin{align}
		\varepsilon_{t-\tau}&=g_{t-\tau}-\nabla f_{t-\tau}(x_{t-\tau})=\nabla f_{t-\tau}(M_{t-\tau}|\hat{\Phi},\hat{w}_{0:t-\tau-1})-\nabla f_t^{\tt pred}(M_{t-\tau}),\label{eqn:gradient error}
\end{align}
where $f_{t-\tau}(\cdot|\hat{\Phi},\hat{w}_{0:t-\tau-1})$ is defined in Definition~\ref{def:counterfactual cost}. We will later show in Section~\ref{sec:R_2} that our analysis developed in Section~\ref{sec:OCO with memory and delay} for the  OCO with memory and delayed feedback can be specialized to our  setting.

{\bf The Decentralized Structure of Algorithm~\ref{algorithm:control design}.} To see that Algorithm~\ref{algorithm:control design} is decentralized, for any $i\in\V$, let $\T_i$ denote the set of disconnected directed trees in $\CP(\U,\CH)$ such that the root node of any tree in $\T_i$ contains $i$. Reloading the notation, also let $\T_i$ denote the set of nodes of all the trees in $\T_i$. Moreover, denote 
\begin{equation}
\label{eqn:leaf nodes in T_i}
\CL(\T_i)=\T_i\cap\CL,
\end{equation}
where $\CL$ is defined in Eq.~\eqref{eqn:set of leaf nodes}, i.e., $\CL(\T_i)$ is the set of leaf nodes of all the trees in $\T_i$. For example, in Fig.~\ref{fig:directed graph}, $\T_1=\{\{1,2,3\},\{1,2\},\{3\},\{1\}\}$, and $\CL(\T_1)=\{\{1\},\{1,2\},\{3\}\}$. For any $i\in\V$, Algorithm~\ref{algorithm:control design} maintains sets $\K_{i,1}$ and $\K_{i,2}$ in its current memory, which are initialized as $\K_{i,1}=\bar{\K}_{i,1}$ and $\K_{i,2}=\bar{\K}_{i,2}$ with 
\small
\begin{equation}
	\label{eqn:initial K_i}
	\begin{split}
		\bar{\K}_{i,1} &= \big\{\hat{w}_{k,j}:k\in\{N^{\prime}-2D_{\max}-2h,\dots,N^{\prime}-D_{ij}-2\},s\in\CL(\T_i),j\in\V,s_{0,j}=s\big\},\\
		\bar{\K}_{i,2} &= \big\{M_{k,s}=[M_{k,s}^{[k^{\prime}]}]_{k^{\prime}\in[h]}:k\in\{N^{\prime}-D_{\max}-1,\dots,N^{\prime}\},s\in\T_i\big\},
	\end{split}
\end{equation}
\normalsize
where $N^{\prime}=N+D_{\max}$, and we let $\hat{w}_{k,j}=0$ for all $\hat{w}_{k,i}\in\bar{\K}_{i,1},$ $M_{k,s}=0$ for all $M_{k,s}\in\bar{\K}_{i,2}$ and $M_{t}=0$ for all $t<N$. {\color{black}For any $i\in\V$ and any $t\ge N+D_{\max}$, $\K_{i,1}$ and $\K_{i,2}$ can be obtained based on the information set $\I_{t,i}$ defined in Eq.~\eqref{eqn:info set}, which implies that $u_{t,i}^{\tt alg}$ 
can be computed from $\I_{t,i}$. Finally, note that the number of iterations of the for loop in lines~6-9 (resp., lines~13-15) in Algorithm~\ref{algorithm:control design} is bounded by $|\CL(\T_i)|\le p$ (resp., $|\T_i|\le|\U|\le p^2-p+1$).}

\begin{remark}
\label{remark:ordering of the elements in L(T_i)}
For any $s,r\in\CL(\T_i)$, let $j_1,j_2\in\V$ be such that $s_{0,j_1}=s$ and $s_{0,j_2}=r$. In Algorithm~\ref{algorithm:control design}, we assume without loss of generality that the elements in $\CL(\T_i)$ are already ordered such that if $D_{ij_1}>D_{ij_2}$, then $s$ comes before $r$ in $\CL(\T_i)$. We then let the for loop in lines~7-14 in Algorithm~\ref{algorithm:control design} iterate over the elements in $\CL(\T_i)$ according to the above order. For example, considering node $2$ in the directed graph $\G(\V,\A)$ in Example~\ref{exp:running example}, we see from Figs.~\ref{fig:directed graph}-\ref{fig:info graph} that $\CL(\T_2)=\{\{1\},\{1,2\},\{3\}\}$, where $s_{0,1}=\{1\}$, $s_{0,2}=\{1,2\}$ and $s_{0,3}=\{3\}$. Since $D_{21}=1$, $D_{22}=0$ and $D_{23}=1$, we let the elements in $\CL(\T_2)$ be ordered such that $\CL(\T_2)=\{\{1\},\{3\},\{1,2\}\}$.
\end{remark}

\subsection{Results for Algorithm~\ref{algorithm:control design}}\label{sec:main results}{\color{black}First, we show that the decentralized control policy given by Algorithm~\ref{algorithm:control design} satisfies the required information constraint described in Section~\ref{sec:dist LQR known matrices}.} We will make the following assumption on the cost matrices $Q,R$.
\begin{assumption}
	\label{ass:structure of Q and R}
	Let $\psi\in\Z_{\ge1}$ be the number of strongly connected components in the directed graph $\G(\V,\CE)$, and let $\V=\cup_{l\in[\psi]}\V_l$, where $\V_l\subseteq\V$ is the set of nodes in the $l$th strongly connected component. For any $l_1,l_2\in[\psi]$, any $i\in\V_{l_1}$ and any $j\in\V_{l_2}$, it holds that $(i,j)\notin\A$ and $(j,i)\notin\A$. Moreover, the cost matrices $Q\in\BS^n_+$ and $R\in\BS^m_{++}$ have a block structure according to the directed graph $\G(\V,\A)$ such that $Q_{\V_l{\V}^c_l}=0$ and $R_{\V_l\V^c_l}=0$ for all $l\in[\psi]$, where $\V_l^c=\V\setminus\V_l$.
\end{assumption}
Supposing the directed graph $\G(\V,\A)$ is strongly connected, one can check that Assumption~\ref{ass:structure of Q and R} holds and $Q$ and $R$ need not possess the block structure. 

\begin{proposition}
	\label{prop:decentralized online algorithm feasible}
	Suppose Assumption~\ref{ass:structure of Q and R} holds and any controller $i\in\V$ at any time step $t\in\{N+D_{\max},\dots, T+D_{\max}-1\}$ has access to the states in $\tilde{\I}_{t,i}$ defined as
	\begin{equation}
		\label{eqn:info set used}
		\tilde{\I}_{t,i}=\big\{x_{k,j}^{\tt alg}:j\in\V,D_{ij}<\infty,k\in\{t-D_{\max}-1,\dots,t-D_{ij}\}\big\}\subseteq\I_{t,i},
	\end{equation}
	where $\I_{t,i}$ is defined in Eq.~\eqref{eqn:info set}. 
	For any $i\in\V$ and any $t\in\{N+D_{\max},\dots,T+D_{\max}-1\}$, the sets $\K_{i,1}$ and $\K_{i,2}$ 
	at the beginning of iteration $t$ of the for loop in lines~5-16 of the algorithm satisfy
	\begin{align}
		\K_{i,1} &= \big\{\hat{w}_{k,j}:k\in\{t-2D_{\max}-2h,\dots,t-D_{ij}-2\},s\in\CL(\T_i),j\in\V,s_{0,j}=s\big\},\label{eqn:K_i 1}\\
		\K_{i,2} &= \big\{M_{k,s}=(M_{k,s}^{[k^{\prime}]})_{k^{\prime}\in[h]}:k\in\{t-D_{\max}-1,\dots,t\},s\in\T_i\big\},\label{eqn:K_i 2}
	\end{align}
	and the control input $u_{t,i}^{\tt alg}$ in line~12 can be determined based on $\K_{i,1}$ and $\K_{i,2}$ after line~9 and before line~16 in iteration $t$ of the for loop in lines~5-16 of the algorithm.
\end{proposition}
{\color{black}Proposition~\ref{prop:decentralized online algorithm feasible} proved in Appendix~\ref{sec:proposition 1 proof} shows that for any $i\in\V$, the sets $\K_{i,1}$ and $\K_{i,2}$ can be recursively updated in the memory of Algorithm~\ref{algorithm:control design} based on $\bar{\I}_{t,i}\subseteq\I_{t,i}$, such that $u_{t,i}^{\tt alg}$ can be determined from the current $\K_{i,1}$ and $\K_{i,2}$ for all $t\in\{N+D_{\max},\dots,T+D_{\max}-1\}$. Thus, Proposition~\ref{prop:decentralized online algorithm feasible} precisely shows that $u_t^{\tt alg}$ for any $i\in\V$ designed by Algorithm~\ref{algorithm:control design} satisfies the required decentralized information constraints and can be implemented in a fully decentralized manner based only on the local state information.} Noting that $|\CL(\T_i)|\le p$ and $|\T_i|\le|\U|\le p^2-p+1$ for all $i\in\U$ \cite{lamperski2015optimal}, one can show that the size of the memory $\K_{i,1}$ and $\K_{i,2}$ satisfies that $|\K_{i,1}|\le(2D_{\max}+2h-1)p$ and $|\K_{i,2}|\le(D_{\max}+2)q$ for all $i\in\U$. 
Algorithm~\ref{algorithm:control design} is also significantly different from the offline learning algorithm for decentralized LQR proposed in \cite{ye2021sample} (see our discussions below for more details). The proof of Proposition~\ref{prop:decentralized online algorithm feasible} then deviates from the proof of \cite[Proposition~2]{ye2021sample}. 
{\color{black}Next, we upper bound the regret of Algorithm~\ref{algorithm:control design}.} 
\begin{theorem}
	\label{thm:regret upper bound}
	Suppose Assumptions~\ref{ass:info structure}-\ref{ass:pairs} and \ref{ass:structure of Q and R} hold. There exist input parameters $\lambda,N,R_x,R_u$, and step sizes $\eta_t$ for all $t\in\{N+D_{\max}+T+D_{\max}-1\}$ for Algorithm~\ref{algorithm:control design} such that $\E[{\tt Regret}]=\tilde{\CO}(\sqrt{T})$, for all $T\ge N+3h+D_{\max}$ with $N=\lceil\max\{\tilde{\Theta}(1),\sqrt{T}\}\rceil$ and $h=\lceil\max\{4D_{\max}+4,\frac{4}{1-\gamma}\log T\}\rceil$, where let $\tilde{\CO}(\cdot)$ and $\tilde{\Theta}(\cdot)$ compress {\it polynomial} factors in $\log T$ and other problem parameters.
	\end{theorem}
	
Theorem~\ref{thm:regret upper bound} shows that Algorithm~\ref{algorithm:control design} achieves $\sqrt{T}$-regret under properly chosen input parameters; the proof is in the next section. Below, we provide some remarks.
	
{\bf Optimality of the Regret.} Despite the information constraints, our regret bound for learning decentralized LQR matches with the regret bound for learning centralized  LQR \cite{abbasi2011regret,cohen2019learning,simchowitz2020improper} in terms of $T$, which has been shown to be the best regret that can be achieved(up to polynomial factors in $\log T$ and some other constants) \cite{simchowitz2020naive,cassel2021online,chen2021black}. Particularly, it was shown that the regret of any learning algorithm for centralized LQR (with stable $A$ matrix) is lower bounded by $\Omega(\sqrt{T})$ \cite{cassel2020logarithmic}. 
	
{\bf Necessity of DFC and OCO and Comparison to \cite{ye2021sample}.} An offline learning algorithm based on certainty equivalence approach for decentralized LQR with  partially nested information structure was proposed in \cite{ye2021sample}. 
It requires that the system restart to the initial state after identifying $A$ and $B$ from the system trajectory and cannot be directly applied to our online setting, where restarting the system to the initial state is not allowed. Moreover, \cite{ye2021sample} only provides a $\tilde{\CO}(1/\sqrt{N})$ sample complexity result (with $\tilde{\CO}(\cdot)$ hiding polynomial factors in $\log N$) for the offline algorithm, where $N$ is the length of the system trajectory for identifying the system matrices.  As argued in \cite{cassel2020logarithmic,mania2019certainty}, ignoring the restart issue, an offline algorithm using a certainty equivalence approach for learning LQR with $\tilde{\CO}(1/\sqrt{N})$ sample complexity can only be translated into an online algorithm with $\tilde{\CO}(T^{2/3})$ regret. Instead, we leverage the DFC structure and the OCO subroutine in our control policy design to provide an online algorithm with $\tilde{\CO}(\sqrt{T})$ regret. Thus, the result in Theorem~\ref{thm:regret upper bound} stands in {\em stark contrast} to the results in \cite{ye2021sample}, since there is a performance gap between the offline learning algorithm proposed for decentralized LQR in \cite{ye2021sample} and that proposed for centralized LQR in \cite{mania2019certainty}.
	
	
{\bf Expected Regret.} Both expected regret and high probability regret have been proved for online learning algorithms for centralized LQR (e.g. \cite{cassel2020logarithmic,cohen2019learning}).  We leave investigating whether the regret bound $\tilde{\CO}(\sqrt{T})$ also holds with a high probability (instead of only in expectation) as future work. 
Note that the online learning algorithms (for centralized LQR) that yield high probability sublinear regret (e.g., \cite{cohen2019learning}) typically incur potentially unbounded regret under a failure event with small probability. In contrast, we upper bound the expected regret of our algorithm on the failure event using the condition in line~10 of the algorithm for each $i\in\V$.
	

\section{Regret Analysis: Proof of Theorem~\ref{thm:regret upper bound}}\label{sec:regret analysis}
Throughout this proof, we assume that Assumptions~\ref{ass:info structure}-\ref{ass:stable A} and \ref{ass:structure of Q and R} hold. We first decompose the regret of Algorithm~\ref{algorithm:control design} defined in Eq.~\eqref{eqn:regret} as
\begin{align}\nonumber
{\tt Regret} &= \underbrace{\sum_{t=0}^{N_0-1}c(x_t^{\tt alg},u_t^{\tt alg})}_{R_0} + \underbrace{\sum_{t=N_0}^{T-1}c(x^{\tt alg}_t,u^{\tt alg}_t)-\sum_{t=N_0}^{T-1}F_t^{\tt pred}(M_{t-h:t})}_{R_1}\\\nonumber
&+\underbrace{\sum_{t=N_0}^{T-1}F^{\tt pred}_t(M_{t-h:t})-\sum_{t=N_0}^{T-1}f_t^{\tt pred}(M_{\tt apx})}_{R_2}+\underbrace{\sum_{t=N_0}^{T-1}f_t^{\tt pred}(M_{\tt apx}) - \inf_{M\in\D_0}\sum_{t=N_0}^{T-1}f_t(M|\Phi,w_{0:t-1})}_{R_3}\\
&+\underbrace{\inf_{M\in\D_0}\sum_{t=N_0}^{T-1}f_t(M|\Phi,w_{0:t-1})-\inf_{M\in\D_0}\sum_{t=N_0}^{T-1}c(x^M_t,u^M_t)}_{R_4}+\underbrace{\inf_{M\in\D_0}\sum_{t=N_0}^{T-1}c(x^M_t,u^M_t)}_{R_5}-TJ_{\star},\label{eqn:regret decomposition}
\end{align}
where $N_0=N+D_{\max}+3h$, $M_{\tt apx}\in\D$ will be specified later with $\D$ given by Eq.~\eqref{eqn:the class of DFCs}, $J_{\star}$ is the optimal cost to \eqref{eqn:dis LQR obj} given by Eq.~\eqref{eqn:opt J}, and $\D_0$ is given by\footnote{To ease our presentation, we assume that $\frac{h}{4}\in\Z_{\ge1}$.}
\begin{align}
\label{eqn:the class of DFCs 0}
\D_0=\Big\{M=\big(M_s^{[k]}\big)_{k\in[\frac{h}{4}],s\in\U}:\nm[\big]{M_s^{[k]}}_F\le\sqrt{n}\kappa p\Gamma^{2D_{\max}+1}\Big\}.
\end{align}

\begin{itemize}[leftmargin=*]
    \item {\bf Proof outline.} We bound each term in the regret decomposition in \eqref{eqn:regret decomposition} by carefully choosing various parameters. Roughly, $R_0$ corresponds to the system identification phase in Algorithm~\ref{algorithm:control design}, $R_1$ (resp., $R_4$) corresponds to the error when approximating the true cost by the prediction cost (resp., counterfactual cost), $R_2$ and $R_3$ correspond to the OCO with memory and delayed feedback subroutine in Algorithm~\ref{algorithm:control design}, and $R_5$ corresponds to the error when approximating the optimal decentralized control policy given by Lemma~\ref{lemma:opt solution} with the (truncated) DFC given by Definition~\ref{def:u_t M}.

    To be more specific, first note that the result in Proposition~\ref{prop:upper bound on est error} shows that the estimation error of the least squares approach used in the first phase of Algorithm~\ref{algorithm:control design} satisfies $\norm{\Delta_N}=\tilde{\CO}(1/\sqrt{N})$ (on the event $\CE$ defined in Eq.~\eqref{eqn:good event}), where $\Delta_N$ is defined in \eqref{eqn:def of Delta} and $N$ is the length of the system trajectory used for system identification. Thus, choosing $N\ge\sqrt{T}$ yields that $\norm{\Delta_N}\le\bar{\varepsilon}$, where $\bar{\varepsilon}$ satisfies $\bar{\varepsilon}\le\tilde{\CO}(1/T^{1/4})$. By the choice of $N$, we show that $R_0$ contributes $\tilde{\CO}(\sqrt{T})$ to $\tt{Regret}$. In the second phase, Algorithm~\ref{algorithm:control design} uses the OCO with memory and delayed feedback subroutine to adaptively choose the control input $u_t^{\tt alg}$ (while satisfying the information constraints). Choosing $M_{\tt apx}\in\D$ carefully, we show that $R_2$ and $R_3$ together contribute $\tilde{\CO}(\bar{\varepsilon}^2T)$ to $\tt{Regret}$. Furthermore, based on the choice of $h$, we show that $R_1$ and $R_4$ together contribute $\tilde{\CO}(1)$ to $\tt{Regret}$, and $R_5$ contributes $\tilde{\CO}(\sqrt{T})+TJ_{\star}$ to $\tt{Regret}$. Putting the above arguments together will give $\E[{\tt Regret}]=\tilde{\CO}(\sqrt{T})$.

    \item{\bf Major challenges in the proof.} We summarize the major differences between our proof and those in the related literature that use DFC and OCO for online centralized LQR.
    
    First, \cite{simchowitz2020improper} studies learning centralized LQR, and proposes an online algorithm that leverages the OCO with memory subroutine from \cite{anava2015online}. Since we study learning the decentralized LQR (with sparsity and delayed constraints), we propose a general OCO algorithm with both memory and delayed feedback, and use it as a subroutine in Algorithm~\ref{algorithm:control design}. We prove the regret of the general OCO algorithm with memory and delayed feedback, and specialize to the learning decentralized LQR setting (Section~\ref{sec:R_2}).

    Second, unlike \cite{anava2015online,simchowitz2020improper} that considers bounded noise, we assume that the noise to system~\eqref{eqn:overall system} is Gaussian. Therefore, we need to first conduct our analysis on an event under which the Gaussian noise is bounded (Section~\ref{sec:good events}), and then bound the (expected) regret of Algorithm~\ref{algorithm:control design} when the Gaussian noise is unbounded (Section~\ref{sec:upper bound on regret}).

    Third, unlike \cite{anava2015online,simchowitz2020improper} that consider high probability regret and set the finite-horizon LQR cost as the benchmark, we prove the expected regret and set the averaged expected infinite-horizon LQR cost as the benchmark. Therefore, the regret decomposition in Eq.~\eqref{eqn:regret decomposition} is different from those in \cite{anava2015online,simchowitz2020improper}.
\end{itemize}

\subsection{Properties under a Good Probabilistic Events}\label{sec:good events}
Before we upper bound the terms in Eq.~\eqref{eqn:regret decomposition}, we introduce several probabilistic events, on which we will prove that several favorable properties hold. To simplify the notations in the sequel, we denote
\begin{equation}
\begin{split}
&R_w=\sigma_w\sqrt{10n\log2T},\\
&\underline{\sigma}=\min\{\sigma_w,\sigma_u\},\ \overline{\sigma}=\max\{\sigma_w,\sigma_u\},
\end{split}
\end{equation}
Define
\begin{equation}
\label{eqn:events}
\begin{split}
\CE_w&=\bigg\{\max_{0\le t\le T-1}\norm{w_t}\le R_w\bigg\},\\
\CE_u&=\bigg\{\max_{0\le t\le N-1}\norm{u_t^{\tt alg}}\le\sigma_u\sqrt{5m\log 4NT}\bigg\},\\
\CE_{\Phi}&=\bigg\{\tr\big(\tilde{\Delta}_N^{\top}V_N\tilde{\Delta}_N\big)\le 4\sigma_w^2n\log\bigg(4nT\frac{\det(V_N)}{\det(\lambda I_{n+m})}\bigg)+2\lambda \norm{\Phi}_F^2\bigg\},\\
\CE_z&=\bigg\{\sum_{t=0}^{N-1}z_tz_t^{\top}\succeq\frac{(N-1)\underline{\sigma}^2}{40}I\bigg\},
\end{split}
\end{equation}
where $N,T$ are the input parameters to Algorithm~\ref{algorithm:control design}, and $\tilde{\Delta}_N$ and $z_t$ are define in \eqref{eqn:def of Delta} and \eqref{eqn:Theta_i and z_Ni}, respectively. Denoting
\begin{align}
\CE&=\CE_{w}\cap\CE_{v}\cap\CE_{\Phi}\cap\CE_z,\label{eqn:good event}
\end{align}
we have the following result that shows $\CE$ holds with a high probability.
\begin{lemma}\cite[Lemma~4]{ye2021sample}
\label{lemma:prob of CE}
For any $N\ge200(n+m)\log 48T$, it holds that $\Prob(\CE)\ge1-1/T$.
\end{lemma}

{\color{black}Recalling $\tilde{\Delta}_N$ and $\Delta_N$ defined in \eqref{eqn:def of Delta}, we then have the following result proved in Appendix~\ref{sec:proofs for good event}, which relates the estimation error of $\hat{A}$ and $\hat{B}$ to that of $\tilde{A}$ and $\tilde{B}$.}
\begin{lemma}	\label{lemma:estimation error of A_hat and B_hat}
Suppose $\norm{\tilde{\Delta}_N}\le\varepsilon$ with $\varepsilon\in\R_{>0}$. Then, $\norm{{\Delta}_N}\le\sqrt{\psi}\varepsilon$, where $\psi\in\Z_{\ge1}$ is the number of strongly connected components in the directed graph $\G(\V,\A)$.
\end{lemma}

The following result adapts \cite[Proposition~1]{ye2021sample} to prove an upper bound on $\norm{\tilde{\Delta}_N}$ and then invokes Lemma~\ref{lemma:estimation error of A_hat and B_hat}; the proof can be found in Appendix~\ref{sec:proofs for good event}.
\begin{proposition}
\label{prop:upper bound on est error}
Denote
\begin{equation*}
z_b=\frac{5\kappa}{1-\gamma}\overline{\sigma}\sqrt{2(\Gamma^2m+m+n)\log2T},
\end{equation*}
and 
\begin{align}
\bar{\varepsilon}&=\min\bigg\{\sqrt{\frac{480\psi}{\sqrt{T}\underline{\sigma}^2}n\sigma_w^2(n+m)\big(\log(T+z_b^2/\lambda)T\big)\Gamma^2},\varepsilon_0\bigg\}.\label{eqn:epsilon_N}
\end{align}
with $\varepsilon_0=\frac{(1-\gamma)\sigma_w}{14p^2q(\Gamma\kappa+1)\sqrt{n}\kappa\Gamma^{2D_{\max}+1}h^2}$. Let the input parameters to Algorithm~\ref{algorithm:control design} satisfy that $T\ge N$ and $\lambda=\sigma_w^2$ with 
\begin{align}
N=\Big\lceil\max\Big\{\frac{480\psi n\sigma_w^2(n+m)\big(\log(T+z_b^2/\lambda)T\big)\Gamma^2}{\underline{\sigma}^2\varepsilon_0^2},\sqrt{T}\Big\}\Big\rceil,\label{eqn:choice of N}
\end{align}
where $\kappa,\gamma,\Gamma$ are given in \eqref{eqn:kappa and gamma}, $D_{\max}$ is given in Eq.~\eqref{eqn:depth of T_i}, and $\psi$ is described in Lemma~\ref{lemma:estimation error of A_hat and B_hat}. Then, on the event $\CE$ defined in Eq.~\eqref{eqn:good event}, it holds that $\norm{\hat{A}-A}\le\bar{\varepsilon}$ and $\norm{\hat{B}-B}\le\bar{\varepsilon}$.
\end{proposition}
{\color{black} Proposition~\ref{prop:upper bound on est error} shows that $\bar{\varepsilon}=\tilde{\CO}(1/T^{1/4})$ by choosing the length $N$ of the system identification phase sufficiently large. Since this phase uses the control $u_t^{\tt }\overset{\text{i.i.d.}}{\sim}\CN(0,\sigma_u^2I_m)$ for $t=0,\dots,N$, Assumption~\ref{ass:stable A} ensures $\norm{x_t^{\tt alg}}=\tilde{\CO}(\sqrt{T})$ for $t=0,\dots,N-1$ on the event $\CE$. }
From now on, we set $N$ as Eq.~\eqref{eqn:choice of N} in Algorithm~\ref{algorithm:control design}. Note from our choices of $h$ and $N$ that $N\ge h$. Moreover, we have the following results proved in Appendix~\ref{sec:proofs for good event}. 

\begin{lemma}
\label{lemma:norm bounds related to u_t star and u_t M}
Suppose the event defined in Eq.~\eqref{eqn:good event} holds. Let $x_t^{\star}$ be the state of system~\eqref{eqn:overall system} when we use the optimal control policy $u_0^{\star},\dots,u_{t-1}^{\star}$ given by Eq.~\eqref{eqn:DFC of u_star}. Then,
\begin{equation}
\norm{u_t^{\star}}\le\frac{pq\kappa\Gamma^{2D_{\max}+1}R_w}{1-\gamma},\ \text{and}\ \norm{x_t^{\star}}\le\frac{pq\kappa\Gamma^{2D_{\max}}R_w}{1-\gamma},
\end{equation}
for all $t\in\Z_{\ge0}$. Letting $M_s^{[k]}=M_{\star,s}^{[k]}$ in Definition~\ref{def:u_t M}, for all $s\in\U$ and all $k\in[g]$, it holds that $M=[M_s^{[k]}]_{k\in[h],s\in\U}$ satisfies $M\in\D$, and that
\begin{equation}
\norm{u_t^M}\le\frac{pq\kappa\Gamma^{2D_{\max}+1}R_w}{1-\gamma},\text{and}\ \norm{x_t^M}\le\frac{2pq\kappa^2\Gamma^{2D_{\max}+2}R_w}{(1-\gamma)^2},
\end{equation}
for all $t\in\Z_{\ge0}$, where $\D$ is defined in Eq.~\eqref{eqn:the class of DFCs}. Moreover, it holds that 
\begin{align}
c(x_t^M,u_t^M)-c(x_t^{\star},u_t^{\star})\le\frac{12\bar{\sigma}p^2q\kappa^4\Gamma^{4D_{\max}+2}R_w^2}{(1-\gamma)^4}\gamma^h,
\end{align}
for all $t\in\Z_{\ge0}$.
\end{lemma}

\begin{lemma}
\label{lemma:bounds on norm of alg and pred input/state}
On the event $\CE$ defined in Eq.~\eqref{eqn:good event}, $\norm{u_t^{\tt alg}}\le R_u$, $\norm{x_t^{\tt alg}}\le R_x$, $\norm{x_t^{\tt pred}(M_{t-h:t-1})}\le R_x$, $\norm{\hat{w}_t}\le R_{\hat{w}}$, and $\norm{\hat{w}_t-w_t}\le\Delta R_w\bar{\varepsilon}$, for all $t\in\{N,\dots,T-1\}$, where 
\begin{align}
R_u&=4qR_wh\sqrt{n}\kappa p^2\Gamma^{2D_{\max}},\label{eqn:R_u alg}\\
R_x&=\frac{\overline{\sigma}\sqrt{20(m+n)}\Gamma\kappa^2}{1-\gamma}+\frac{(\Gamma R_u+R_w)\kappa}{1-\gamma},\label{eqn:R_x}\\
R_{\hat{w}}&=\Big(\frac{\Gamma\kappa}{1-\gamma}+1\Big)\bar{\varepsilon} R_{u} + R_w+\frac{\bar{\varepsilon}\kappa}{1-\gamma}(\kappa\overline{\sigma}\sqrt{20(m+n)}\Gamma+R_w),\label{eqn:R_w hat}\\
\Delta R_w &=\Big(\frac{\Gamma\kappa}{1-\gamma}+1\Big) R_{u}+\frac{\kappa}{1-\gamma}(\kappa\overline{\sigma}\sqrt{20(m+n)}\Gamma+R_w).\label{eqn:delta R_w}
\end{align}
\end{lemma}
Lemma~\ref{lemma:norm bounds related to u_t star and u_t M} shows that under the event $\CE$, the DFC given in Definition~\ref{def:u_t M} can be used to approximate the optimal control policy $u_t^{\star}$ given by Eq.~\eqref{eqn:DFC of u_star} with the approximation error (in terms of the resulting cost) decaying exponentially in $h$. Lemma~\ref{lemma:bounds on norm of alg and pred input/state} provides bounds on the norm of the state, input, estimated disturbance and its estimation error in Algorithm~\ref{algorithm:control design} when the event $\CE$ holds. We now set $R_u$ and $R_x$ in Algorithm~\ref{algorithm:control design} as Eqs.~\eqref{eqn:R_u alg} and \eqref{eqn:R_x}, respectively.  Proposition~\ref{prop:upper bound on est error} and Lemma~\ref{lemma:bounds on norm of alg and pred input/state} yield that the conditions in line~10 in Algorithm~\ref{algorithm:control design} are not satisfied for any $i\in\V$ on the event $\CE$.  The above results will be useful to bound $R_0,\dots,R_5$ in the sequel.

\subsection{Upper Bounds on $R_0$,  $R_1$, $R_4$ and $R_5$}\label{sec:R_0 and R_1}
The following results are proved in  Appendix~\ref{sec:proofs for R_0 R_1 R_4 R_5}.
\begin{lemma}
\label{lemma:upper bound on R_0}
On the event $\CE$ defined in Eq.~\eqref{eqn:good event}, it holds that
$R_0\le\Big(\sigma_1(Q)R_x^2+\sigma_1(R)\frac{R_u^2\sigma_u^2}{\sigma_w^2}\Big)N_0$.
\end{lemma}

\begin{lemma}
\label{lemma:upper bound on R_1}
On the event $\CE$ defined in Eq.~\eqref{eqn:good event}, it holds that $R_1\le2R_x\sigma_1(Q)(\Gamma R_u+R_w)\frac{\kappa\gamma^h}{1-\gamma}T$.
\end{lemma}

\begin{lemma}
	\label{lemma:upper bound on R_4}
	On the event $\CE$ defined in \eqref{eqn:good event}, it holds that 
	\begin{align}\nonumber
		R_4\le8\sigma_1(Q)q^2nh^2\kappa^2 p^2\Gamma^{4D_{\max}+4}\frac{R_w^2\kappa^2\gamma^h}{(1-\gamma)^2}T.
	\end{align}
\end{lemma}

\begin{lemma}
	\label{lemma:upper bound on R_5}
	$\E\big[\ind\{\CE\}R_5\big]\le\frac{12\bar{\sigma}p^2q\kappa^4\Gamma^{4D_{\max}+2}R_w^2}{(1-\gamma)^4}\gamma^{\frac{h}{4}}T+TJ_{\star}$.
\end{lemma}
Note that the upper bounds shown in Lemma~\ref{lemma:upper bound on R_0} naturally hold for the expected values $\E[\ind\{\CE\}R_0]$, $\E[\ind\{\CE\}R_1]$ and $\E[\ind\{\CE\}R_4]$. $R_0$ corresponds to the system identification phase in Algorithm~\ref{algorithm:control design}. Based on the choice of the system identification phase length $N$ and Lemma~\ref{lemma:bounds on norm of alg and pred input/state}, we show in Lemma~\ref{lemma:upper bound on R_0} that $\E[\ind\{\CE\}R_0]=\tilde{\CO}(\sqrt{T})$, where recall that $N_0=N+D_{\max}+3h$. $R_1$ (resp., $R_4$) corresponds to the error when approximating the true cost by the prediction cost (resp., counterfactual cost), and the proof of the upper bounds on $R_1,R_4$ in Lemma~\ref{lemma:upper bound on R_0} uses the results in Lemma~\ref{lemma:bounds on norm of alg and pred input/state}. Recalling $0<\gamma<1$, one can choose $h$ to be sufficiently large as $h=\CO(\log T)$ such that $\gamma^hT=\tilde{\CO}(1)$, which implies via Lemma~\ref{lemma:upper bound on R_0} that $\E[\ind\{\CE\}(R_1+R_4)]=\tilde{\CO}(1)$. Finally, $R_5$ corresponds to the error when approximating the optimal decentralized control policy given by Lemma~\ref{lemma:opt solution} with the (truncated) DFC and we show that $\E[\ind\{\CE\}R_5]=\tilde{\CO}(\sqrt{T})+TJ_{\star}$. 

\begin{remark}
W{\tiny }e upper bound $\E[\ind\{\CE\}R_5]$ rather than $\ind\{\CE\}R_5$ in Lemma~\ref{lemma:upper bound on R_5}. The reason is that by taking the expectation $\E[\cdot]$ (with respect to the disturbances $w_t$), we can leverage the assumption made before that the disturbances $w_{t,i}$ and $w_{t,j}$ are independent for different $i,j\in\V$ of system~\eqref{eqn:system for node i} so that $\E[w_{t,i}w_{t,j}]=0$ for $i\neq j$. We then use Lemma~\ref{lemma:norm bounds related to u_t star and u_t M} to relate the optimal finite-horizon decentralized LQR cost achieved by the DFC to the optimal expected averaged infinite-horizon decentralized LQR cost, 
and get the upper bound on $\E[\ind\{\CE\}R_5]$ in Lemma~\ref{lemma:upper bound on R_5}.
\end{remark}

\subsection{Upper Bound on $R_2$ and $R_3$}\label{sec:R_2}
{\color{black}$R_2$ and $R_3$ are due to the OCO subroutine given by Algorithm~\ref{algorithm:OCO} with memory and delayed feedback. To upper bound $R_2,R_3$, we will apply Proposition~\ref{prop:true f with memory upper bound} shown for the general OCO algorithm. } 
{\color{black}Recall that the OCO subroutine in Algorithm~\ref{algorithm:control design} is applied to the function sequence $F_t^{\tt pred}:\D^{h+1}\to\R$ and $f_t^{\tt pred}:\D\to\R$ for $t\in\{N+D_{\max},\dots,T+D_{\max}-1\}$, where $F_t^{\tt pred}(\cdot)$ and $f_t^{\tt pred}(\cdot)$ are defined in Definition~\ref{def:true prediction cost} and $\D$ defined in Eq.~\eqref{eqn:the class of DFCs} is convex.} We first show that Assumptions~\ref{ass:cost F_t}-\ref{ass:gradient bound} hold in our problem setting. We need to define the following. 
\begin{definition}
\label{def:conditional prediction loss}
Let $\{\F_t\}_{t\ge N}$ be a filtration with $\F_t=\sigma(w_N,\dots,w_t)$ for all $t\ge N$. For any $t\ge N+k_f$ with $k_f\in\Z_{\ge0}$ and any $M\in\D$, define $f_{t;k_f}^{\tt pred}(M)=\E\big[f_t^{\tt pred}(M)|\F_{t-k_f}\big]$.
\end{definition}
Note that in order to leverage the result in Proposition~\ref{prop:true f with memory upper bound}, we first view the function $f_t^{\tt alg}(\cdot)$ as a function of ${\tt Vec}(M)$, where ${\tt Vec}(M)$ denotes the vector representation of $M=[M_s^{[k]}]_{s\in\U,k\in[h]}$. We then show in Lemma~\ref{lemma:lipschitz of f pred} that $f_t^{\tt alg}(\cdot)$ is Lipschitz with respect to ${\tt Vec}(M)$.\footnote{Note that $\norm{P}_F^2=\norm{{\tt Vec}(P)}^2$ for all $P\in\R^{n\times n}$, which explains why we define $\D$ in Eq.~\eqref{eqn:the class of DFCs} based on $\norm{\cdot}_F^2$. To be more precise, projecting $M=[M_s^{[k]}]_{s\in\U,k\in[h]}$ onto the set $\D$ is equivalent to projecting ${\tt Vec}(M)$ onto a Euclidean ball in $\R^{{\tt dim}({\tt Vec}(M))}$.} In Lemmas~\ref{lemma:upper bound on norn of hessian of f_pred}-\ref{lemma:strongly convex of conditional cost}, we assume for notational simplicity that $M$ is already written in its vector form ${\tt Vec}(M)$. The proofs of Lemmas~\ref{lemma:lipschitz of f pred}-\ref{lemma:strongly convex of conditional cost} are included in Appendix~\ref{sec:proofs R_2}.
\begin{lemma}
	\label{lemma:lipschitz of f pred}
	Suppose the event $\CE$ defined in Eq.~\eqref{eqn:good event} holds. For any $t\in\{N,\dots,T-1\}$, $f_t^{\tt pred}:\D\to\R$ is $L_f^{\prime}$-Lipschitz and $F_t^{\tt pred}:\D^{h+1}\to\R$ is $L_c^{\prime}$-coordinatewise-Lipschitz, where
	\begin{align}
		L_f^{\prime}=L_c^{\prime}=2(R_u+R_x)\max\{\sigma_1(Q),\sigma_1(R)\}\frac{(\Gamma\kappa+1-\gamma)\sqrt{qh}R_{\hat{w}}}{1-\gamma}.
	\end{align}
\end{lemma}

\begin{lemma}
	\label{lemma:upper bound on norn of hessian of f_pred}
	On the event $\CE$ defined in Eq.~\eqref{eqn:good event}, it holds that $\norm{\nabla^2 f_t^{\tt pred}(M)}\le\beta^{\prime}$ for all $t\ge\{N,\dots,T-1\}$ and all $M\in\D$, where
	\begin{align}\nonumber
		\beta^{\prime}=\frac{2qhR_{\hat{w}}^2\big(\sigma_1(Q)(1-\gamma)^2+\sigma_1(R)\kappa^2\Gamma^2\big)}{(1-\gamma)^2}.
	\end{align}
\end{lemma}

\begin{lemma}
	\label{llemma:gradient error} 
	On the event $\CE$ defined in Eq.~\eqref{eqn:good event}, it holds that 
	\begin{align}
		\nm[\Big]{\frac{\partial f_t(M)}{\partial M}-\frac{\partial f_t^{\tt pred}(M)}{\partial M}}\le\Delta L_f\bar{\varepsilon},\label{eqn:norm of gradient error}
	\end{align}
	for all $t\in\{N,\dots,T-1\}$ and all $M\in\D$, where 
	\begin{align}
		\Delta L_f = \frac{34\sigma_1(Q)\sqrt{qh}R_{\hat{w}}\kappa^2\tilde{\Gamma}}{(1-\gamma)^2}\Big((R_{\hat{w}}+\tilde{\Gamma}R_u)\frac{34\kappa^2}{(1-\gamma)^2}+R_x\Big),
	\end{align}
	with $\tilde{\Gamma}=\Gamma+1$.
\end{lemma}

\begin{lemma}
	\label{lemma:strongly convex of conditional cost}
	Suppose the event $\CE$ defined in Eq.~\eqref{eqn:good event} holds, and let $k_f\ge D_{\max}+h+1$, where $D_{\max}$ is defined in Eq.~\eqref{eqn:depth of T_i}. Then, for any $t\ge N+k_f$, the function $f_{t;k_f}(\cdot)$ given by Definition~\ref{def:conditional prediction loss} is $\frac{\sigma_m(R)\sigma_w^2}{2}$-strongly convex.
\end{lemma}

Recalling the expression of $R_2$ given in Eq.~\eqref{eqn:regret decomposition} and leveraging Lemmas~\ref{lemma:lipschitz of f pred}-\ref{lemma:strongly convex of conditional cost}, we can specialize Proposition~\ref{prop:true f with memory upper bound} shown for the general OCO with memory subroutine to the following upper bound on $R_2$.
\begin{lemma}
\label{lemma:upper bound on R_2} 
Let $\eta_t=\frac{3}{\alpha^{\prime}t}$ for all $t\in\{N+D_{\max},\dots,T+D_{\max}-1\}$ in Algorithm~\ref{algorithm:control design}, where $\alpha^{\prime}\triangleq\frac{\sigma_m(R)\sigma_w^2}{2}$. Then, for any $M_{\tt apx}\in\D$ and any $\delta>0$, the following holds with probability at least $1-\delta$:
\small
\begin{align}\nonumber
	R_2&\le-\frac{\alpha^{\prime}}{12}\sum_{t=N}^{T-1}\norm{{\tt Vec}(M_t)-{\tt Vec}(M_{\tt apx})}^2\\\nonumber
	&+\CO(1)\bigg(\sum_{t=N_0}^{T-1}\frac{\Delta L_f^2}{\alpha^{\prime}}\norm{\bar{\varepsilon}}^2+\frac{k_fd^{\prime}L_f^{\prime2}}{\alpha^{\prime}}\log\Big(\frac{T\big(1+\log_+(\alpha^{\prime} G^{\prime2})\big)}{\delta}\Big)\\
	&\qquad\qquad\qquad+\alpha^{\prime} G^{\prime2}(k_f+D_{\max})+\frac{L_g^{\prime}}{\alpha^{\prime}}\big(L_g^{\prime}D_{\max}+(\beta^{\prime} G^{\prime}+L_f^{\prime})k_f+L_c^{\prime}h^2\big)\log T\bigg),\label{eqn:upper bound on R_2}
\end{align}
\normalsize
where $M_t$ is chosen by Algorithm~\ref{algorithm:control design} for all $t\in\{N,\dots,T-1\}$, $\alpha^{\prime},\beta^{\prime},L_f^{\prime},L_g^{\prime},L_c^{\prime},\Delta L_f$ are given in Lemmas~\ref{lemma:lipschitz of f pred}-\ref{lemma:strongly convex of conditional cost}, $\bar{\varepsilon}$ is given in Eq.~\eqref{eqn:epsilon_N}, $L_g^{\prime}=L_f^{\prime}+\Delta L_f$, $k_f=3h+D_{\max}+1$, $d^{\prime}={\tt dim}({\tt Vec}(M))=h\sum_{s\in\U}n_{s}n_{\CL_s}$, and $G^{\prime}=4\sqrt{qhn}\kappa p\Gamma^{2D_{\max}+1}$.
\end{lemma}

The upper bound on $R_2$ in Lemma~\ref{lemma:upper bound on R_2} holds for all $M_{\tt apx}\in\D$. To upper bound $R_3$, we choose a particular $M_{\tt apx}\in\D$ and upper bound the difference between $f_t^{\tt pred}(M_{\tt apx})$ and $f_t(\tilde{M}_{\star}|\Phi,w_{0:t-1})$, where $\tilde{M}_{\star}\in\arg\inf_{M\in\D_0}\sum_{t=N_0}^{T-1}f_t(M|\Phi,w_{0:t-1})$ with $\D_0$ defined in Eq.~\eqref{eqn:the class of DFCs 0} and $f_t(\cdot|\Phi,w_{0:t-1})$ given by Definition~\ref{def:counterfactual cost}. We begin with the  following intermediate lemma; the proof is included in Appendix~\ref{sec:proof R_3}.

\begin{lemma}
\label{lemma:f_pred minus f_counterfactual}
Suppose the event $\CE$ defined in Eq.~\eqref{eqn:good event} holds. Let 
\begin{align}
\tilde{M}_{\star}\in\arg\inf_{M\in\D_0}\sum_{t=N_0}^{T-1}f_t(M|\Phi,w_{0:t-1}),\label{eqn:tilde M star}
\end{align}
where $\D_0$ is defined in Eq.~\eqref{eqn:the class of DFCs 0}, and $f_t(\cdot|\Phi,w_{0:t-1})$ is given by Definition~\ref{def:counterfactual cost}. Then, for any $t\in\{N_0,\dots,T-1\}$ and any $M_{\tt apx}\in\D$,
\begin{multline}
f_t^{\tt pred}(M_{\tt apx})-f_t(\tilde{M}_{\star}|\Phi,w_{0:t-1})\le2\big(R_u\sigma_1(R)+R_x\sigma_1(Q)\big)\frac{\Gamma\kappa}{1-\gamma}\\\times\max_{k\in\{t-h,\dots,t\}}\nm[\big]{u_k(M_{\tt apx}|\hat{w}_{0:k-1})-u_k(\tilde{M}_{\star}|w_{0:k-1})}.\label{eqn:f_pred minus f_counterfactual}
\end{multline}
\end{lemma}

Next, the following lemma finds an $M_{\tt apx}$ that depends on $\tilde{M}_{\star}$ such that the difference between $\tilde{M}_{\star}$ and $M_{\tt apx}$ can be related to the difference between $M_t$ and $M_{\tt apx}$. The proof of Lemma~\ref{lemma:u M_apx minus u tilde M_star} is provided in Appendix~\ref{sec:proof R_3}.
\begin{lemma}
\label{lemma:u M_apx minus u tilde M_star}
Suppose the event $\CE$ defined in Eq.~\eqref{eqn:good event} holds. Then, there exists $M_{\tt apx}\in\D$ such that for any $k\in\{N+2h,\dots,T-1\}$ and any $\mu\in\R_{>0}$,
\begin{align}\nonumber
\nm[\big]{u_k(\tilde{M}_{\star}|w_{0:k-1})-u_k(M_{\tt apx}|\hat{w}_{0:k-1})}&\le\bar{\varepsilon}(\Gamma R_u+R_w)p^2qh\sqrt{n}\frac{\kappa^2\Gamma^{2D_{\max}+1}\gamma^{\frac{h}{4}}}{4(1-\gamma)}+\frac{p^2qh^2}{8\mu}\sqrt{n}\kappa\Gamma^{2D_{\max}+1}(\kappa\Gamma+1)\bar{\varepsilon}^2\\
+\frac{p^5q^3n^{\frac{3}{2}}}{8}h^5\kappa^4\Gamma^{6D_{\max}+4}&(\kappa\Gamma+1)R_{\hat{w}}\bar{\varepsilon}^2+\frac{\mu}{2}\max_{t_1\in\{k-h,\dots,k-1\}} \nm[\big]{u_{t_1}(M_{t_1}-M_{\tt apx}|\hat{w}_{0:{t_1}-1})}^2.\label{eqn:upper bound on the difference between two u's}
\end{align}
where $M_t$ is chosen by Algorithm~\ref{algorithm:control design} for all $t\in\{N,\dots,T-1\}$.
\end{lemma}

Combining Lemmas~\ref{lemma:f_pred minus f_counterfactual}-\ref{lemma:u M_apx minus u tilde M_star} yields the following upper bound on $R_3$; the proof is included in Appendix~\ref{sec:proof R_3}.
\begin{lemma}
\label{lemma:upper bound on R_3}
Suppose the event $\CE$ defined in Eq.~\eqref{eqn:good event} holds. Then, there exists $M_{\tt apx}\in\D$ such that for any $\mu\in\R_{>0}$,
\begin{align}\nonumber
R_3&\le \big(R_u\sigma_1(R)+R_x\sigma_1(Q)\big)\frac{2\Gamma\kappa}{1-\gamma}\Big(\bar{\varepsilon}(\Gamma R_u+R_w)p^2qh\sqrt{n}\frac{\kappa^2\Gamma^{2D_{\max}+1}\gamma^{\frac{h}{4}}}{4(1-\gamma)}T+\frac{p^2qh^2}{8\mu}\sqrt{n}\kappa\Gamma^{2D_{\max}+1}(\kappa\Gamma+1)\bar{\varepsilon}^2T\\
&\qquad\qquad\qquad+\frac{p^5q^3n^{\frac{3}{2}}}{8}h^5\kappa^4\Gamma^{6D_{\max}+4}(\kappa\Gamma+1)R_{\hat{w}}\bar{\varepsilon}^2T+qh^2R_{\hat{w}}\mu\sum_{t=N}^{T-1}\nm[\big]{{\tt Vec}(M_t)-{\tt Vec}(M_{\tt apx})}^2\Big).\label{eqn:upper bound on R_3 result}
\end{align}
\end{lemma}

\subsection{Upper bound on {\tt Regret}}\label{sec:upper bound on regret}
We are now in place to upper bound $\E[{\tt Regret}]$, where ${\tt Regret}$ is defined in Eq.~\eqref{eqn:regret}. First, we recall that Lemma~\ref{lemma:upper bound on R_2} provides a high probability upper bound on $R_2$. Hence, we define $\CE_{R_2}$ to be the event on which the upper provided in Lemma~\ref{lemma:upper bound on R_2} holds with probability at least $1-1/T$, i.e., we let $\delta=1/T$ in \eqref{eqn:upper bound on R_2}. We then have from Lemma~\ref{lemma:upper bound on R_2} that on the event $\CE\cap\CE_{R_2}$,
\begin{align}\nonumber
R_2&\le-\frac{\alpha^{\prime}}{12}\sum_{t=N}^{T-1}\norm{{\tt Vec}(M_t)-{\tt Vec}(M_{\tt apx})}^2+\CO(1)\bigg(\sum_{t=N_0}^{T-1}\frac{\Delta L_f^2}{\alpha^{\prime}}\norm{\bar{\varepsilon}}^2+\frac{k_fd^{\prime}L_f^{\prime2}}{\alpha^{\prime}}\log\Big(T\big(1+\log_+(\alpha^{\prime} G^{\prime2})\big)\Big)\\\nonumber
&\qquad\qquad\qquad\qquad\qquad+\alpha^{\prime} G^{\prime2}(k_f+D_{\max})+\frac{L_g^{\prime}}{\alpha^{\prime}}\big(L_g^{\prime}D_{\max}+(\beta^{\prime} G^{\prime}+L_f^{\prime})k_f+L_c^{\prime}h^2\big)\log T\bigg).
\end{align}
Next, we have from Lemma~\ref{lemma:upper bound on R_3} that on the event $\CE\cap\CE_{R_2}$, the following holds for any $\mu\in\R_{>0}$:
\begin{align}\nonumber
R_3&\le \big(R_u\sigma_1(R)+R_x\sigma_1(Q)\big)\frac{2\Gamma\kappa}{1-\gamma}\Big(\bar{\varepsilon}(\Gamma R_u+R_w)p^2qh\sqrt{n}\frac{\kappa^2\Gamma^{2D_{\max}+1}\gamma^{\frac{h}{4}}}{4(1-\gamma)}T+\frac{p^2qh^2}{8\mu}\sqrt{n}\kappa\Gamma^{2D_{\max}+1}(\kappa\Gamma+1)\bar{\varepsilon}^2T\\\nonumber
&\qquad\qquad\qquad+\frac{p^5q^3n^{\frac{3}{2}}}{8}h^5\kappa^4\Gamma^{6D_{\max}+4}(\kappa\Gamma+1)R_{\hat{w}}\bar{\varepsilon}^2T+qh^2R_{\hat{w}}\mu\sum_{t=N}^{T-1}\nm[\big]{{\tt Vec}(M_t)-{\tt Vec}(M_{\tt apx})}^2\Big).
\end{align}
Thus, setting $\mu$ properly and summing up the above two inequalities, we cancel the $\sum_{t=N}^{T-1}\nm[\big]{{\tt Vec}(M_k)-{\tt Vec}(M_{\tt apx})}^2$ term. Moreover, since $h\ge\frac{4}{1-\gamma}\log T$ by our choice of $h$, one can show that $\gamma^{h/4}\le\frac{1}{T^{1/4}}$. Noting from Eq.~\eqref{eqn:epsilon_N} that $\bar{\varepsilon}\le\tilde{\CO}(\frac{1}{T^{1/4}})$, we can combine the above arguments together and obtain that 
\begin{align}
\E\big[\ind\{\CE\cap\CE_{R_2}\}(R_2+R_3)\big]=\tilde{\CO}(\sqrt{T}).\label{eqn:R_2+R_3}
\end{align}
Recalling from Lemma~\ref{lemma:prob of CE} that $\Prob(\CE)\ge1-1/T$, we have from the union bound that $\Prob(\CE\cap\CE_{R_2})\ge1-2/T$. Similarly, combining the results in Lemmas~\ref{lemma:upper bound on R_0}, \ref{lemma:upper bound on R_1}, \ref{lemma:upper bound on R_4} and \ref{lemma:upper bound on R_5}, we obtain that 
\begin{align}
\E\big[\ind\{\CE\cap\CE_{R_2}\}(R_0+R_1+R_4+R_5)\big]=\tilde{\CO}(\sqrt{T})+TJ_{\star}.\label{eqn:R_0+R_1+R_4+R_5}
\end{align}
It then follows from Eqs.~\eqref{eqn:R_2+R_3}-\eqref{eqn:R_0+R_1+R_4+R_5} that 
\begin{align}
\E\bigg[\ind\{\CE\cap\CE_{R_2}\}\sum_{t=0}^{T-1}c(x_t^{\tt alg},u_t^{\tt alg})\bigg]=\E\big[\ind\{\CE\cap\CE_{R_2}\}\sum_{i=0}^5R_i\big]=\tilde{\CO}(\sqrt{T})+TJ_{\star}.\label{eqn:regert on good event}
\end{align}

Finally, in order to prove $\E[{\tt Regret}]=\tilde{\CO}(\sqrt{T})$, it remains to upper bound $\E\big[\ind\{(\CE\cap\CE_{R_2})^c\}\sum_{t=0}^{T-1}c(x_t^{\tt alg},u_t^{\tt alg})\big]$.  We prove the following result in Appendix~\ref{sec:proof of failure regret}.
\begin{lemma}
\label{lemma:failure regret}
It holds that
\begin{align}
	\E\big[\ind\{(\CE\cap\CE_{R_2})^c\}\sum_{t=0}^{T-1}c(x_t^{\tt alg},u_t^{\tt alg})\big]=\tilde{\CO}(1).\label{eqn:regret on good event c}
\end{align}
\end{lemma}
Combining the above arguments complete the proof of Theorem~\ref{thm:regret upper bound}.

\section{Extensions}\label{sec:extensions}
\subsection{General Information Structure}\label{sec:genearl info structure}
We show that our analyses and results can be extended to system~\eqref{eqn:system for node i} with general information structure even when Assumption~\ref{ass:info structure} does not hold. To this end, let us consider the following decentralized LQR problem which is the finite-horizon counterpart of \eqref{eqn:dis LQR obj}:
\begin{equation}
\label{eqn:dis LQR obj finite-horizon}
\begin{split}
\min_{u_0^M,\dots, u_{T-1}^M}&\E\Big[\frac{1}{T}\sum_{t=0}^{T-1}(x_t^{\top}Qx_t+u_t^{\top}Ru_t)\Big]\\
s.t.\ x_{t+1}&=Ax_t+Bu_t+w_t,\\
u_t^M &= \sum_{s\in\U}I_{\V,s}\sum_{k=1}^h M_{s}^{[k]}\eta_{t-1-k,s} \ \forall i\in\V,\forall t\in\{0,\dots,T-1\},\\
M &= [M_s^{[k]}]_{k\in[h],s\in\U}\in\D^{\prime}_0,
\end{split}
\end{equation}
where $\D^{\prime}_0\triangleq\{M=[M_s^{[k]}]_{k\in[h],s\in\U}:\norm{M_s^{[k]}}_F\le\kappa^{\prime}\}$ with $\kappa^{\prime}\in\R_{>0}$ and $h\in\Z_{\ge1}$ to be parameters that one can choose, and we denote the optimal solution to \eqref{eqn:dis LQR obj finite-horizon} as $J^{\prime}$. We note from our arguments in Section~\ref{sec:DFC} that given the system matrices $A$ and $B$, \eqref{eqn:dis LQR obj finite-horizon} in fact minimizes over a class of linear state-feedback controller $u_0^M,\dots,u_{T-1}^M$ with $u_{t,i}^M\in\pi_i(\I_{t,i})$ for all $i\in\V$ and all $t\in\{0,\dots,T-1\}$. Since Assumption~\ref{ass:info structure} does not hold, Lemma~\ref{lemma:opt solution} cannot be applied to ensure that the class of linear controller is optimal to the decentralized LQR problem. However, the class of linear controller considered in \eqref{eqn:dis LQR obj finite-horizon} already includes any linear state-feedback controller (that depends on the states that are at most $h$ time steps in the past) under the information constraint. Now, we apply Algorithm~\ref{algorithm:control design} to \eqref{eqn:dis LQR obj finite-horizon} and compute $u_t^M$ in \eqref{eqn:dis LQR obj finite-horizon} as $u_t^{\tt alg}$ given by the algorithm, since the system matrices $A$ and $B$ are unavailable to obtain $\eta_{t-1-k,s}$ for $u_t^M$. We then consider the following regret of Algorithm~\ref{algorithm:control design}:
\begin{equation}
\label{eqn:regret prime}
{\tt Regret}^{\prime}=\sum_{t=0}^{T-1}\big(c(x_t^{\tt alg},u_t^{\tt alg}) - J_{\star}^{\prime}\big). 
\end{equation}
Following similar arguments to those in the proof of Theorem~\ref{thm:regret upper bound}, one can decompose ${\tt Regret}^{\prime}$ as
\begin{equation}
{\tt Regret}^{\prime}=\sum_{i=0}^5 R_i^{\prime} - T J_{\star}^{\prime},
\end{equation}
where $R_0^{\prime},\dots,R_5^{\prime}$ are defined similarly to $R_0,\dots,R_5$ in the decomposition in Eq.~\eqref{eqn:regret decomposition} with $D_0$ and $D$ replaced by $D_0^{\prime}$ and $D^{\prime}$, respectively, where $D^{\prime}\triangleq\{M=[M_s^{[k]}]_{k\in[h],s\in\U}:\norm{M_s^{[k]}}_F\le2\kappa^{\prime}\}$. Moreover, following similar arguments to those in the proof of Lemma~\ref{lemma:upper bound on R_5}, one can show that $\E[\ind\{\E\}R_5]\le TJ^{\prime}$. Thus, one can obtain similar results for ${\tt Regret}^{\prime}$ to those in Theorem~\ref{thm:regret upper bound}  with $\E[{\tt Regret}^{\prime}]=\tilde{\CO}(\sqrt{T})$.

\subsection{Stabilizable Systems}\label{sec:unstable system}
{\color{black}Assumption~\ref{ass:stable A} can be relaxed to assuming $(A,B)$ is stabilizable with a priori known stabilizing $K$, as in works on centralized LQR \cite{mania2019certainty,cohen2019learning,cassel2020logarithmic,simchowitz2020naive}.} With a known $K\in\R^{m\times n}$ such that $A+BK$ is stable and $u_t=Kx_t$ satisfies the information constraint, we can consider $u_t=Kx_t+u_t^{\prime}$ with $u_{t,i}^{\prime}\in\pi_i(\I_{t,i})$ for all $i\in\V$ in \eqref{eqn:dis LQR obj}, and optimize over $u_0^{\prime},u_1^{\prime},\dots$. Note that $u_t=Kx_t+u_t^{\prime}$ also satisfies the information constraint. Replacing $A$ with $\bar{A}\triangleq A+BK$, one can check that our analyses and results follow verbatim. 
{\color{black} Finally, using the techniques from \cite{chen2021black,yu2022online}, one can add a preliminary phase in Algorithm~\ref{algorithm:control design} to obtain a stabilizing controller from system trajectory, which incurs an extra regret of $2^{\CO(n)}$ to the algorithm.}   

\section{Numerical Results}\label{sec:numerical results}
For the system in Example~\ref{exp:running example}, we set $\sigma_w=1$ and $Q=R=I_3$, and generate the system matrices $A$ and $B$ randomly with each nonzero entry chosen independently from a normal distribution. In Fig.~\ref{fig:regret per round}, we plot {\tt Regret} defined in Eq.~\eqref{eqn:regret} when Algorithm~\ref{algorithm:control design} is used and $N=\sqrt{T}$. The results are averaged over $10$ independent experiments to approximate $\E[{\tt Regret}]$, and the shaded areas in Fig.~\ref{fig:regret per round} display quartiles. 
{\color{black}We see from Fig.~\ref{fig:regret per round} that when $T$ increases, ${\tt Regret}/T$ decreases, ${\tt Regret}/\sqrt{T}$ slowly increases and ${\tt Regret}/(\sqrt{T}\log T)$ remains unchanged (after $T\ge300$). Thus, Fig.~\ref{fig:regret per round} matches with the regret bound $\E[{\tt Regret}]=\CO(\sqrt{T}\log T)$ in Theorem~\ref{thm:regret upper bound} and also show that the bound is tight for certain problem instances.}

\begin{figure}[htbp]
	\centering
	\subfloat[a][${\tt Regret}/T$ vs. $T$]{
		\includegraphics[width=0.32\linewidth]{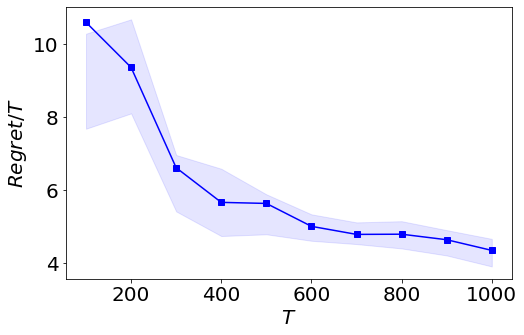}}
	\subfloat[b][${\tt Regret}/\sqrt{T}$ vs. $T$]{    \includegraphics[width=0.32\linewidth]{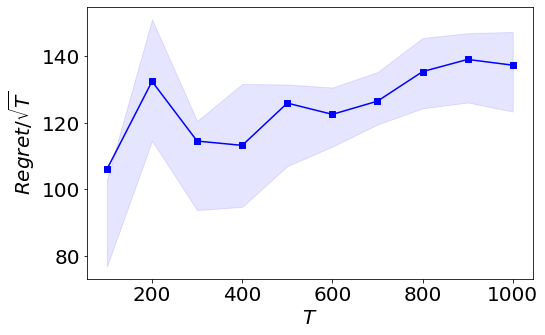}}
	\subfloat[c][${\tt Regret}/(\sqrt{T}\log T$) vs. $T$]{
		\includegraphics[width=0.32\linewidth]{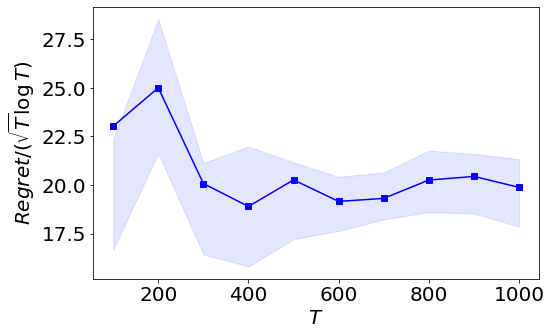}}
	\caption{Results for {\tt Regret} when Algorithm~\ref{algorithm:control design} is applied to Example~\ref{exp:running example}.}
	\label{fig:regret per round}
\end{figure}

{\color{black} Next, we compare our approach 
	with the one in \cite{ye2021sample} based on the certainty equivalence (CE) approach. We consider the same instance of Example~\ref{exp:running example} as that constructed above. 
	The results in Table~\ref{tab:compare}(a) are averaged over $5$ independent experiments. As we argued in Section~\ref{sec:main results}, there is no regret guarantee of the algorithm in \cite{ye2021sample} since it is an offline algorithm; thus, we can see that Algorithm~\ref{algorithm:control design} achieves better performance. 
	
	We further compare the performances of Algorithm~\ref{algorithm:control design} under different information patterns. Specifically, we use the information propagation pattern in Example~\ref{exp:running example} given in the directed graph $\G(\V,\A)$ in Fig.~\ref{fig:directed graph} as the benchmark (i.e., {\bf Info pattern $1$}). We obtain {\bf Info pattern $2$} by removing the (solid) directed edge from node $1$ to node $2$ in Fig.~\ref{fig:directed graph}, and  {\bf Info pattern~3} by further removing the direct edge from node $2$ to node $3$. We keep all the other problem parameters the same as Example~\ref{exp:running example}. Neither of the {\bf Info patterns} $2$ or $3$ is partially nested,  
	so there is no closed-form solution to \eqref{eqn:dis LQR obj} (or \eqref{eqn:dis LQR obj finite-horizon}) as that given by Lemma~\ref{lemma:opt solution}. 
	Hence, we compute the cost $\sum_{t=1}^Tc(x_t^{\tt alg},u_t^{\tt alg})$ of Algorithm~\ref{algorithm:control design} directly 
	and obtain the results in Table~\ref{tab:compare}(b) which are averaged over $5$ independent experiments. In general, {\bf Info pattern~1} (resp., {\bf Info pattern~3}) tends to have the lowest (resp., highest) cost among the three for $t=100,\dots,1000$, possibly because  graph $\G(\V,\A)$ with more edges implies that the controller at $i\in\V$ becomes more informative as it receives more state information from other subsystems in $\V$.}
\begin{table}[!ht]
	\centering
	\subfloat[Compare {\tt Regret} of different algorithms under different $T$'s.]{
		\resizebox{!}{0.04\textwidth}{
			\centering
			\begin{tabular}{c c c c c c c c c c c}
				{\tt Regret} ${\bf(\times10^3)}$ &  $\bf{100}$  &  $\bf{200}$ &   $\bf{300}$ &  $\bf{400}$ &  $\bf{500}$ &   $\bf{600}$ & $\bf{700}$ & $\bf{800}$ & $\bf{900}$ & $\bf{1000}$\\
				\hline 
				{\bf Algorithm~\ref{algorithm:control design}} & $2.27$   & $3.08$  & $4.13$ & $4.83$ & $5.89$ & $5.65$ & $6.59$ & $7.97$ & $7.32$ & $8.66$ \\
				
				{\bf CE} &  $3.16$   & $4.13$  & $4.99$ & $5.92$   & $6.80$  & $7.83$ & $8.62$ & $9.51$ & $10.86$ & $11.10$ 
			\end{tabular}
	}}\\
	\subfloat[Compare costs of Algorithm~\ref{algorithm:control design} for different info patterns  under different $T$'s.]{
		\resizebox{!}{0.05\textwidth}{
			\centering
			\begin{tabular}{c c c c c c c c c c c}
				{\bf Cost} ${\bf(\times10^4)}$ &  $\bf{100}$  &  $\bf{200}$ &   $\bf{300}$ &  $\bf{400}$ &  $\bf{500}$ &   $\bf{600}$ & $\bf{700}$ & $\bf{800}$ & $\bf{900}$ & $\bf{1000}$\\
				\hline 
				{\bf Info pattern $1$} & ${\bf0.95}$   & ${\bf1.01}$  & $1.73$ & $4.81$ & $4.40$ & $4.41$ & $5.21$ & ${\bf8.04}$ & $\bf{7.91}$ & ${\bf6.66}$ \\
				
				{\bf Info pattern $2$} &  $0.81$   & $1.36$  & $1.63$ & $4.33$   & $4.17$  & $3.60$ & $4.69$ & $8.46$ & $7.79$ & $8.03$ \\
				
				{\bf Info pattern $3$}&  ${\bf1.28}$   & ${\bf1.77}$  & ${\bf2.89}$ & $4.12$   & ${\bf4.45}$  & $3.59$ & $4.73$ & ${\bf8.90}$ & $7.80$ & $7.53$
		\end{tabular}}
	}
	\caption{{\color{black}Comparisons of different algorithms and information patterns.}}\label{tab:compare}
\end{table}

\section{Conclusion}
We considered the problem of learning decentralized linear quadratic regulator under an information constraint on the control policy, with unknown system models. We proposed a model-based online learning algorithm that adaptively designs a control policy when new data samples from a single system trajectory become available. Our algorithm design was built upon a disturbance-feedback representation of state-feedback controllers, and an online convex optimization with memory and delayed feedback subroutine. We showed that our online algorithm yields a controller that satisfies the desired information constraint and yields a $\sqrt{T}$ expected regret. We validated our theoretical results using numerical simulations.

\bibliography{bibliography}

\begin{thebibliography}{51}
\providecommand{\natexlab}[1]{#1}
\providecommand{\url}[1]{\texttt{#1}}
\expandafter\ifx\csname urlstyle\endcsname\relax
  \providecommand{\doi}[1]{doi: #1}\else
  \providecommand{\doi}{doi: \begingroup \urlstyle{rm}\Url}\fi

\bibitem[Abbasi-Yadkori and Szepesv{\'a}ri(2011)]{abbasi2011regret}
Y.~Abbasi-Yadkori and C.~Szepesv{\'a}ri.
\newblock Regret bounds for the adaptive control of linear quadratic systems.
\newblock In \emph{Proc. Conference on Learning Theory}, pages 1--26, 2011.

\bibitem[Agarwal et~al.(2019)Agarwal, Bullins, Hazan, Kakade, and
  Singh]{agarwal2019online}
N.~Agarwal, B.~Bullins, E.~Hazan, S.~Kakade, and K.~Singh.
\newblock Online control with adversarial disturbances.
\newblock In \emph{Proc. International Conference on Machine Learning}, pages
  111--119, 2019.

\bibitem[Anava et~al.(2015)Anava, Hazan, and Mannor]{anava2015online}
O.~Anava, E.~Hazan, and S.~Mannor.
\newblock Online learning for adversaries with memory: price of past mistakes.
\newblock \emph{Advances in Neural Information Processing Systems}, 28, 2015.

\bibitem[Blondel and Tsitsiklis(2000)]{blondel2000survey}
V.~D. Blondel and J.~N. Tsitsiklis.
\newblock A survey of computational complexity results in systems and control.
\newblock \emph{Automatica}, 36\penalty0 (9):\penalty0 1249--1274, 2000.

\bibitem[Boyd and Vandenberghe(2004)]{boyd2004convex}
S.~P. Boyd and L.~Vandenberghe.
\newblock \emph{Convex optimization}.
\newblock Cambridge university press, 2004.

\bibitem[Bu et~al.(2019)Bu, Mesbahi, Fazel, and Mesbahi]{bu2019lqr}
J.~Bu, A.~Mesbahi, M.~Fazel, and M.~Mesbahi.
\newblock {LQR} through the lens of first order methods: Discrete-time case.
\newblock \emph{arXiv preprint arXiv:1907.08921}, 2019.

\bibitem[Bubeck(2011)]{bubeck2011introduction}
S.~Bubeck.
\newblock Introduction to online optimization.
\newblock \emph{Lecture Notes}, 2011.

\bibitem[Cassel and Koren(2021)]{cassel2021online}
A.~Cassel and T.~Koren.
\newblock Online policy gradient for model free learning of linear quadratic
  regulators with $\sqrt{T}$ regret.
\newblock \emph{arXiv preprint arXiv:2102.12608}, 2021.

\bibitem[Cassel et~al.(2020)Cassel, Cohen, and Koren]{cassel2020logarithmic}
A.~Cassel, A.~Cohen, and T.~Koren.
\newblock Logarithmic regret for learning linear quadratic regulators
  efficiently.
\newblock In \emph{Proc. International Conference on Machine Learning}, pages
  1328--1337, 2020.

\bibitem[Chen and Hazan(2021)]{chen2021black}
X.~Chen and E.~Hazan.
\newblock Black-box control for linear dynamical systems.
\newblock In \emph{Proc. Conference on Learning Theory}, pages 1114--1143,
  2021.

\bibitem[Cohen et~al.(2019)Cohen, Koren, and Mansour]{cohen2019learning}
A.~Cohen, T.~Koren, and Y.~Mansour.
\newblock Learning linear-quadratic regulators efficiently with only $\sqrt{T}$
  regret.
\newblock In \emph{Proc. International Conference on Machine Learning}, pages
  1300--1309, 2019.

\bibitem[Dean et~al.(2018)Dean, Mania, Matni, Recht, and Tu]{dean2018regret}
S.~Dean, H.~Mania, N.~Matni, B.~Recht, and S.~Tu.
\newblock Regret bounds for robust adaptive control of the linear quadratic
  regulator.
\newblock In \emph{Proc. International Conference on Neural Information
  Processing Systems}, pages 4192--4201, 2018.

\bibitem[Dean et~al.(2020)Dean, Mania, Matni, Recht, and Tu]{dean2020sample}
S.~Dean, H.~Mania, N.~Matni, B.~Recht, and S.~Tu.
\newblock On the sample complexity of the linear quadratic regulator.
\newblock \emph{Foundations of Computational Mathematics}, 20\penalty0
  (4):\penalty0 633--679, 2020.

\bibitem[Fattahi et~al.(2019)Fattahi, Matni, and Sojoudi]{fattahi2019learning}
S.~Fattahi, N.~Matni, and S.~Sojoudi.
\newblock Learning sparse dynamical systems from a single sample trajectory.
\newblock In \emph{Proc. IEEE Conference on Decision and Control}, pages
  2682--2689, 2019.

\bibitem[Fattahi et~al.(2020)Fattahi, Matni, and Sojoudi]{fattahi2020efficient}
S.~Fattahi, N.~Matni, and S.~Sojoudi.
\newblock Efficient learning of distributed linear-quadratic control policies.
\newblock \emph{SIAM Journal on Control and Optimization}, 58\penalty0
  (5):\penalty0 2927--2951, 2020.

\bibitem[Fazel et~al.(2018)Fazel, Ge, Kakade, and Mesbahi]{fazel2018global}
M.~Fazel, R.~Ge, S.~Kakade, and M.~Mesbahi.
\newblock Global convergence of policy gradient methods for the linear
  quadratic regulator.
\newblock In \emph{Proc. International Conference on Machine Learning}, pages
  1467--1476, 2018.

\bibitem[Feng and Lavaei(2019)]{feng2019exponential}
H.~Feng and J.~Lavaei.
\newblock On the exponential number of connected components for the feasible
  set of optimal decentralized control problems.
\newblock In \emph{American Control Conference}, pages 1430--1437, 2019.

\bibitem[Furieri et~al.(2020)Furieri, Zheng, and
  Kamgarpour]{furieri2020learning}
L.~Furieri, Y.~Zheng, and M.~Kamgarpour.
\newblock Learning the globally optimal distributed {LQ} regulator.
\newblock In \emph{Proc. Learning for Dynamics and Control Conference}, pages
  287--297, 2020.

\bibitem[Ghadimi and Lan(2013)]{ghadimi2013stochastic}
S.~Ghadimi and G.~Lan.
\newblock Stochastic first-and zeroth-order methods for nonconvex stochastic
  programming.
\newblock \emph{SIAM Journal on Optimization}, 23\penalty0 (4):\penalty0
  2341--2368, 2013.

\bibitem[Gravell et~al.(2020)Gravell, Esfahani, and
  Summers]{gravell2020learning}
B.~Gravell, P.~M. Esfahani, and T.~H. Summers.
\newblock Learning optimal controllers for linear systems with multiplicative
  noise via policy gradient.
\newblock \emph{IEEE Transactions on Automatic Control}, 2020.

\bibitem[Ho et~al.(1972)]{ho1972team}
Y.-C. Ho et~al.
\newblock Team decision theory and information structures in optimal control
  problems--part i.
\newblock \emph{IEEE Transactions on Automatic control}, 17\penalty0
  (1):\penalty0 15--22, 1972.

\bibitem[Horn and Johnson(2012)]{horn2012matrix}
R.~A. Horn and C.~R. Johnson.
\newblock \emph{Matrix analysis}.
\newblock Cambridge university press, 2012.

\bibitem[Hsu et~al.(2012)Hsu, Kakade, and Zhang]{hsu2012tail}
D.~Hsu, S.~M. Kakade, and T.~Zhang.
\newblock A tail inequality for quadratic forms of subgaussian random vectors.
\newblock \emph{Electronic Communications in Probability}, 17\penalty0
  (52):\penalty0 1--6, 2012.

\bibitem[Lale et~al.(2020)Lale, Azizzadenesheli, Hassibi, and
  Anandkumar]{lale2020logarithmic}
S.~Lale, K.~Azizzadenesheli, B.~Hassibi, and A.~Anandkumar.
\newblock Logarithmic regret bound in partially observable linear dynamical
  systems.
\newblock \emph{arXiv preprint arXiv:2003.11227}, 2020.

\bibitem[Lamperski and Doyle(2012)]{lamperski2012dynamic}
A.~Lamperski and J.~C. Doyle.
\newblock Dynamic programming solutions for decentralized state-feedback {LQG}
  problems with communication delays.
\newblock In \emph{Proc. American Control Conference}, pages 6322--6327, 2012.

\bibitem[Lamperski and Lessard(2015)]{lamperski2015optimal}
A.~Lamperski and L.~Lessard.
\newblock Optimal decentralized state-feedback control with sparsity and
  delays.
\newblock \emph{Automatica}, 58:\penalty0 143--151, 2015.

\bibitem[Langford et~al.(2009)Langford, Smola, and Zinkevich]{langford2009slow}
J.~Langford, A.~Smola, and M.~Zinkevich.
\newblock Slow learners are fast.
\newblock \emph{arXiv preprint arXiv:0911.0491}, 2009.

\bibitem[Li et~al.(2021{\natexlab{a}})Li, Das, Shamma, and Li]{li2021safe}
Y.~Li, S.~Das, J.~Shamma, and N.~Li.
\newblock Safe adaptive learning-based control for constrained linear quadratic
  regulators with regret guarantees.
\newblock \emph{arXiv preprint arXiv:2111.00411}, 2021{\natexlab{a}}.

\bibitem[Li et~al.(2021{\natexlab{b}})Li, Tang, Zhang, and
  Li]{li2019distributed}
Y.~Li, Y.~Tang, R.~Zhang, and N.~Li.
\newblock Distributed reinforcement learning for decentralized linear quadratic
  control: A derivative-free policy optimization approach.
\newblock \emph{IEEE Transactions on Automatic Control}, 2021{\natexlab{b}}.

\bibitem[Malik et~al.(2020)Malik, Pananjady, Bhatia, Khamaru, Bartlett, and
  Wainwright]{malik2020derivative}
D.~Malik, A.~Pananjady, K.~Bhatia, K.~Khamaru, P.~L. Bartlett, and M.~J.
  Wainwright.
\newblock Derivative-free methods for policy optimization: Guarantees for
  linear quadratic systems.
\newblock \emph{Journal of Machine Learning Research}, 21\penalty0
  (21):\penalty0 1--51, 2020.

\bibitem[Mania et~al.(2019)Mania, Tu, and Recht]{mania2019certainty}
H.~Mania, S.~Tu, and B.~Recht.
\newblock Certainty equivalence is efficient for linear quadratic control.
\newblock In \emph{Neural Information Processing Systems}, pages 10154--10164,
  2019.

\bibitem[Mohammadi et~al.(2021)Mohammadi, Zare, Soltanolkotabi, and
  Jovanovi{\'c}]{mohammadi2021convergence}
H.~Mohammadi, A.~Zare, M.~Soltanolkotabi, and M.~R. Jovanovi{\'c}.
\newblock Convergence and sample complexity of gradient methods for the
  model-free linear--quadratic regulator problem.
\newblock \emph{IEEE Transactions on Automatic Control}, 67\penalty0
  (5):\penalty0 2435--2450, 2021.

\bibitem[Ouyang et~al.(2017)Ouyang, Gagrani, and Jain]{ouyang2017learning}
Y.~Ouyang, M.~Gagrani, and R.~Jain.
\newblock Learning-based control of unknown linear systems with thompson
  sampling.
\newblock \emph{arXiv preprint arXiv:1709.04047}, 2017.

\bibitem[Quattoni et~al.(2009)Quattoni, Carreras, Collins, and
  Darrell]{quattoni2009efficient}
A.~Quattoni, X.~Carreras, M.~Collins, and T.~Darrell.
\newblock An efficient projection for $\ell_{\infty}$ regularization.
\newblock In \emph{Proc. International Conference on Machine Learning}, pages
  857--864, 2009.

\bibitem[Rotkowitz and Lall(2005)]{rotkowitz2005characterization}
M.~Rotkowitz and S.~Lall.
\newblock A characterization of convex problems in decentralized control.
\newblock \emph{IEEE Transactions on Automatic Control}, 50\penalty0
  (12):\penalty0 1984--1996, 2005.

\bibitem[Rotkowitz and Martins(2011)]{rotkowitz2011nearest}
M.~C. Rotkowitz and N.~C. Martins.
\newblock On the nearest quadratically invariant information constraint.
\newblock \emph{IEEE Transactions on Automatic Control}, 57\penalty0
  (5):\penalty0 1314--1319, 2011.

\bibitem[Sarkar et~al.(2021)Sarkar, Rakhlin, and Dahleh]{sarkar2021finite}
T.~Sarkar, A.~Rakhlin, and M.~A. Dahleh.
\newblock Finite time lti system identification.
\newblock \emph{Journal of Machine Learning Research}, 22\penalty0
  (26):\penalty0 1--61, 2021.

\bibitem[Shah and Parrilo(2013)]{shah2013cal}
P.~Shah and P.~A. Parrilo.
\newblock $\mathcal{H}_2$-optimal decentralized control over posets: A
  state-space solution for state-feedback.
\newblock \emph{IEEE Transactions on Automatic Control}, 58\penalty0
  (12):\penalty0 3084--3096, 2013.

\bibitem[Shin et~al.(2023)Shin, Lin, Qu, Wierman, and Anitescu]{shin2023near}
S.~Shin, Y.~Lin, G.~Qu, A.~Wierman, and M.~Anitescu.
\newblock Near-optimal distributed linear-quadratic regulator for networked
  systems.
\newblock \emph{SIAM Journal on Control and Optimization}, 61\penalty0
  (3):\penalty0 1113--1135, 2023.

\bibitem[Simchowitz and Foster(2020)]{simchowitz2020naive}
M.~Simchowitz and D.~Foster.
\newblock Naive exploration is optimal for online {LQR}.
\newblock In \emph{Proc. International Conference on Machine Learning}, pages
  8937--8948, 2020.

\bibitem[Simchowitz et~al.(2020)Simchowitz, Singh, and
  Hazan]{simchowitz2020improper}
M.~Simchowitz, K.~Singh, and E.~Hazan.
\newblock Improper learning for non-stochastic control.
\newblock In \emph{Proc. Conference on Learning Theory}, pages 3320--3436,
  2020.

\bibitem[Tu and Recht(2019)]{tu2019gap}
S.~Tu and B.~Recht.
\newblock The gap between model-based and model-free methods on the linear
  quadratic regulator: An asymptotic viewpoint.
\newblock In \emph{Proc. Conference on Learning Theory}, pages 3036--3083,
  2019.

\bibitem[Witsenhausen(1968)]{witsenhausen1968counterexample}
H.~S. Witsenhausen.
\newblock A counterexample in stochastic optimum control.
\newblock \emph{SIAM Journal on Control}, 6\penalty0 (1):\penalty0 131--147,
  1968.

\bibitem[Xin et~al.(2022)Xin, Ye, Chiu, and Sundaram]{xin2022identifying}
L.~Xin, L.~Ye, G.~Chiu, and S.~Sundaram.
\newblock Identifying the dynamics of a system by leveraging data from similar
  systems.
\newblock In \emph{Proc. American Control Conference}, pages 818--824, 2022.

\bibitem[Ye et~al.(2023)Ye, Zhu, and Gupta]{ye2021sample}
L.~Ye, H.~Zhu, and V.~Gupta.
\newblock On the sample complexity of decentralized linear quadratic regulator
  with partially nested information structure.
\newblock \emph{IEEE Transactions on Automatic Control}, 68\penalty0
  (8):\penalty0 4841--4856, 2023.

\bibitem[Yu et~al.(2023)Yu, Ho, and Wierman]{yu2022online}
J.~Yu, D.~Ho, and A.~Wierman.
\newblock Online adversarial stabilization of unknown networked systems.
\newblock \emph{Proc. of the ACM on Measurement and Analysis of Computing
  Systems}, 7\penalty0 (1):\penalty0 1--43, 2023.

\bibitem[Zames(1981)]{zames1981feedback}
G.~Zames.
\newblock Feedback and optimal sensitivity: Model reference transformations,
  multiplicative seminorms, and approximate inverses.
\newblock \emph{IEEE Transactions on automatic control}, 26\penalty0
  (2):\penalty0 301--320, 1981.

\bibitem[Zhang et~al.(2020)Zhang, Hu, and Basar]{zhang2020policy}
K.~Zhang, B.~Hu, and T.~Basar.
\newblock Policy optimization for $\mathcal{H}_2$ linear control with
  $\mathcal{H}_{\infty}$ robustness guarantee: Implicit regularization and
  global convergence.
\newblock In \emph{Proc. Learning for Dynamics and Control Conference}, pages
  179--190, 2020.

\bibitem[Zheng et~al.(2020)Zheng, Furieri, Papachristodoulou, Li, and
  Kamgarpour]{zheng2020equivalence}
Y.~Zheng, L.~Furieri, A.~Papachristodoulou, N.~Li, and M.~Kamgarpour.
\newblock On the equivalence of {Youla}, system-level, and input--output
  parameterizations.
\newblock \emph{IEEE Transactions on Automatic Control}, 66\penalty0
  (1):\penalty0 413--420, 2020.

\bibitem[Zheng et~al.(2021)Zheng, Furieri, Kamgarpour, and Li]{zheng2021sample}
Y.~Zheng, L.~Furieri, M.~Kamgarpour, and N.~Li.
\newblock Sample complexity of linear quadratic gaussian {(LQG)} control for
  output feedback systems.
\newblock In \emph{Proc. Learning for Dynamics and Control Conference}, pages
  559--570, 2021.

\bibitem[Ziemann and Sandberg(2022)]{ziemann2022regret}
I.~Ziemann and H.~Sandberg.
\newblock Regret lower bounds for learning linear quadratic gaussian systems.
\newblock \emph{arXiv preprint arXiv:2201.01680}, 2022.

\end{thebibliography}

\section*{Appendix}
\appendix
The Appendix is organized as follows. In Appendix~\ref{sec:OCO proofs}, we proof the results for the general OCO with memory and delayed feedback setting and characterize the regret of Algorithm~\ref{algorithm:OCO}. In Appendix~\ref{sec:proposition 1 proof}, we prove the first main restult summarized in Section~\ref{sec:main results} (i.e., Proposition~\ref{prop:decentralized online algorithm feasible}), which shows that Algorithm~\ref{algorithm:control design} can be implemented in a fully decentralized manner while satisfying the information constraints given by \eqref{eqn:info set} (during its decentralized online control phase). Appendices~\ref{sec:proofs for good event}-\ref{sec:proof of failure regret} contain all the proofs that lead to the final regret guarantee in Theorem~\ref{thm:regret upper bound}. Appendix~\ref{sec:technical lemma} contain some helper lemmas used in our proofs.

\section{Proofs Pertaining to OCO with Memory and Delayed Feedback}\label{sec:OCO proofs}
\subsection{Proof of Lemma~\ref{lemma:upper bound for conditional cost}}
First, let us consider any $x_{\star}\in\W$ and any $t\in\{k,\dots,T-1\}$. We know from the strong convexity of $f_{t;k}(\cdot)$ that (e.g., \cite{boyd2004convex})
\begin{align}
	2\big(f_{t;k}(x_t)-f_{t;k}(x_{\star})\big)\le2\nabla f_{t;k}(x_t)^{\top}(x_t-x_{\star})-\alpha\norm{x_{\star}-x_t}^2.\label{eqn:strong convextiy}
\end{align}
Since $x_{t+1}=\Pi_{\W}(x_t-\eta_tg_{t-\tau})$, we have
\begin{align}\nonumber
	\norm{x_{t+1}-x_t}^2&\le\norm{x_t-\eta_tg_{t-\tau}-x_{\star}}^2\\\nonumber
	&=\norm{x_t-x_{\star}}+\eta_t^2\norm{g_{t-\tau}}^2-2\eta_tg_{t-\tau}^{\top}(x_t-x_{t-\tau})-2\eta_tg_{t-\tau}^{\top}(x_{t-\tau}-x_{\star}),
\end{align}
where the inequality follows from the fact that the distance between two vectors cannot increase when projecting onto a convex set. It follows that
\begin{align}
	-2g_{t-\tau}^{\top}(x_{t-\tau}-x_{\star})\ge\frac{\norm{x_{t+1}-x_{\star}}^2-\norm{x_t-x_{\star}}^2}{\eta_t}-\eta_t\norm{g_{t-\tau}}^2+2g_{t-\tau}^{\top}(x_t-x_{t-\tau}).\label{eqn:inner product lower bound}
\end{align}
Moreover, noting that $g_{t-\tau}=\nabla f_{t-\tau;k}(x_{t-\tau})+\varepsilon^s_{t-\tau}+\varepsilon_{t-\tau}$, we have
\begin{align}\nonumber
	-g_{t-\tau}^{\top}(x_{t-\tau}-x_{\star})&=-(\nabla f_{t-\tau;k}(x_{t-\tau})+\varepsilon_{t-\tau}^s)^{\top}(x_{t-\tau}-x_{\star})-\varepsilon_{t-\tau}^{\top}(x_{t-\tau}-x_{\star})\\
	&\le-(\nabla f_{t-\tau;k}(x_{t-\tau})+\varepsilon_{t-\tau}^s)^{\top}(x_{t-\tau}-x_{\star})+\frac{\alpha\mu}{2}\norm{x_{t-\tau}-x_{\star}}^2+\frac{1}{2\alpha\mu}\norm{\varepsilon_{t-\tau}}^2,\label{eqn:inner product upper bound}
\end{align}
where the inequality follows from the fact that $ab\le\frac{a^2}{2\mu}+\frac{\mu}{2}b^2$ for any $a,b\in\R$ and any $\mu\in\R_{>0}$. To proceed, we let $0<\mu<1$. Combining \eqref{eqn:inner product lower bound}-\eqref{eqn:inner product upper bound}, and recalling Assumption~\ref{ass:gradient bound}, we obtain
\begin{align}\nonumber
	2\nabla f_{t-\tau;k}(x_{t-{\tau}})^{\top}(x_{t-\tau}-x_{\star})&\le\frac{\norm{x_t-x_{\star}}^2-\norm{x_{t+1}-x_{\star}}^2}{\eta_t}+\eta_tL_g^2+\alpha\mu\norm{x_{t-\tau}-x_{\star}}^2+\frac{1}{\alpha\mu}\norm{\varepsilon_{t-\tau}}^2\\
	&\qquad\qquad\qquad\qquad\qquad\qquad-2\varepsilon_{t-\tau}^{s\top}(x_{t-\tau}-x_{\star})-2g_{t-\tau}^{\top}(x_t-x_{t-\tau}),\label{eqn:gradient upper bound 2}
\end{align}
where we note that 
\begin{align}\nonumber
	-2g_{t-\tau}^{\top}(x_t-x_{t-\tau})&=2\norm{g_{t-\tau}}\norm{x_t-x_{t-1}+\cdots+x_{t-\tau+1}-x_{t-\tau}}\\\nonumber
	&\le 2\norm{g_{t-\tau}}\sum_{j=1}^{\tau}\eta_{t-j}\norm{g_{t-j-\tau}}\le 2L_g^2\sum_{j=1}^{\tau}\eta_{t-j},
\end{align}
where we let $g_{t-j-\tau}=0$ if $t-j-\tau<0$. Now, we see from \eqref{eqn:strong convextiy} and \eqref{eqn:gradient upper bound 2} that 
\begin{align}\nonumber
	\sum_{t=k}^{T-1}\big(f_{t:k}(x_t)-f_{t;k}(x_{\star})\big)&\le-\frac{\alpha}{2}\sum_{t=k}^{T-1}\norm{x_t-x_{\star}}^2+\sum_{t=k+\tau}^{T-1+\tau}\nabla f_{t-\tau;k}(x_{t-{\tau}})^{\top}(x_{t-\tau}-x_{\star})\\\nonumber
	&\le\frac{1}{2}\sum_{t=k+\tau}^{T-2}\big(\frac{1}{\eta_{t+1}}-\frac{1}{\eta}_t-\alpha(1-\mu)\big)\norm{x_{t+1}-x_{\star}}^2 + \frac{1}{2}\sum_{t=T-1}^{T-1+\tau}(\frac{1}{\eta_{t+1}}-\frac{1}{\eta_t})\norm{x_{t+1}-x_{\star}}^2\\
	+\sum_{t=k+\tau}^{T-1+\tau}&\big(\frac{L_g^2\eta_t}{2}+\frac{1}{2\alpha\mu}\norm{\varepsilon_{t-\tau}}^2-\varepsilon_{t-\tau}^{s\top}(x_{t-\tau}-x_{\star})+L_g^2\sum_{j=1}^{\tau}\eta_{t-j}\big) + \frac{\norm{x_{k+\tau}-x_{\star}}^2}{2\eta_{k+\tau}}.\label{eqn:conditional f upper bound 2}
\end{align}
Setting $\eta_t=\frac{3}{\alpha t}$ for all $t\in\{\tau,\dots,T+\tau\}$ and $\mu=\frac{1}{3}$, one can show that the following hold:
\begin{align*}
	&\frac{1}{\eta_{t+1}}-\frac{1}{\eta_t}-\alpha(1-\mu)=-\frac{\alpha}{3},\ \sum_{t=T-1}^{T-1+\tau}(\frac{1}{\eta_{t+1}}-\frac{1}{\eta_t})\norm{x_{t+1}-x_{\star}}^2\le\frac{\alpha(\tau+1)}{3}G^2,\\
	&\sum_{t=k+\tau}^{T-1+\tau}L_g^2\eta_t\le\sum_{t=k}^{T-1}L_g^2\eta_t\le\frac{3L_g^2}{\alpha}\log T,\ \frac{\norm{x_{k+\tau}-x_{\star}}^2}{\eta_{k+\tau}}\le\frac{\alpha G^2(k+\tau)}{3}\\
	&L_g^2\sum_{t=k+\tau}^{T-1+\tau}\sum_{j=1}^{\tau}\eta_{t-j}\le\frac{3L_g^2\tau}{\alpha}\sum_{t=k+\tau}^{T-1+\tau}\frac{1}{t-\tau}\le\frac{3L_g^2\tau}{\alpha}\log T.
\end{align*}
Meanwhile, we have
\begin{align}\nonumber
	-\frac{\alpha}{6}\sum_{t=k+\tau+1}^{T-1}\norm{x_t-x_{\star}}^2\le\frac{\alpha G^2(k+\tau+1)}{6}-\frac{\alpha}{6}\sum_{t=0}^{T-1}\norm{x_t-x_{\star}}^2.
\end{align}
It then follows from \eqref{eqn:conditional f upper bound 2} that \eqref{eqn:conditional f upper bound} holds.$\hfill\blacksquare$

\subsection{Proof of Lemma~\ref{lemma:true f upper bound}}
Using similar arguments to those for the proof of \cite[Lemma~K.2]{simchowitz2020improper} one can show that 
\begin{align}\nonumber
	\sum_{t=k}^{T-1}\big(f_t(x_t)-f_t(x_{\star})\big)-\sum_{t=k}^{T-1}\varepsilon_t^{s\top}(x_t-x_{\star})&\le\sum_{t=k}^{T-1}\big(f_{t;k}(x_t)-f_{t;k}(x_{\star})\big)\\\nonumber
	&\quad-\sum_{t=k}^{T-1}X_t(x_{\star})+\sum_{t=k}^{T-1}(2\beta G+4L_f)\norm{x_t-x_{t-k}}.
\end{align}
Moreover, we have
\begin{align}\nonumber
	\sum_{t=k}^{T-1}(2\beta G+4L_f)\norm{x_t-x_{t-k}}&=(2\beta G+4L_f)\sum_{t=k}^{T-1}\norm{x_t-x_{t-1}+\cdots+x_{t-k+1}-x_{t-k}}\\\nonumber
	&\le(2\beta G+4L_f)\sum_{t=k}^{T-1}\sum_{j=1}^k\eta_{t-j}\norm{g_{t-j-\tau}}\\\nonumber
	&\le(4\beta G+8L_f)L_g\sum_{t=k+1}^{T-1}\frac{3k}{\alpha(t-k)}\\\nonumber
	&\le(4\beta G+8L_f)\frac{3L_g k}{\alpha}\log T,
\end{align}
where we let $\eta_0=\frac{3}{\alpha}$, $\eta_{t^{\prime}}=\frac{3}{\alpha t^{\prime}}$ if $t^{\prime}\ge1$, and $g_{t-j-\tau}=0$ if $t-j-\tau<0$. Combining the above arguments with Lemma~\ref{lemma:upper bound for conditional cost}, we conclude that \eqref{eqn:true f upper bound} holds.$\hfill\blacksquare$

\subsection{Proof of Proposition~\ref{prop:true f with memory upper bound}}
First, let us consider any $\delta>0$ and any $x_{\star}\in\W$. Using similar arguments to those for \cite[Lemma~K.3]{simchowitz2020improper}, one can show that the following holds with probability at least $1-\delta$:
\begin{align}\nonumber
	-\sum_{t=k}^{T-1}X_t(x_{\star})-\frac{\alpha}{12}\sum_{t=0}^{T-1}\norm{x_t-x_{\star}}^2\le\CO(1)\frac{kL_f^2}{\alpha}\log\Big(\frac{k\big(1+\log_+(\alpha TG^2)\big)}{\delta}\Big),
\end{align}
where $\log_+(x)=\log(\max\{1,x\})$. It then follows from Lemma~\ref{lemma:true f upper bound} that the following holds with probability at least $1-\delta$:
\begin{align}\nonumber
	\sum_{t=k}^{T-1}\big(f_{t}(x_t)-f_{t}(x_{\star})\big)&\le-\frac{\alpha}{12}\sum_{t=0}^{T-1}\norm{x_t-x_{\star}}^2+\CO(1)\bigg(\alpha G^2(k+\tau)+\frac{L_g}{\alpha}\big(L_g\tau+(\beta G+L_f)k\big)\log T\\\nonumber
	&\qquad\qquad\qquad\qquad+\sum_{t=k}^{T-1}\frac{1}{\alpha}\norm{\varepsilon_{t}}^2+\frac{kL_f^2}{\alpha}\log\Big(\frac{T\big(1+\log_+(\alpha G^2)\big)}{\delta}\Big)\bigg),
\end{align}
where we use the facts that $k\le T$ and $\log_+(\alpha TG^2)\le T\log_+(\alpha G^2)$. Next, following \cite[Claim~K.4]{simchowitz2020improper} and \cite[Claim~K.5]{simchowitz2020improper}, one can further show that the following hold with probability at least $1-\delta$:
\begin{align}\nonumber
	\sum_{t=k}^{T-1}\big(f_{t}(x_t)-f_{t}(x_{\star})\big)&\le-\frac{\alpha}{12}\sum_{t=0}^{T-1}\norm{x_t-x_{\star}}^2+\CO(1)\bigg(\sum_{t=k}^{T-1}\frac{1}{\alpha}\norm{\varepsilon_{t}}^2+\frac{kdL_f^2}{\alpha}\log\Big(\frac{T\big(1+\log_+(\alpha G^2)\big)}{\delta}\Big)\\
	&\qquad\qquad\qquad\qquad+\alpha G^2(k+\tau)+\frac{L_g}{\alpha}\big(L_g\tau+(\beta G+L_f)k\big)\log T\bigg),\ \forall x_{\star}\in\W.\label{eqn:true f upper bound 2}
\end{align}
Now, following the proof of \cite[Theorem~4.6]{agarwal2019online}, we have that for any $x_{\star}\in\W$,
\begin{align}
	\sum_{t=k}^{T-1}\big(F_t(x_t,\dots,x_{t-h})-f_t(x_{\star})\big)=\sum_{t=k}^{T-1}\big(f_t(x_t)-f_t(x_{\star})\big)+\sum_{t=k}^{T-1}\big(F_t(x_t,\dots,x_{t-h})-f_t(x_t)\big).\label{eqn:true f with memory decomposition}
\end{align}
The second term on the right-hand side of the above equation can be bounded as 
\begin{align}\nonumber
	\sum_{t=k}^{T-1}\big(F_t(x_t,\dots,x_{t-h})-f_t(x_t)\big)&\le L_c\sum_{t=k}^{T-1}\sum_{i=1}^h\norm{x_t-x_{t-i}}\\\nonumber
	&\le L_c\sum_{t=k}^{T-1}\sum_{i=1}^h\sum_{j=1}^{i}\eta_{t-j}\norm{g_{t-j-\tau}}\\\nonumber
	&\le L_cL_g\sum_{t=k}^{T-1}\sum_{i=1}^h\frac{3i}{\alpha(t-i)}\\
	&\le L_cL_gh^2\sum_{t=k}^{T-1}\frac{3}{\alpha(t-h)}\le\frac{3L_cL_gh^2}{\alpha}\log T,\label{eqn:true f with memory upper bound}
\end{align}
where we let $\eta_{t^{\prime}}=\frac{3}{\alpha t^{\prime}}$ for all $t^{\prime}\ge1$, and $g_{t-j-\tau}=0$ if $t-j-\tau<0$. Combining \eqref{eqn:true f upper bound 2}-\eqref{eqn:true f with memory upper bound} together completes the proof of the proposition.$\hfill\blacksquare$

\section{Proof of Proposition~\ref{prop:decentralized online algorithm feasible}}\label{sec:proposition 1 proof}
Let us consider any $i\in\V$. To prove part~(a), we use an induction on $t=N+D_{\max},\dots,T+D_{\max}-1$. For the base case $t=N$, we see directly from line~2 in Algorithm~\ref{algorithm:control design} and Eq.~\eqref{eqn:initial K_i} that $\K_{i,1}$ (resp., $\K_{i,2}$) satisfies Eq.~\eqref{eqn:K_i 1} (resp., Eq.~\eqref{eqn:K_i 2}) at the beginning of iteration $t=N$ of the for loop in lines~7-18 of the algorithm. For the inductive step, consider any $t\in\{N+D_{\max},\dots,T+D_{\max}-1\}$ and suppose $\K_{i,1}$ (resp., $\K_{i,2}$) satisfies Eq.~\eqref{eqn:K_i 1} (resp., Eq.~\eqref{eqn:K_i 2}) at the beginning of iteration $t$ of the for loop in lines~7-18 of the algorithm.
    
First, let us consider $K_{i,1}$, and consider any $s\in\CL(\T_i)$ with $j\in\V$ and $s_{0,j}=s$ in the for loop in lines~8-11 of the algorithm, where $\CL(\T_i)$ is defined in Eq.~\eqref{eqn:leaf nodes in T_i}. We aim to show that $\hat{w}_{t-D_{ij}-1,j}$ can be determined, using Eq.~\eqref{eqn:est w_i t} and the current $\K_{i,1}$ and $\K_{i,2}$. We see from Eq.~\eqref{eqn:est w_i t} that in order to determine $\hat{w}_{t-D_{ij}-1,j}$, we need to know $x_{t-D_{ij},j}^{\tt alg}$, and $x_{t-D_{ij}-1,j_1}^{\tt alg}$ and $u^{\tt alg}_{t-D_{ij}-1,j_1}$ for all $j_1\in\CN_j$. Note that $D_{ij_1}\le D_{ij}+1$ for all $j_1\in\CN_j$, which implies that $t-D_{\max}-1\le t-D_{ij}-1\le t-D_{ij_1}$. Thus, we have that $x_{t-D_{ij},j}^{\tt alg}\in\tilde{\I}_{t,i}$, and $x_{t-D_{ij}-1,j_1}^{\tt alg}\in\tilde{\I}_{t,i}$ for all $j_1\in\CN_j$, where $\tilde{\I}_{t,i}$ is defined in Eq.~\eqref{eqn:info set used}. We then focus on showing that $u^{\tt alg}_{t-D_{ij}-1,j_1}$ for all $j_1\in\CN_j$ can be determined based on the current $\K_{i,1}$ and $\K_{i,2}$. Considering any $j_1\in\CN_j$, one can show via line~12 of the algorithm and Eq.~\eqref{eqn:alg control input} that 
\begin{align}\nonumber
u_{t-D_{ij}-1,j_1}^{\tt alg}=\sum_{r\ni j_1}I_{\{j_1\},r}\sum_{k=1}^hM_{t-D_{ij}-1,r}^{[k]}\hat{\eta}_{t-D_{ij}-1-k,r},
\end{align}
where $\hat{\eta}_{t-D_{ij}-1-k,r}=\begin{bmatrix}\hat{w}^{\top}_{t-D_{ij}-1-k-l_{vs},j_v}\end{bmatrix}_{v\in\CL_r}^{\top}$. Now, let us consider any $r\ni j_1$. Recalling the definition of $\CP(\U,\CH)$ in Eq.~\eqref{eqn:def of info graph}, and noting that $j_v\in v$, $j_1\in r$ and $v \rightsquigarrow r$, one can show that 
\begin{equation}
\label{eqn:relation between l_vr and D_j1jv}
D_{j_1j_v}\le l_{vr}\le D_{\max}.
\end{equation}
To proceed, we split our arguments into two cases: $D_{ij_v}\le D_{ij}$ and $D_{ij_v}\ge D_{ij}+1$. Supposing $D_{ij_v}\le D_{ij}$, we have
\begin{equation}
\begin{split}
\label{eqn:distance relation first case}
&t-2D_{\max}-2h\le t-D_{ij}-1-k-l_{vs}\\
&t-D_{ij_v}-2\ge t-D_{ij}-1-k-l_{vs},
\end{split}
\end{equation}
for all $k\in[h]$. From the induction hypothesis, we know that $\hat{w}_{k^{\prime},j_v}\in\K_{i,1}$ for all $k^{\prime}\in\{t-2D_{\max}-2h,\dots,t-D_{ij_v}-2\}$, which implies via \eqref{eqn:distance relation first case} that $\hat{\eta}_{t-D_{ij}-1-k,r}$ for all $k\in[h]$ can be determined based on the current $\K_{i,1}$. Moreover, noting that $j_1\in r$ and $j_1\in\CN_j$, one can show that $r\in\T_i$. Since $M_{t-D_{ij}-1,r}\in\K_{i,2}$ for all $r\in\T_i$, we then have from the above arguments that $u_{t-D_{ij}-1,j_1}^{\tt alg}$ can be determined based on the current $\K_{i,1}$ and $\K_{i,2}$. Next, suppose that $D_{ij_v}\ge D_{ij}+1$. Noting that $j_v\rightsquigarrow j_1\rightsquigarrow i$, i.e., there exists a directed path from node $j_v$ to node $i$ that goes through $j_1$ in $\G(\V,\A)$, we have
\begin{align}
D_{ij_v}\le D_{ij_1}+D_{j_1j_v}\le D_{ij}+D_{jj_1}+D_{j_1j_v},\label{eqn:relation between D_ijv and Dij}
\end{align}
where $D_{jj_1}\in\{0,1\}$. Combining \eqref{eqn:relation between l_vr and D_j1jv} and \eqref{eqn:relation between D_ijv and Dij}, we have
\begin{equation}
\begin{split}
\label{eqn:distance relation second case}
&t-2D_{\max}-2h\le t-D_{ij}-1-k-l_{vs}\\
&t-D_{ij_v}-1\ge t-D_{ij}-1-k-l_{vs},
\end{split}
\end{equation}
for all $k\in[h]$. Recalling from Remark~\ref{remark:ordering of the elements in L(T_i)} that we have assumed without loss of generality that the for loop in lines~8-11 of Algorithm~\ref{algorithm:control design} iterate over the elements in $\CL(\T_i)$ according to a certain order of the elements in $\CL(\T_i)$. We then have from the fact $D_{ij_v}\ge D_{ij}+1$ that when considering $s\in\CL(\T_i)$ (with $j\in\V$ and $s_{0,j}=s$) in the for loop in lines~8-11 of Algorithm~\ref{algorithm:control design}, the element $s_{0,j_v}\in\CL(\T_i)$ has already been considered by the algorithm, i.e., $\hat{w}_{k^{\prime},j_v}$ for all $k^{\prime}\in\{t-2D_{\max}-2h,\dots,t-D_{ij_v}-1\}$ are in the current $\K_{i,1}$. It then follows from \eqref{eqn:distance relation second case} similar arguments to those above that $u_{t-D_{ij}-1,j_1}^{\tt alg}$ can be determined based on the current $\K_{i,1}$ and $\K_{i,2}$. This completes the inductive step of the proof that $\K_{i,1}$ satisfies Eq.~\eqref{eqn:K_i 1} at the beginning of any iteration $t\in\{N+D_{\max},\dots,T+D_{\max}-1\}$ of the for loop in lines~7-18 of Algorithm~\ref{algorithm:control design}.
    
Next, let us consider $\K_{i,2}$, and consider any $s\in\T_i$ in the for loop in lines~15-17 of the algorithm. Note that $j\rightsquigarrow i$ for all $j\in s$. To simplify the notations in this proof, we denote 
\begin{align}\nonumber
f_{t_D}(M_t)=f_{t_D}(M_{t_D}|\Phi,\hat{w}_{t-D_D-1}),
\end{align}
Moreover, we see from Definition~\ref{def:counterfactual cost} that 
\begin{equation*}
f_{t_D}(M_t) = c(x_{t_D},u_{t_D}),
\end{equation*}
where, for notational simplicity, we denote
\begin{align}\nonumber
&x_{t_D} = \sum_{k=t_D-h}^{t_D-1}\hat{A}^{t_D-(k+1)}(\hat{w}_k+\hat{B}u_k),\label{eqn:exp for x_tD}\\
&u_k = u_k(M_{t_D}|\hat{\Phi},\hat{w}_{0:k-1})=\sum_{s\in\U}\sum_{k^{\prime}=1}^{h}I_{\V,s}M_{t_D,s}^{[k^{\prime}]}\hat{\eta}_{k-k^{\prime},s},
\end{align}
for all $k\in\{t_D-h,\dots,t_D-1\}$, where $\hat{\eta}_{k-k^{\prime},r}=\begin{bmatrix}\hat{w}^{\top}_{k-k^{\prime}-l_{vs},j_v}\end{bmatrix}_{v\in\CL_r}^{\top}$. Similarly, one can also show that
\begin{align}
u_{k,j}=\sum_{t\ni j}I_{\{j\},r}\sum_{k^{\prime}}^{h}M_{t_D,r}^{[k^{\prime}]}\hat{\eta}_{k-k^{\prime},r},\label{eqn:input in prop 1 u_k,i}
\end{align}
for all $k\in\{t_D-h,\dots,t_D-1\}$ and all $j\in\V$. We will then show that $M_{t+1,s}$, i.e., $\frac{\partial f_{t_D}(M_{t_D})}{\partial M_{{t_D},s}}$, can be determined based on the current $\K_{i,1}$ and $\K_{i,2}$. In other words, viewing $f_{t_D}(M_{t_D})$ as a function of $M_{t_D}$, we will show that for any term in $f_{t_D}(M_{t_D})$ that contains $M_{t_D,s}$, the coefficients of that term can be determined based on the current $\K_{i,1}$ and $\K_{i,2}$. To proceed, we notice from Assumption~\ref{ass:structure of Q and R} that
\begin{align}\nonumber
c(x_{t_D},u_{t_D}) = \sum_{l\in[\psi]}\big(x_{t_D,\V_l}^{\top}Q_{\V_l\V_l}x_{t_D,\V_l}+u_{t_D,\V_l}^{\top}R_{\V_l\V_l}u_{t_D,\V_l}\big),\label{eqn:decompose the cost}
\end{align}
where $x_{t_D,\V_l}=\begin{bmatrix}x_{t_D,j}^{\top}\end{bmatrix}_{j\in\V_l}^{\top}$ and $u_{t_D,\V_l}=\begin{bmatrix}u_{t_D,j}^{\top}\end{bmatrix}_{j\in\V_l}^{\top}$. Let us consider any $l\in[\psi]$. One can show using Eq.~\eqref{eqn:input in prop 1 u_k,i} that the term $u_{t_D,\V_l}^{\top}R_{\V_l\V_l}u_{t_D,\V_l}$ contains $M_{t_D,s}$ if and only if there exists $j\in s\in\T_i$ such that $j\in\V_l$. Supposing that $l\in[\psi]$ satisfies this condition, one can show that either $i\to j$ or $j\to i$ holds in $\G(\V,\A)$, which implies via Assumption~\ref{ass:structure of Q and R} that $i\in\V_l$. Viewing the term $u_{t_D,\V_{l}}^{\top}R_{\V_{l}\V_{l}}u_{t_D,\V_{l}}$ as a function of $M_{t_D,s}$, we then aim to show that the coefficients of this term can be determined based on the current $\K_{i,1}$ and $\K_{i,2}$. To this end, considering any $j_1\in\V_{l}$, we see from Eq.~\eqref{eqn:input in prop 1 u_k,i} that 
\begin{align}\nonumber
u_{t_D,j_1}=\sum_{r\ni j_1}I_{\{j_1\},r}\sum_{k=1}^hM_{t_D,r}^{[k]}\hat{\eta}_{t_D-k,r},
\end{align}
where $\hat{\eta}_{t_D-k,r}=\begin{bmatrix}\hat{w}^{\top}_{t_D-k-l_{vs},j_v}\end{bmatrix}_{v\in\CL_r}^{\top}$. Since $i\in\V_l$ as we argued above, it holds that $j\rightsquigarrow i$ in $\G(\V,\A)$, which implies that $r\in\T_i$ for all $r\ni j_1$. It then follows from Eq.~\eqref{eqn:K_i 2} that $M_{t,r}\in\K_{i,2}$ for all $r\ni j_1$. Moreover, noting that that $l_{vs}\le D_{\max}$, we then have from Eq.~\eqref{eqn:K_i 1} that $\hat{\eta}_{t_D-k,r}$ for all $k\in[h]$ and all $r\ni j_1$ can be determined based on the current $\K_{i,1}$. It follows that the coefficients in the term $u_{t_D,\V_l}^{\top}R_{\V_l\V_l}u_{t_D,\V_l}$ can be determined based on the current $\K_{i,1}$ and $\K_{i,2}$.
    
Now, let us consider any $l\in[\psi]$ with the corresponding term $x_{t_D,\V_l}^{\top}Q_{\V_l\V_l}x_{t_D,\V_l}$. Recall from Algorithm~\ref{algorithm:least squares} that $\hat{A}$ and $\hat{B}$ satisfy that for any $i,j\in\V$, $\hat{A}_{ij}=0$ and $\hat{B}_{ij}=0$ if and only if $D_{ij}=\infty$, where $D_{ij}<\infty$ if and only if there is a directed path from node $j$ to node $i$ in $\G(\V,\A)$. One can also show that for any $i,j\in\V$ and any $k\in\{t_D-h,\dots,t_D-1\}$, $(\hat{A}^{t_D-(k+1)})_{ij}=0$ and $(\hat{A}^{t_D-(k+1)}\hat{B})_{ij}=0$ if and only if $D_{ij}=\infty$. It then follows from Eq.~\eqref{eqn:exp for x_tD} and Assumption~\ref{ass:structure of Q and R} that 
\begin{equation}
\label{eqn:subvector of x_tD}
x_{t_D,\V_l} = \sum_{k=t_D-h}^{t_D-1}\big((\hat{A}^{t_D-(k+1)})_{\V_l\V_l}\hat{w}_{k,\V_l}+(\hat{A}^{t_D-(k+1)}\hat{B})_{ij}u_{k,\V_l}\big).
\end{equation}
One can show via Eq.~\eqref{eqn:subvector of x_tD} that for any $l\in[\psi]$, the term $x_{t_D,\V_l}^{\top}Q_{\V_l\V_l}x_{t_D,\V_l}$ contains $M_{t_D,s}$, if and only if there exists $j\in s$ such that $j\in\V_l$. Supposing that $l\in[\psi]$ satisfies this condition, we have from similar arguments to those above that $i\in\V_l$. Viewing the term $x_{t_D,\V_{l}}^{\top}Q_{\V_{l}\V_{l}}x_{t_D,\V_{l}}$ as a function of $M_{t_D,s}$, we aim to show that the coefficients of this term can be determined based on the current $\K_{i,1}$ and $\K_{i,2}$. To this end, let us consider any $j_1\in\V_{l}$ and any $k\in\{t_D-h,\dots,t_D-1\}$. Since $i\in\V_l$ as we argued above, it holds that $j_1\rightsquigarrow i$ in $\G(\V,\A)$, which also implies that $r\in\T_i$ for all $r\ni j_1$. It follows from Eq.~\eqref{eqn:K_i 1} that $\hat{w}_{k,j_1}$ can be determined based on the current $\K_{i,1}$. Moreover, recalling that $u_{k,j_1}=\sum_{r\ni j_1}I_{\{j_1\},r}\sum_{k^{\prime}=1}^hM_{t_D,r}^{[k^{\prime}]}\hat{\eta}_{k-k^{\prime},r}$, one can show via similar arguments to those above that $M_{t_D,r}\in\K_{i,2}$ for all $r\ni j_1$, and $\hat{\eta}_{k-k^{\prime},r}$ for all $k^{\prime}\in[h]$ and all $r\ni j_1$ can be determined based on the current $\K_{i,1}$. It now follows from Eq.~\eqref{eqn:subvector of x_tD} that the coefficients of the term $x_{t_D,\V_l}^{\top}Q_{\V_l\V_l}x_{t_D,\V_l}$ can be determined based on the current $\K_{i,1}$ and $\K_{i,2}$. This completes the inductive step of the proof that $\K_{i,2}$ satisfies Eq.~\eqref{eqn:K_i 1} at the beginning of any iteration $t\in\{N+D_{\max},\dots,T+D_{\max}-1\}$ of the for loop in lines~7-18 of Algorithm~\ref{algorithm:control design}.
    
We then prove part~(b). It suffices to show that for any $t\in\{N+D_{\max},\dots,T+D_{\max}-1\}$ and any $i\in\V$, $u_{t,i}^{\tt alg}$ can be determined based on the current $\K_{i,1}$ and $\K_{i,2}$ after line~9 and before line~16 in iteration $t$ of the algorithm, where we note that the current $\K_{i,1}$ is given by 
\begin{align}\nonumber
\K_{i,1} &= \big\{\hat{w}_{k,j}:k\in\{t-2D_{\max}-2h,\dots,t-D_{ij}-1\},s\in\CL(\T_i),j\in\V,s_{0,j}=s\big\}.\label{eqn:K_i 1 updated}
\end{align}
Note again from Eq.~\eqref{eqn:alg control input} that
\begin{align}\nonumber
u_{t,i}^{\tt alg}=\sum_{r\ni i}I_{\{i\},r}\sum_{k=1}^hM_{t,r}^{[k]}\hat{\eta}_{t-k,r},
\end{align}
where $\hat{\eta}_{t-k,r}=\begin{bmatrix}\hat{w}^{\top}_{t-k-l_{vs},j_v}\end{bmatrix}_{v\in\CL_r}^{\top}$. Considering $r\ni i$ and any $j_v$ with $w_{j_v}\to v$ and $v\in\CL_r$, we have
\begin{equation*}
\label{eqn:distance relation u alg}
\begin{split}
&t-2D_{\max}-2h\le t-k-l_{vr},\\
&t-D_{ij_v}-1\ge t-k-l_{vr},
\end{split}
\end{equation*}
for all $k\in[h]$. Since $i\in r\in\T_i$, we know that $j_v\rightsquigarrow i$. It then follows from Eq.~\eqref{eqn:K_i 1} that $\hat{\eta}_{t-k,r}$ for all $k\in[h]$ and all $r\ni i$ can be determined based on the current $\K_{i,1}$. Also noting from Eq.~\eqref{eqn:K_i 2} that $M_{t,r}\in\K_{i,2}$ for all $r\in\T_i$, we conclude that $u_{t,i}^{\tt alg}$ can be determined based on the current $\K_{i,1}$ and $\K_{i,2}$.$\hfill\blacksquare$
    
\section{Proofs Omitted in Section~\ref{sec:good events}}\label{sec:proofs for good event}
\subsection{Proof of Lemma~\ref{lemma:estimation error of A_hat and B_hat}}
Considering the directed graph $\G(\V,\A)$, let $\V=\cup_{i\in[\psi]}\V_i$, where $\V_i\subseteq\V$ denotes the set of nodes in the $i$th strongly connected component of $\G(\V,\A)$. For any $i\in[\psi]$, let $\tilde{\V}_i$ be the set of nodes in $\G(\V,\A)$ that can reach any node in $\V_i$ via a directed path in $\G(\V,\A)$, and let $\tilde{\V}^c_i=\V\setminus\tilde{\V}_i$. Denoting $\Delta A=A-\hat{A}$ and $\Delta B=B-\hat{B}$, we see from line~6 in Algorithm~\ref{algorithm:least squares} that $\begin{bmatrix}\Delta A_{\V_i\tilde{\V}_i} & \Delta B_{\V_i\tilde{\V}_i}\end{bmatrix}$ is a submatrix of $\tilde{\Delta}_N$. One can then show that $\nm[big]{\begin{bmatrix}\Delta A_{\V_i\tilde{\V}_i} & \Delta B_{\V_i\tilde{V}_i}\end{bmatrix}}\le\norm{\tilde{\Delta}_N}$. Moreover, we see from line~6 in Algorithm~\ref{algorithm:least squares} that $\hat{A}_{\V_i\tilde{\V}^c_i}=0$ and $\hat{B}_{\V_i\tilde{\V}^c_i}=0$. Also noting from Eq.~\eqref{eqn:overall system} that $A_{\V_i\tilde{\V}^c_i}=0$ and $B_{\V_i\tilde{\V}^c_i}=0$, one can then show that $\nm[\big]{\begin{bmatrix}\Delta A_{\V_i} & \Delta B_{\V_i}\end{bmatrix}}=\nm[\big]{\begin{bmatrix}\Delta A_{\V_i\tilde{\V}_i} & \Delta B_{\V_i\tilde{\V}_i}\end{bmatrix}}$. Now, assuming without loss of generality that the nodes in $\V$ are ordered such that $\hat{\Phi}_N=\begin{bmatrix}\hat{\Phi}_{\V_i}^{\top}\end{bmatrix}^{\top}_{i\in[\psi]}$, we obtain from the above arguments that $\norm{\Delta_N}\le\sqrt{\sum_{i\in[\psi]}\nm[\big]{\begin{bmatrix}\Delta A_{\V_i} & \Delta B_{\V_i}\end{bmatrix}}^2}\le\sqrt{\psi}\varepsilon$.$\hfill\blacksquare$

\subsection{Proof of Proposition~\ref{prop:upper bound on est error}}
First, following the arguments for \cite[Proposition~1]{ye2021sample}, one can show that under the choice of $N$ given by Eq.~\eqref{eqn:choice of N},
\begin{align}\nonumber
		\norm{\tilde{\Delta}_N}^2\le\frac{160}{N\underline{\sigma}^2}\bigg(2n\sigma_w^2(n+m)\big(\log(N+z^2_b/\lambda)T\big)+\lambda n\Gamma^2\bigg).
\end{align}
We then have from Lemma~\ref{lemma:estimation error of A_hat and B_hat} that 
\begin{align}\nonumber
		\norm{\Delta_N}^2&\le\frac{160\psi}{N\underline{\sigma}^2}\bigg(2n\sigma_w^2(n+m)\big(\log(N+z^2_b/\lambda)T\big)+\lambda n\Gamma^2\bigg)\\
		&\le\frac{480\psi}{N\underline{\sigma}^2}n\sigma_w^2(n+m)\big(\log(T+z_b^2/\lambda)T\big)\Gamma^2,\label{eqn:upper bound on Delta N}
\end{align}
where $\Delta_N$ is defined in \eqref{eqn:def of Delta}.
Plugging Eq.~\eqref{eqn:choice of N} into the above inequality and noting that $\norm{\hat{A}-A}\le\norm{\Delta(N)}^2$ and $\norm{\hat{B}-B}\le\norm{\Delta(N)}^2$, one can show that $\norm{\hat{A}-A}\le\bar{\varepsilon}$ and $\norm{\hat{B}-B}\le\bar{\varepsilon}$. $\hfill\blacksquare$

\subsection{Proof of Lemma~\ref{lemma:norm bounds related to u_t star and u_t M}}
Consider any $t\in\Z_{\ge0}$. We see from Eq.~\eqref{eqn:DFC of u_star} that
\begin{align*}
\norm{u_t^{\star}}&\le\sum_{s\in\U}\sum_{k=1}^{t}\nm[\big]{M_{\star,s}^{[k]}}\nm[\big]{\eta_{t-k,s}}\\
&\le q\kappa p\Gamma^{2D_{\max}+1}R_w\sum_{k=0}^{t-1}\gamma^k\\
&\le\frac{pq\kappa\Gamma^{2D_{\max}+1}R_w}{1-\gamma}.
\end{align*}
Next, we know from \cite[Lemma~14]{lamperski2012dynamic} that $x_t^{\star}$ satisfies $x_t^{\star}=\sum_{s\in\U}I_{\V,s}\zeta_{t,s}$ for all $t\in\Z_{\ge0}$, where $\zeta_{t,s}$ is given by Eq.~\eqref{eqn:dynamics of zeta}. Using similar arguments to those above, one can then show via Eqs.~\eqref{eqn:dynamics of zeta self loop unroll}-\eqref{eqn:dynamics of zeta no self loop} that
\begin{align*}
\norm{x_t^{\star}}&\le\sum_{s\in\U}\norm{\zeta_{t,s}}\le\frac{pq\kappa\Gamma^{2D_{\max}}R_w}{1-\gamma}.
\end{align*}
    
Similarly, we obtain from Eq.~\eqref{eqn:u_t M} that 
\begin{equation*}
\norm{u_t^M}\le\frac{pq\kappa\Gamma^{2D_{\max}+1}R_w}{1-\gamma}.
\end{equation*}
From Eq.~\eqref{eqn:overall system}, we know that 
\begin{equation*}
x_t^{M} = \sum_{k=0}^{t-1}A^{t-(k+1)}(Bu_k^M+w_k),
\end{equation*}
which implies that 
\begin{align}\nonumber
\norm{x_t^M}&\le\sum_{k=0}^{t-1}\nm[\big]{A^{t-(k+1)}}\nm[\big]{Bu_k^M+w_k}\\\nonumber
&\le\Big(\Gamma\frac{pq\kappa\Gamma^{2D_{\max}+1}R_w}{1-\gamma}+R_w\Big)\kappa\sum_{k=0}^{t-1}\gamma^k\\\nonumber
&\le\frac{2pq\kappa^2\Gamma^{2D_{\max}+2}R_w}{(1-\gamma)^2}.
\end{align}
Note that $M\in\D$ trivially holds. Now, we have from the above arguments and Lemma~\ref{lemma:lipschitz c} that 
\begin{align*}
c(x_t^M,u_t^M)-c(x_t^{\star},u_t^{\star})&\le2\bar{\sigma}\Big(\frac{pq\kappa\Gamma^{2D_{\max}+1}R_w}{1-\gamma}+\frac{2pq\kappa^2\Gamma^{2D_{\max}+2}R_w}{(1-\gamma)^2}\Big)\big(\norm{x^M_t-x^{\star}_t}+\norm{u^M_t-u^{\star}_t}\big)\\
&\le6\bar{\sigma}\frac{pq\kappa^2\Gamma^{2D_{\max}+2}R_w}{(1-\gamma)^2}\big(\norm{x^M_t-x^{\star}_t}+\norm{u^M_t-u^{\star}_t}\big).
\end{align*}
Moreover, we have
\begin{align*}
\norm{u_t^M-u_t^{\star}}&=\nm[\Big]{\sum_{s\in\U}\sum_{k=h+1}^{t}I_{\V,s}M_{\star,s}^{[k]}\eta_{t-k,s}}\\
&\le R_w\kappa p\Gamma^{2D_{\max}+1}\sum_{k=h}^{t-1}\gamma^k\\
&\le\frac{\kappa p\Gamma^{2D_{\max}+1}\gamma^hR_w}{1-\gamma}.
\end{align*}
From Eq.~\eqref{eqn:overall system}, we have
\begin{align*}
\norm{x_t^M-x_t^{\star}}&=\nm[\Big]{\sum_{k=0}^{t-1}A^{t-(k+1)}B(u_k^M-u_k^{\star})}\\
&\le\frac{\kappa p\Gamma^{2D_{\max}+2}\gamma^hR_w}{1-\gamma}\kappa\sum_{k=0}^{t-1}\gamma^k\\
&\le\frac{\kappa^2 p\Gamma^{2D_{\max}+2}\gamma^hR_w}{(1-\gamma)^2}.
\end{align*}
Combining the above inequalities completes the proof of the lemma.$\hfill\blacksquare$

\subsection{Proof of Lemma~\ref{lemma:bounds on norm of alg and pred input/state}}
First, we see from Assumption~\ref{ass:stable A} and \eqref{eqn:events} that 
\begin{align*}
\norm{x_N^{\tt alg}}&=\nm[\Big]{\sum_{k=0}^{N-1}A^k(Bu_{N-k-1}^{\tt alg}+w_{N-k-1})}\\
&\le\big(\sigma_u\sqrt{5m\log 4NT}+\sigma_w\sqrt{10n\log2T}\big)\Gamma\kappa\sum_{k=0}^{N-1}\gamma^k\\
&\le\overline{\sigma}\sqrt{10}(\sqrt{n}+\sqrt{m})\frac{\Gamma\kappa}{1-\gamma}\le\overline{\sigma}\frac{\sqrt{20(m+n)}\Gamma\kappa}{1-\gamma}.
\end{align*}
Next, we recall from Eqs.~\eqref{eqn:overall system} and \eqref{eqn:est w_i t} that for any $t\in\{N,\dots,T-1\}$, $w_t = x_{t+1}^{\tt alg} - Ax_t^{\tt alg} - Bu_t^{\tt alg}$ and $\hat{w}_t = x_{t+1}^{\tt alg} - \hat{A}x_t^{\tt alg} - Bu_t^{\tt alg}$, respectively. It follows that for any $t\in\{N,\dots,T-1\}$,
\begin{align}\nonumber
w_t - \hat{w}_t &= (\hat{A}-A)x_t^{\tt alg} + (\hat{B}-B)u_t^{\tt alg}\\\nonumber
&=(\hat{A}-A)\Big(A^{t-N}x_N^{\tt alg}+\sum_{k=0}^{t-N-1}A^k(Bu_{t-k-1}^{\tt alg}+w_{t-k-1})\Big)+(\hat{B}-B)u_t^{\tt alg}.
\end{align}

Let us denote $R_{u_t}=\max_{k\in\{N,\dots,t\}}\norm{u_k^{\tt alg}}$ for all $t\in\{N,\dots,T-1\}$, where we note that $R_{u_{N}}=0$ (from line~4 in Algorithm~\ref{algorithm:control design}) and that $R_{u_t}\ge R_{u_{t-1}}$. Now, consider any $t\in\{N,\dots, T-1\}$. Recalling from Proposition~\ref{prop:upper bound on est error} that $\norm{\hat{A}-A}\le\bar{\varepsilon}$ and $\norm{\hat{B}-B}\le\bar{\varepsilon}$, we have
\begin{align}\nonumber
\norm{w_t-\hat{w}_t}&\le\bar{\varepsilon}\norm{A^{t-N}}\norm{x_N^{\tt alg}}+\bar{\varepsilon}\nm[\Big]{\sum_{k=0}^{t-N-1}A^k\big(Bu^{\tt alg}_{t-k-1}+w_{t-k-1}\big)}+\bar{\varepsilon}\norm{u_t^{\tt alg}}\\\nonumber
&\le\bar{\varepsilon}\kappa\gamma^{t-N}\frac{2\overline{\sigma}\sqrt{10n}\Gamma\kappa}{1-\gamma}+\bar{\varepsilon}(\Gamma R_{u_{t-1}}+R_w)\kappa\sum_{k=0}^{t-1}\gamma^k + \bar{\varepsilon} R_{u_t}\\
&\le\Big(\frac{\Gamma\kappa}{1-\gamma}+1\Big)\bar{\varepsilon} R_{u_t}+\frac{\bar{\varepsilon}\kappa}{1-\gamma}(\kappa\overline{\sigma}\sqrt{20(m+n)}\Gamma+R_w),\label{eqn:est error of w_t hat}
\end{align}
which implies that 
\begin{align}
\norm{\hat{w}_t}\le\norm{\hat{w}_t-w_t}+\norm{w_t}\le \Big(\frac{\Gamma\kappa}{1-\gamma}+1\Big)\bar{\varepsilon} R_{u_t} + R_w+\frac{\bar{\varepsilon}\kappa}{1-\gamma}(\kappa\overline{\sigma}\sqrt{20(m+n)}\Gamma+R_w).\label{eqn:upper bound on norm of w hat}
\end{align}
Recall from Eq.~\eqref{eqn:alg control input} that $u_t^{\tt alg}=\sum_{s\in\U}\sum_{k=1}^{h}I_{\V,s}M_{t,s}^{[k]}\hat{\eta}_{t-k,s}$, where $\hat{\eta}_{t-k,s}=\begin{bmatrix}\hat{w}_{t-k-l_{vs},j_v}^{\top}\end{bmatrix}^{\top}_{v\in\CL_s}$, with $w_{j_v}\to v$, for all $k\in[h]$ and all $s\in\U$. Noting that $|\CL_s|\le p$, we have 
\begin{align}\nonumber
\norm{\hat{\eta}_{t-k,s}}&=\nm[\Big]{\begin{bmatrix}\hat{w}_{t-k-l_{vs},j_v}^{\top}\end{bmatrix}^{\top}_{v\in\CL_s}}\le\sum_{v\in\CL_s}\norm{\hat{w}_{t-k-l_{vs},j_v}}\le p\norm{\hat{w}_{t-k-l_{vs}}}\\\nonumber
&\le p \Big(\frac{\Gamma\kappa}{1-\gamma}+1\Big)\bar{\varepsilon} R_{u_t} + pR_w+\frac{p\bar{\varepsilon}\kappa}{1-\gamma}(\kappa\overline{\sigma}\sqrt{20(m+n)}\Gamma+R_w).
\end{align}
Noting from the definition of Algorithm~\ref{algorithm:control design} that $M_t\in\D$ for all $t\in\{N,\dots,T-1\}$, we then have 
\begin{align}\nonumber
R_{u_{t}}&\le2\Big(pq\Big(\frac{\Gamma\kappa}{1-\gamma}+1\Big)\bar{\varepsilon} R_{u_{t-1}} + pqR_w+\frac{pq\bar{\varepsilon}\kappa}{1-\gamma}(\kappa\overline{\sigma}\sqrt{20(m+n)}\Gamma+R_w)\Big)h\sqrt{n}\kappa p\Gamma^{2D_{\max}+1}\\
&\le\frac{1}{4}R_{u_{t-1}}+3pqR_wh\sqrt{n}\kappa p\Gamma^{2D_{\max}},\label{eqn:recursion for R_ut alg}
\end{align}
where the first inequality follows from the definition of $\D$ given by Eq.~\eqref{eqn:the class of DFCs}, and the second inequality follows from the choice of $\bar{\varepsilon}$ in Eq.~\eqref{eqn:epsilon_N}. Unrolling \eqref{eqn:recursion for R_ut alg} with $R_{u_N}=0$ yields
\begin{align}\nonumber
R_{u_{T-1}}&\le\frac{1}{4^{T-1-N}}R_{u_N}+3qR_wh\sqrt{n}\kappa p^2\Gamma^{2D_{\max}}\sum_{k=0}^{T-2-N}\frac{1}{4^k}\\\nonumber
&\le4qR_wh\sqrt{n}\kappa p^2\Gamma^{2D_{\max}},
\end{align}
which implies that $\norm{u_t^{\tt alg}}\le R_u $ for all $t\in\{N,\dots,T-1\}$. Recalling \eqref{eqn:est error of w_t hat}-\eqref{eqn:upper bound on norm of w hat} yields $\norm{\hat{w}_t-w_t}\le\Delta R_w\bar{\varepsilon}$ and $\norm{\hat{w}_t}\le R_{\hat{w}}$, respectively, for all $t\in\{N,\dots,T-1\}$.

To proceed, we note that  
\begin{align}\nonumber
\norm{x_t^{\tt alg}}&=\nm[\Big]{A^{t-N}x_N^{\tt alg}+\sum_{k=0}^{t-N-1}A^k(Bu_{t-k-1}^{\tt alg}+w_{t-k-1})}\\\nonumber
&\le\kappa\gamma^{t-N}\frac{\overline{\sigma}\sqrt{20(m+n)}\Gamma\kappa}{1-\gamma}+(\Gamma R_{u}+R_w)\kappa\sum_{k=0}^{t-1}\gamma^k\\\nonumber
 &\le\frac{\overline{\sigma}\sqrt{20(m+n)}\Gamma\kappa^2}{1-\gamma}+\frac{(\Gamma R_{u}+R_w)\kappa}{1-\gamma}=R_x,
\end{align}
for all $t\in\{N,\dots,T-1\}$. Similarly, we have from Definition~\ref{def:true prediction cost} that
\begin{align}\nonumber
\norm{x^{\tt pred}_t(M_{t-h:t-1})}&=\nm[\Big]{\sum_{k=t-h}^{t-1}A^{t-(k+1)}\big(Bu_k(M_k|\hat{\Phi},\hat{w}_{0:k-1})+w_k\big)}\\\nonumber
&\le(\Gamma R_u+R_w)\kappa\sum_{k=0}^{h-1}\gamma^k\\\nonumber
&\le\frac{(\Gamma R_{u}+R_w)\kappa}{1-\gamma}\le R_x,
\end{align}
for all $t\in\{N,\dots,T-1\}$, where we use the fact that $u_k(M_k|\hat{\Phi},\hat{w}_{0:k-1})=u_t^{\tt alg}$ if $k\ge N$, and $u_k(M_k|\hat{\Phi},\hat{w}_{0:k-1})=0$ if $k<N$.$\hfill\blacksquare$

\section{Proofs for Upper Bounds on $R_0$, $R_1$, $R_4$ and $R_5$ }\label{sec:proofs for R_0 R_1 R_4 R_5}
\subsection{Proof of Lemma~\ref{lemma:upper bound on R_0}}
Suppose the event $\CE$ defined in Eq.~\eqref{eqn:good event} holds. Considering any $t\in\{N,\dots,N_0-1\}$, we know from Lemma~\ref{lemma:bounds on norm of alg and pred input/state} that $\norm{u_t^{\tt alg}}\le R_u$ and $\norm{x_t^{\tt alg}}\le R_x$ for all $t\in\{N,\dots,N_0-1\}$. Next, consider any $t\in\{0,\dots,N-1\}$. We have from \eqref{eqn:events} that $\norm{u_t^{\tt alg}}\le\sigma_u\sqrt{5m\log 4NT}$. Moreover, under Assumption~\ref{ass:stable A}, we see that 
	\begin{align}\nonumber
		\norm{x_t^{\tt alg}}&=\nm[\Big]{\sum_{k=0}^{t-1}A^k(Bu_{t-k-1}^{\tt alg}+w_{t-k-1})}\\\nonumber
		&\le\big(\sigma_u\sqrt{5m\log 4NT}+\sigma_w\sqrt{10n\log2T}\big)\Gamma\kappa\sum_{k=0}^{t-1}\gamma^k\\\nonumber
		&\le\overline{\sigma}\sqrt{10}(\sqrt{n}+\sqrt{m})\frac{\Gamma\kappa}{1-\gamma}\le\overline{\sigma}\frac{\sqrt{20(m+n)}\Gamma\kappa}{1-\gamma}\le R_x.
	\end{align}
	It then follows that
	\begin{align}\nonumber
		c(x_t^{\tt alg},u_t^{\tt alg})&=x_t^{{\tt alg}\top}Qx_t^{\tt alg}+u_t^{{\tt alg}\top}Ru_t^{\tt alg}\\\nonumber
		&\le \sigma_1(Q)R_x^2+\sigma_1(R)\frac{R_u^2\sigma_u^2}{\sigma_w^2},
	\end{align}
for all $t\in\{0,\dots,N_0-1\}$.$\hfill\blacksquare$

\subsection{Proof of Lemma~\ref{lemma:upper bound on R_1}}
Suppose the event $\CE$ holds, and consider any $t\in\{N_0,\dots,T-1\}$. Noting from Lemma~\ref{lemma:bounds on norm of alg and pred input/state} that $\norm{x_t^{\tt alg}}\le R_x$ and $\norm{x_t^{\tt pred}}\le R_x$, we have from Definition~\ref{def:true prediction cost} that 
	\begin{align}\nonumber
		c(x_t^{\tt alg},u_t^{\tt alg})-F_t^{\tt pred}(M_{t-h:t})&=x_t^{{\tt alg}\top}Qx_t^{\tt alg}-x^{{\tt pred}\top}_tQx^{{\tt pred}}\\\nonumber
		&\le2R_x\sigma_1(Q)\norm{x_t^{\tt alg}-x_t^{\tt pred}}\\\nonumber
		&\le2R_x\sigma_1(Q)\nm[\Big]{\sum_{k=h}^{t-1}A^k\big(Bu_{t-k-1}^{\tt alg}+w_{t-k-1}\big)}\\\nonumber
		&\le2R_x\sigma_1(Q)(\Gamma R_u+R_w)\frac{\kappa\gamma^h}{1-\gamma},
	\end{align}
where the first inequality follows from similar arguments to those for Lemma~\ref{lemma:lipschitz c}.$\hfill\blacksquare$

\subsection{Proof of Lemma~\ref{lemma:upper bound on R_4}}
Consider any $M\in\D_0$. First, recalling Eq.~\eqref{eqn:u_t M} and Definition~\ref{def:counterfactual cost}, we see that $u_t^M=u_t(M|w_{0:t-1})=\sum_{s\in\U}\sum_{k=1}^{h}I_{\V,s}M_{s}^{[k]}\eta_{t-k,s}$, for all $t\ge0$, where
\begin{align*}
	\norm{u_t^{M}}&\le\sum_{s\in\U}\sum_{k=1}^{h}\nm[\big]{M_{s}^{[k]}}\nm[\big]{\eta_{t-k,s}}\\
	&\le \sqrt{n}hq\kappa p\Gamma^{2D_{\max}+1}R_w.
\end{align*}
Considering any $t\in\{N_0,\dots,T-1\}$, we then see that  
\begin{align}\nonumber
	f_t(M|\Phi,w_{0:t-1})-c(x^M_t,u^M_t)=x_t(M|\Phi,w_{0:t-1})^{\top}Qx_t(M|\Phi,w_{0:t-1})-x_t^{M\top}Qx_t^M,
\end{align}
where
\begin{align}\nonumber
	&x_t(M|\Phi,w_{0:t-1})=\sum_{k=t-h}^{t-1}A^{t-(k+1)}\big(w_k+Bu_k(M|w_{0:k-1})\big)\\\nonumber
	&x_t^M=\sum_{k=0}^{t-1}A^{t-(k+1)}(w_k+Bu_k^M).
\end{align}
Moreover, we have
\begin{align}\nonumber
	\nm{x_t(M|\Phi,w_{0:t-1})}&=\nm[\Big]{\sum_{k=t-h}^{t-1}A^{t-(k+1)}\big(w_k+Bu_k(M|w_{0:k-1})\big)}\\\nonumber
	&\le\Big(1+\sqrt{n}hq\kappa p\Gamma^{2D_{\max}+2}\Big)R_w\nm[\big]{\sum_{k=0}^{h-1}A^k}\\\nonumber
	&\le2\sqrt{n}hq\kappa p\Gamma^{2D_{\max}+2}\frac{R_w\kappa}{1-\gamma}.
\end{align}
Similarly, we have
\begin{align}\nonumber
	\nm{x_t^M}\le2\sqrt{n}hq\kappa p\Gamma^{2D_{\max}+2}\frac{R_w\kappa}{1-\gamma}.
\end{align}
It then follows from similar arguments to those for Lemma~\ref{lemma:lipschitz c} that 
\begin{align}\nonumber
	f_t(M|\Phi,w_{0:t-1})-c(x^M_t,u^M_t)\le4\sigma_1(Q)\sqrt{n}hq\kappa p\Gamma^{2D_{\max}+2}\frac{R_w\kappa}{1-\gamma}\nm[\big]{x_t(M|\Phi,w_{0:t-1})-x_t^M},
\end{align}
where
\begin{align}\nonumber
	\nm[\big]{x_t(M|\Phi,w_{0:t-1})-x_t^M}&=\nm[\Big]{\sum_{k=0}^{t-h-1}A^{t-(k+1)}(w_k+Bu_k^M)}\\\nonumber
	&\le2q\sqrt{n}h\kappa p\Gamma^{2D_{\max}+2}R_w\sum_{k=h}^{t-1}\norm{A^k}\\\nonumber
	&\le2q\sqrt{n}h\kappa p\Gamma^{2D_{\max}+2}\frac{R_w\kappa\gamma^h}{1-\gamma},
\end{align}
which implies that 
\begin{align}\nonumber
	f_t(M|\Phi,w_{0:t-1})-c(x^M_t,u^M_t)\le8\sigma_1(Q)q^2nh^2\kappa^2 p^2\Gamma^{4D_{\max}+4}\frac{R_w^2\kappa^2\gamma^h}{(1-\gamma)^2}.
\end{align}
Thus, denoting $M_{\star\star}\in\arg\inf_{M\in\D_0}\sum_{t=N_0}^{T-1}c(x^M_t,u^M_t)$, we have from the above arguments that
\begin{align}\nonumber
	R_4&\le\sum_{t=N_0}^{T-1}\big(f_t(M_{\star\star}|\Phi,w_{0:t-1})-c(x_t^{M_{\star\star}},u_t^{M_{\star\star}})\big)\\\nonumber
	&\le8\sigma_1(Q)q^2nh^2\kappa^2 p^2\Gamma^{4D_{\max}+4}\frac{R_w^2\kappa^2\gamma^h}{(1-\gamma)^2}T.
\end{align}
$\hfill\blacksquare$

\subsection{Proof of Lemma~\ref{lemma:upper bound on R_5}}
Let $M_s^{[k]}=M_{\star,s}^{[k]}$ for all $s\in\U$ and all $k\in[h/4]$, where $M_{\star,s}^{[k]}$ is given by Eq.~\eqref{eqn:M_star self loop}. We know from Lemma~\ref{lemma:norm bounds related to u_t star and u_t M} that $M\in\D_0$, where $M=[M_s^{[k]}]_{s\in\U,k\in[h/4]}$. It follows that 
\begin{align}\nonumber
	R_5\le \sum_{t=N_0}^{T-1}c(x^M_t,u_t^M)\le\sum_{t=0}^{T-1}c(x_t^M,u_t^M).
\end{align}
Using similar arguments to those for Lemma~\ref{lemma:norm bounds related to u_t star and u_t M}, one can show that on the event $\CE$,
\begin{align}\nonumber
	c(x_t^M,u_t^M)-c(x_t^{\star},u_t^{\star})\le\frac{12\bar{\sigma}p^2q\kappa^4\Gamma^{4D_{\max}+2}R_w^2}{(1-\gamma)^4}\gamma^{\frac{h}{4}},
\end{align}
for all $t\in\Z_{\ge0}$, where $c(x_t^{\star},u_t^{\star})=x_t^{\star\top}Qx_t^{\star}+u_t^{\star\top}Ru_t^{\star}$, and $x_t^{\star}$ is the state corresponding to the optimal control policy $u_t^{\star}$ given by Eq.~\eqref{eqn:DFC of u_star}. 

Now, considering the internal state $\zeta_{t,s}$ given in Eq.~\eqref{eqn:dynamics of zeta} for all $t\in\Z_{\ge0}$ and all $s\in\U$, we recall from \cite[Lemma~14]{lamperski2015optimal} that for any $t\in\Z_{\ge0}$, the following hold: (a) $\E[\zeta_{t,s}]=0$, for all $s\in\U$; (b) $x^{\star}_t=\sum_{s\in\U}I_{\V,s}\zeta_{t,s}$; (c) $\zeta_{t,s_1}$ and $\zeta_{t,s_2}$ are independent for all $s_1,s_2\in\U$ with $s_1\neq s_2$. Based on these results, we can relate $\sum_{t=0}^{T-1}c(x_t^{\star},u_t^{\star})$ to $J_{\star}$. Specifically, considering any $t\in\{0,\dots,T-2\}$, we first recall from Eq.~\eqref{eqn:dynamics of zeta} that $\zeta_{t+1,s} = \sum_{r\rightarrow s}(A_{sr}+B_{sr}K_r)\zeta_{t,r} + \sum_{w_i\rightarrow s}I_{s,\{i\}}w_{t,i}$ for all $s\in\U$, where $K_r$ is given by Eq.~\eqref{eqn:set of DARES K}. For any $s\in\U$, we have
\begin{align}\nonumber
	\E\big[\zeta_{t+1,s}^{\top}P_s\zeta_{t+1,s}\big]&=\E\Big[\big(\sum_{r\to s}(A_{sr}+B_{sr}K_r)\zeta_{t,r}+\sum_{w_i\rightarrow s}I_{s,\{i\}}w_{t,i}\big)^{\top}P_s\big(\sum_{r\to s}(A_{sr}+B_{sr}K_r)\zeta_{t,r}+\sum_{w_i\rightarrow s}I_{s,\{i\}}w_{t,i}\big)\Big]\\\nonumber
	&=\E\Big[\sum_{r\to s}\zeta_{t,r}^{\top}(A_{sr}+B_{sr}K_r)^{\top}P_s(A_{sr}+B_{sr}K_r)\zeta_{t,r}\Big]+\E\Big[\sum_{w_i\to s}w_{t,i}^{\top}I_{\{i\},s}P_sI_{s,\{i\}}w_{t,i}\Big],
\end{align}
where $P_s$ is given by Eq.~\eqref{eqn:set of DARES P}. Moreover, for any $s\in\U$, we have from Eq.~\eqref{eqn:set of DARES P} that 
\begin{align}\nonumber
	\E\Big[\sum_{r\to s}\zeta_{t,r}^{\top}P_r\zeta_{t,r}\Big]&=\E\Big[\sum_{r\to s}\zeta_{t,r}^{\top}(Q_{rr}+K_r^{\top}R_{rr}K_r)\zeta_{t,r}\Big]+\E\Big[\sum_{r\to s}\zeta_{t,r}^{\top}(A_{sr}+B_{sr}K_r)^{\top}P_s(A_{sr}+B_{sr}K_r)\zeta_{t,r}\Big]\\\nonumber
	&=\E\Big[\sum_{r\to s}\zeta_{t,r}^{\top}(Q_{rr}+K_r^{\top}R_{rr}K_r)\zeta_{t,r}\Big]+\E\big[\zeta_{t+1,s}^{\top}P_s\zeta_{t+1,s}\big]-\E\Big[\sum_{w_i\to s}w_{t,i}^{\top}I_{\{i\},s}P_sI_{s,\{i\}}w_{t,i}\Big].
\end{align}
Summing over all $s\in\U$ yields the following:
\begin{align}\nonumber
	&\E\Big[\sum_{s\in\U}\sum_{r\to s}\zeta_{t,r}^{\top}P_r\zeta_{t,r}\Big]-\E\Big[\sum_{s\in\U}\zeta_{t+1,s}^{\top}P_s\zeta_{t+1,s}\Big]\\\nonumber
	=&\E\Big[\sum_{s\in\U}\sum_{r\to s}\zeta_{t,r}^{\top}(Q_{rr}+K_r^{\top}R_{rr}K_r)\zeta_{t,r}\Big]-\E\Big[\sum_{s\in\U}\sum_{w_i\to s}w_{t,i}^{\top}I_{\{i\},s}P_sI_{s,\{i\}}w_{t,i}\Big],
\end{align}
which implies that 
\begin{align}
	&\E\Big[\sum_{s\in\U}\zeta_{t,s}^{\top}P_s\zeta_{t,s}\Big]-\E\Big[\sum_{s\in\U}\zeta_{t+1,s}^{\top}P_s\zeta_{t+1,s}\Big]=\E\Big[\sum_{s\in\U}\zeta_{t,s}^{\top}(Q_{ss}+K_s^{\top}R_{ss}K_s)\zeta_{t,s}\Big]-J_{\star}\label{eqn:relate finite cost to infinite cost}
\end{align}
where we use the definition of $J_{\star}$ given by Eq.~\eqref{eqn:opt J}, and the fact that the information graph $\CP(\U,\CH)$ has a tree structure as shown in \cite{lamperski2015optimal,ye2021sample}. Noting from Lemma~\ref{lemma:opt solution} that $x^{\star}_t=\sum_{s\in\U}I_{\V,s}\zeta_{t,s}$ and $u_t^{\star}=\sum_{s\in\U}I_{\V,s}K_s\zeta_{t,s}$ for all $t\in\{0,\dots,T\}$, one can then show via Eq.~\eqref{eqn:relate finite cost to infinite cost} that 
\begin{align}\nonumber
	\E\Big[\sum_{s\in\U}\zeta_{t,s}^{\top}P_s\zeta_{t,s}\Big]-\E\Big[\sum_{s\in\U}\zeta_{t+1,s}^{\top}P_s\zeta_{t+1,s}\Big]=\E\big[x_t^{\star\top}Qx_t^{\star}+u_t^{\star\top}Ru_t^{\star}\big]-J_{\star},
\end{align}
for all $t\in\{0,\dots,T-1\}$. It then follows that 
\begin{align}\nonumber
	\E\Big[\sum_{t=0}^{T-1}(x_t^{\star\top}Qx_t^{\star}+u_t^{\star\top}Ru_t^{\star})\Big]&=TJ_{\star}+\sum_{t=0}^{T-1}\E\Big[\sum_{s\in\U}(\zeta_{t,s}^{\top}P_s\zeta_{t,s}-\zeta_{t+1,s}^{\top}P_s\zeta_{t+1,s})\Big]\\\nonumber
	&\le TJ_{\star}+\E\Big[\sum_{s\in\U}\zeta_{0,s}^{\top}P_s\zeta_{0,s}\Big]= TJ_{\star},
\end{align}
where we use the fact that $\zeta_{0,s}=0$ as we assumed previously. 

Putting the above arguments together, we have
\begin{align}\nonumber
	\E\big[\ind\{\CE\}R_5\big]&\le\E\Big[\ind\{\CE\}\sum_{t=0}^{T-1}\big(c(x_t^M,u_t^M)-c(x_t^{\star},u_t^{\star})\big)\Big]+\E\Big[\ind\{\CE\}\sum_{t=0}^{T-1}c(x_t^{\star},u_t^{\star})\Big]\\\nonumber
	&\le\frac{12\bar{\sigma}p^2q\kappa^4\Gamma^{4D_{\max}+2}R_w^2}{(1-\gamma)^4}\gamma^{\frac{h}{4}}T+\E\Big[\sum_{t=0}^{T-1}c(x_t^{\star},u_t^{\star})\Big]\\\nonumber
	&\le\frac{12\bar{\sigma}p^2q\kappa^4\Gamma^{4D_{\max}+2}R_w^2}{(1-\gamma)^4}\gamma^{\frac{h}{4}}T+TJ_{\star}.
\end{align}
$\hfill\blacksquare$

\section{Proofs Pertaining to Upper Bounding $R_2$}\label{sec:proofs R_2}
\subsection{Proof of Lemma~\ref{lemma:lipschitz of f pred}}
Suppose the event $\CE$ holds, and consider any $t\in\{N,\dots,T-1\}$. For any $M\in\D$, where $\D$ is given by Eq.~\eqref{eqn:the class of DFCs}, we have from Lemma~\ref{lemma:bounds on norm of alg and pred input/state} that $\norm{u_t(M|\hat{w}_{0:t-1})}\le R_u$ and $\norm{x_t^{\tt pred}(M)}\le R_x$. Now, consider any $M,\tilde{M}\in\D$, we have from Lemma~\ref{lemma:lipschitz c} and Definition~\ref{def:true prediction cost} that
	\begin{align}\nonumber
		\big|f_t^{\tt pred}(M)-f_t^{\tt pred}(\tilde{M})\big|&\le2(R_u+R_x)\max\{\sigma_1(Q),\sigma_1(R)\}\\
		&\qquad\qquad\qquad\times\big(\nm[\big]{x_t^{\tt pred}(M-\tilde{M})}+\nm[\big]{u_t(M-\tilde{M}|\hat{w}_{0:t-1})}\big).\label{eqn:lipschitz of f pred}
	\end{align}
	Recalling Definition~\ref{def:counterfactual cost}, we have
	\begin{align}\nonumber
		\nm[\big]{u_t(M-\tilde{M}|\hat{w}_{0:t-1})}=\nm[\Big]{\sum_{s\in\U}\sum_{k=1}^h I_{\V,s}(M_s^{[k]}-\tilde{M}_s^{[k]})\hat{\eta}_{t-k,s}},
	\end{align}
	where we note from Lemma~\ref{lemma:bounds on norm of alg and pred input/state} that $\norm{\hat{\eta}_{t-k,s}}=\nm[\big]{\begin{bmatrix}\hat{w}_{t-k-l_{vs},j_v}^{\top}\end{bmatrix}^{\top}_{v\in\CL_s}}\le R_{\hat{w}}$. One can write the right-hand side of the above equation into a matrix form and obtain
	\begin{align}\nonumber
		\nm[\big]{u_t(M-\tilde{M}|\hat{w}_{0:t-1})}=\nm{\Delta M\eta_t},
	\end{align}
	where $\Delta M=\begin{bmatrix}I_{\V,s}(M_s^{[k]}-\tilde{M}_s^{[k]})\end{bmatrix}_{s\in\U,k\in[h]}$ and $\eta_t=\begin{bmatrix}\hat{\eta}_{t-k,s}^{\top}\end{bmatrix}_{s\in\U,k\in[h]}^{\top}$. One can also show that $\norm{\eta_t}\le\sqrt{qh}R_{\hat{w}}$. It then follows that
	\begin{align}\nonumber
		\nm[\big]{u_t(M-\tilde{M}|\hat{w}_{0:t-1})}&\le\sqrt{qh}R_{\hat{w}}\norm{\Delta M}\\
		&\le\sqrt{qh}R_{\hat{w}}\norm{{\tt Vec}(M-\tilde{M})},\label{eqn:u_t lipschitz}
	\end{align}
	where ${\tt Vec}(M-\tilde{M})$ denotes the vector representation of $(M-\tilde{M})$. Now, we have from Definition~\ref{def:true prediction cost} that
	\begin{align}\nonumber
		\nm[\big]{x_t^{\tt pred}(M-\tilde{M}|\hat{\Phi},\hat{w}_{0:t-1})}&=\nm[\Big]{\sum_{k=t-h}^{t-1}A^{t-(k+1)}B\big(u_k(M-\tilde{M}|\hat{w}_{0;K-1})\big)}\\\nonumber
		&\le\Gamma\sqrt{qh}R_{\hat{w}}\kappa\sum_{k=0}^{h-1}\gamma^k\norm{{\tt Vec}(M-\tilde{M})}\\
		&\le\frac{\Gamma\sqrt{qh}R_{\hat{w}}\kappa}{1-\gamma}\norm{{\tt Vec}(M-\tilde{M})}.\label{eqn:x_pred lipschitz}
	\end{align}
	Going back to \eqref{eqn:lipschitz of f pred}, we conclude that $f_t^{\tt alg}(\cdot)$ is $L_f^{\prime}$-Lipschitz. Now, noting the definition of $F_t^{\tt alg}(\cdot)$ from Definition~\ref{def:true prediction cost}, one can use similar arguments to those above and show that $F_t^{\tt alg}(\cdot)$ is $L_c^{\prime}$-Lipschitz.$\hfill\blacksquare$

\subsection{Proof of Lemma~\ref{lemma:upper bound on norn of hessian of f_pred}}
Suppose the event $\CE$ holds, and consider any $t\in\{N,\dots,T-1\}$ and any $M\in\D$. For notational simplicity in this proof, we denote $\tilde{x}_t(M)=x_t^{\tt pred}(M)$ and $\tilde{u}_t(M)=u_t(M|\hat{w}_{0:t-1})$, where $x_t^{\tt pred}(M)$ and $u_t(M|\hat{w}_{0:t-1})$ are given by Definitions~\ref{def:true prediction cost} and \ref{def:counterfactual cost}, respectively. Moreover, we assume that $M=(M_{s}^{[k]})_{k\in[h],s\in\U}$ is already written in its vector form ${\tt Vec}(M)$, i.e., we let $M={\tt Vec}(M)$. We then have
\begin{align}\nonumber
f_t^{\tt pred}(M)=\tilde{x}_t(M)^{\top}Q\tilde{x}_t(M)+\tilde{u}_t(M)^{\top}R\tilde{u}_t(M).
\end{align}
Taking the derivative, we obtain
\begin{align}\nonumber
\frac{\partial f_t^{\tt pred}(M)}{\partial M}=2\tilde{x}_t(M)^{\top}Q\frac{\partial\tilde{x}_t(M)}{\partial M}+2\tilde{u}_t(M)^{\top}R\frac{\partial\tilde{u}_t(M)}{\partial M}.
\end{align}
Further taking the derivative, we obtain
\begin{align}
\nabla^2f_t^{\tt pred}(M)=2\frac{\partial\tilde{x}_t(M)}{\partial M}^{\top}Q\frac{\partial\tilde{x}_t(M)}{\partial M}+2\frac{\partial\tilde{u}_t(M)}{\partial M}^{\top}R\frac{\partial\tilde{u}_t(M)}{\partial M},\label{eqn:f_pred second derivative}
\end{align}
where we use the fact that $\tilde{x}_t(M)$ and $\tilde{u}_t(M)$ are linear functions in $M$. Moreover, we know from \eqref{eqn:u_t lipschitz}-\eqref{eqn:x_pred lipschitz} in the proof of Lemma~\ref{lemma:lipschitz of f pred} that $\tilde{u}_t(\cdot)$ is $\sqrt{qh}R_{\hat{w}}$-Lipschitz and $\tilde{x}_t(\cdot)$ is $\frac{\Gamma\sqrt{qh}R_{\hat{w}}\kappa}{1-\gamma}$-Lipschitz, which implies that $\nm[\big]{\frac{\partial\tilde{u}_t(M)}{\partial M}}\le\sqrt{qh}R_{\hat{w}}$ and $\nm[\big]{\frac{\partial\tilde{x}_t(M)}{\partial M}}\le\frac{\Gamma\sqrt{qh}R_{\hat{w}}\kappa}{1-\gamma}$ (e.g., \cite{boyd2004convex}). Combining the above arguments, we complete the proof of the lemma.$\hfill\blacksquare$

\subsection{Proof of Lemma~\ref{llemma:gradient error}}
Suppose the event $\CE$ holds, and consider any $t\in\{N,\dots,T-1\}$ and any $M\in\D$. For notational simplicity in this proof, we denote $\tilde{x}_t(M)=x_t^{\tt pred}(M)$, $\tilde{u}_t(M)=u_t(M|\hat{w}_{0:t-1})$, and $\hat{x}_t(M)=x_t(M|\hat{\Phi},\hat{w}_{0:t-1})$, where $x_t^{\tt pred}(M)$ is given by Definition~\ref{def:true prediction cost}, and $u_t(M|\hat{w}_{0:t-1}),x_t(M|\hat{\Phi},\hat{w}_{0:t-1})$ are given by Definition~\ref{def:counterfactual cost}. We then have
\begin{align}\nonumber
f_t(M|\hat{\Phi},\hat{w}_{0:t-1})-f_t^{\tt pred}(M)=\hat{x}_t(M)^{\top}Q\hat{x}_t(M)-\tilde{x}_t(M)^{\top}Q\tilde{x}_t(M).
\end{align}
Taking the derivative, we obtain
\begin{align}\nonumber
\frac{\partial f_t(M)}{\partial M}-\frac{\partial f_t^{\tt pred}(M)}{\partial M}&=2\frac{\partial \hat{x}_t(M)}{\partial M}Q\hat{x}_t(M)-2\frac{\partial\tilde{x}_t(M)}{\partial M}Q\tilde{x}_t(M)\\\nonumber
&=2\frac{\partial \hat{x}_t(M)}{\partial M}Q\hat{x}_t(M)-2\frac{\partial \tilde{x}_t(M)}{\partial M}Q\hat{x}_t(M)+2\frac{\partial \tilde{x}_t(M)}{\partial M}Q\hat{x}_t(M)-2\frac{\partial\tilde{x}_t(M)}{\partial M}Q\tilde{x}_t(M)\\
&=2\Big(\frac{\partial \hat{x}_t(M)}{\partial M}-\frac{\partial\tilde{x}_t(M)}{\partial M}\Big)Q\hat{x}_t(M)+2\frac{\partial\tilde{x}_t(M)}{\partial M}Q\big(\hat{x}_t(M)-\tilde{x}_t(M)\big),\label{eqn:gradient difference}
\end{align}
where we assume for notational simplicity that $M=(M_{s}^{[k]})_{k\in[h],s\in\U}$ is already written in its vector form ${\tt Vec}(M)$, i.e., we let $M={\tt Vec}(M)$. 

First, we upper bound $\norm{\hat{x}_t(M)-\tilde{x}_t(M)}$. We see from Definitions~\ref{def:counterfactual cost}-\ref{def:true prediction cost} that
\begin{align}\nonumber
\norm{\hat{x}_t(M)-\tilde{x}_t(M)}&=\nm[\Big]{\sum_{k=t-h}^{t-1}\Big(\hat{A}^{t-(k+1)}\big(\hat{w}_k+\hat{B}\tilde{u}_k(M)\big)-A^{t-(k+1)}\big(\hat{w}_k+B\tilde{u}_k(M)\big)\Big)}\\\nonumber
&=\nm[\Big]{\sum_{k=t-h}^{t-1}\Big((\hat{A}^{t-(k+1)}-A^{t-(k+1)})\hat{w}_k+(\hat{A}^{t-(k+1)}\hat{B}-A^{t-(k+1)}B)\tilde{u}_k(M)\Big)}.
\end{align}
Note from Proposition~\ref{prop:upper bound on est error} that $\norm{A-\hat{A}}\le\bar{\varepsilon}$ and $\norm{B-\hat{B}}\le\bar{\varepsilon}$, and note from Assumption~\ref{ass:stable A} that $\norm{A^k}\le\kappa\gamma^k$ for all $k\ge0$. We then have from Lemma~\ref{lemma:matrix power perturbation bound} that 
\begin{align}\nonumber
\nm[\big]{\hat{A}^k-A^k}&\le k\kappa^2(\kappa\bar{\varepsilon}+\frac{1+\gamma}{2})^{k-1}\bar{\varepsilon}\\\nonumber
&\le k\kappa^2\Big(\frac{3+\gamma}{4}\Big)^{k-1}\bar{\varepsilon},\ \forall k\ge0,
\end{align}
where we use the fact that $\bar{\varepsilon}\le\frac{1-\gamma}{4\kappa}$. We also have
\begin{align}\nonumber
\nm[\big]{\hat{A}^k\hat{B}-A^kB}&\le\nm[\big]{\hat{A}^k-A^k}\nm[\big]{\hat{B}}+\nm[\big]{A^k}\nm[\big]{\hat{B}-B}\\\nonumber
&\le k\kappa^2\Big(\frac{3+\gamma}{4}\Big)^{k-1}\bar{\varepsilon}(\Gamma+\bar{\varepsilon})+\kappa\gamma^k\bar{\varepsilon}.
\end{align}
Now, noting that one can use similar arguments to those for Lemma~\ref{lemma:bounds on norm of alg and pred input/state} and show that $\norm{\hat{w}_k}\le R_{\hat{w}}$ and $\norm{\tilde{u}_t(M)}\le R_u$ for all $k\ge0$, we have from the above arguments that
\begin{align}\nonumber
\norm{\hat{x}_t(M)-\tilde{x}_t(M)}&\le R_{\hat{w}}\sum_{k=0}^{h-1}\nm[big]{\hat{A}^k-A^k}+R_u\sum_{k=0}^{h-1}\nm[\big]{\hat{A}^k\hat{B}-A^kB}\\\nonumber
&\le \big(R_{\hat{w}}+(\Gamma+\bar{\varepsilon})R_u\big)\kappa^2\bar{\varepsilon}\sum_{k=0}^{h-1}k\Big(\frac{3+\gamma}{4}\Big)^{k-1}+R_u\kappa\bar{\varepsilon}\sum_{k=0}^{h-1}\gamma^k\\\nonumber
&\le\big(R_{\hat{w}}+(\Gamma+\bar{\varepsilon})R_u\big)\kappa^2\bar{\varepsilon}\frac{16}{(1-\gamma)^2}+\frac{R_u\kappa\bar{\varepsilon}}{1-\gamma},
\end{align}
where the third inequality follows from a standard formula for series. Moreover, one can use similar arguments for Lemma~\ref{lemma:bounds on norm of alg and pred input/state} and show that $\norm{\tilde{x}_t(M)}\le R_x$, which further implies that 
\begin{align}\nonumber
\norm{\hat{x}_t(M)}&\le\norm{\hat{x}_t(M)-\tilde{x}_t(M)}+\norm{\tilde{x}_t(M)}\\\nonumber
&\le\big(R_{\hat{w}}+(\Gamma+\bar{\varepsilon})R_u\big)\kappa^2\bar{\varepsilon}\frac{16}{(1-\gamma)^2}+\frac{R_u\kappa\bar{\varepsilon}}{1-\gamma}+R_x.
\end{align}

Next, we upper bound $\norm{\partial\hat{x}_t(M)/\partial M-\partial\tilde{x}_t(M)/\partial M}$. Noting from \eqref{eqn:u_t lipschitz} in the proof of Lemma~\ref{lemma:lipschitz of f pred} that $\tilde{u}_t(\cdot)$ is $\sqrt{qh}R_{\hat{w}}$-Lipschitz, we have that $\nm[\big]{\frac{\partial\tilde{u}_t(M)}{\partial M}}\le\sqrt{qh}R_{\hat{w}}$ (e.g., \cite{boyd2004convex}). It then follows from Definitions~\ref{def:counterfactual cost}-\ref{def:true prediction cost} that 
\begin{align}\nonumber
\nm[\Big]{\frac{\partial \hat{x}_t(M)}{\partial M}-\frac{\partial\tilde{x}_t(M)}{\partial M}}&=\nm[\Big]{\sum_{k=t-h}^{t-1}\big(\hat{A}^{t-(k+1)}\hat{B}-A^{t-(k+1)}B\big)\frac{\partial\tilde{u}_k}{\partial M}}\\\nonumber
&\le\sqrt{qh}R_{\hat{w}}\sum_{k=0}^{h-1}\nm[\big]{\hat{A}^k\hat{B}-A^kB}\\\nonumber
&\le\sqrt{qh}R_{\hat{w}}\sum_{k=0}^{h-1}\Big(k\kappa^2\Big(\frac{3+\gamma}{4}\Big)^{k-1}\bar{\varepsilon}(\Gamma+\bar{\varepsilon})+\kappa\gamma^k\bar{\varepsilon}\Big)\\\nonumber
&\le\sqrt{qh}R_{\hat{w}}\Big(\kappa^2\bar{\varepsilon}(\Gamma+\bar{\varepsilon})\frac{16}{(1-\gamma)^2}+\frac{\kappa\bar{\varepsilon}}{1-\gamma}\Big).
\end{align}
Similarly, noting from \eqref{eqn:x_pred lipschitz} in the proof of Lemma~\ref{lemma:lipschitz of f pred} that $\tilde{x}_t(\cdot)$ is $\frac{\Gamma\sqrt{qh}R_{\hat{w}}\kappa}{1-\gamma}$-Lipschitz, which implies that $\nm[\big]{\frac{\partial\tilde{x}_t(M)}{\partial M}}\le\frac{\Gamma\sqrt{qh}R_{\hat{w}}\kappa}{1-\gamma}$.

Finally, combining the above arguments together and noting that $\bar{\varepsilon}\le 1$ and $\kappa\ge1$, one can show via \eqref{eqn:gradient difference} and algebraic manipulations that \eqref{eqn:norm of gradient error} holds.$\hfill\blacksquare$

\subsection{Proof of Lemma~\ref{lemma:strongly convex of conditional cost}}
Let us consider any $t\in\{N+k_f,\dots,T-1\}$ and any $M\in\D$. First, we recall from Definition~\ref{def:counterfactual cost} that 
\begin{align}\nonumber
u_t(M|\hat{w}_{0:t-1})=\sum_{s\in\U}\sum_{k=1}^hI_{\V,s}M_s^{[k]}\hat{\eta}_{t-k,s},
\end{align}
where $\hat{\eta}_{t-k,s}=\begin{bmatrix}\hat{w}_{t-k-l_{vs},j_v}^{\top}\end{bmatrix}^{\top}_{v\in\CL_s}$. Writing the right-hand side of the above equation into a matrix form yields
\begin{align}\nonumber
u_t(M|\hat{w}_{0:t-1})={\tt Mat}(M)\hat{\eta}_t,
\end{align}
where ${\tt Mat}(M)=\begin{bmatrix}I_{\V,s}M_s^{[k]}\end{bmatrix}_{s\in\U,k\in[h]}$ and $\hat{\eta}_t=\begin{bmatrix}\hat{\eta}_{t-k,s}^{\top}\end{bmatrix}_{s\in\U,k\in[h]}^{\top}$. Further writing $M=[M_s^{[k]}]_{s\in\U,k\in[h]}$ into its vector form ${\tt Vec}(M)$, one can show that 
\begin{align}\nonumber
u_t(M|\hat{w}_{0:t-1})=\hat{\Lambda}_t M,
\end{align}
where we let $M={\tt Vec}(M)$ for notational simplicity, and $\hat{\Lambda}_t=\text{diag}((\hat{\eta}_t^{i\top})_{i\in[n]})$ with $\hat{\eta}_t^i=\hat{\eta}_t$. Taking the derivative, we obtain
\begin{align}\nonumber
\frac{\partial u_t(M|\hat{w}_{0:t-1})}{\partial M}=\hat{\Lambda}_t.
\end{align}

Next, we recall from Eq.~\eqref{eqn:f_pred second derivative} in the proof of Lemma~\ref{lemma:upper bound on norn of hessian of f_pred} that 
\begin{align}\nonumber
\nabla^2f_t^{\tt pred}(M)&=2\frac{\partial\tilde{x}_t(M)}{\partial M}^{\top}Q\frac{\partial\tilde{x}_t(M)}{\partial M}+2\frac{\partial\tilde{u}_t(M)}{\partial M}^{\top}R\frac{\partial\tilde{u}_t(M)}{\partial M}\\\nonumber
&\succeq2\sigma_m(R)\frac{\partial\tilde{u}_t(M)}{\partial M}^{\top}\frac{\partial\tilde{u}_t(M)}{\partial M}\\
&=2\sigma_m(R)\hat{\Lambda}_t^{\top}\hat{\Lambda}_t,\label{eqn:lower bound on hessian of f_pred}
\end{align}
where we denote $\tilde{x}_t(M)=x_t^{\tt pred}(M)$ and $\tilde{u}_t(M)=u_t(M|\hat{w}_{0:t-1})$, with $x_t^{\tt pred}(M)$ and $u_t(M|\hat{w}_{0:t-1})$ given by Definitions~\ref{def:true prediction cost} and \ref{def:counterfactual cost}, respectively. To proceed, let us consider any $i\in[n]$ and any element of the vector $\hat{\eta}_t^i$, and denote the element as $\hat{w}\in\R$. Meanwhile, let $w\in\R$ denote the element of $\eta_t=\begin{bmatrix}\eta_{t-k,s}^{\top}\end{bmatrix}_{s\in\U,k\in[h]}^{\top}$ that corresponds to $\hat{w}$, where $\eta_{t-k,s}=\begin{bmatrix}w_{t-k-l_{vs},j_v}^{\top}\end{bmatrix}^{\top}_{v\in\CL_s}$. Note that $\hat{w}^2$ is a diagonal element of the matrix $\hat{\Lambda}_t^{\top}\hat{\Lambda}_t$. One can check that $w$ is an element of some vector $w_{t-k-l_{vs},j_v}\sim\CN(0,\sigma_w^2I_{n_{j_v}})$, where $k\in[h]$, $s\in\U$ and $v\in\CL_{s}$, which implies that $w\sim\CN(0,\sigma_w^2)$. Recalling from the definition of the information graph $\CP(\U,\CH)$ given by \eqref{eqn:def of info graph} that $l_{vs}\le D_{\max}$ for all $v,s\in\U$, where $D_{\max}$ is defined in Eq.~\eqref{eqn:depth of T_i}, we see that $t-k-l_{vs}> t-k_f$. Combining the above arguments, we obtain
\begin{align}\nonumber
\E\big[\hat{w}^2|\F_{t-k_f}\big]&\ge\frac{1}{2}\E\big[w^2|\F_{t-k_f}\big]-\E\big[(\hat{w}-w)^2|\F_{t-k_f}\big]\\
&=\frac{\sigma_w^2}{2}-\E\big[(\hat{w}-w)^2|\F_{t-k_f}\big],\label{eqn:lower bound on E w_hat}
\end{align}
where the equality follows from the fact that $\F_{t-k_f}$ is generated by the stochastic sequence $w_N,\dots,w_{t-k_f}$. Now, suppose the event $\CE$ holds. We know from Lemma~\ref{lemma:bounds on norm of alg and pred input/state} that $\norm{\hat{w}_t-w_t}\le\Delta R_w\bar{\varepsilon}\le\sigma_w/2$ for all $t\in\{N,\dots,T-1\}$. Recalling from Definition~\ref{def:conditional prediction loss} that $f_{t;k_f}^{\tt pred}(M)=\E[f_t^{\tt pred}(M)|\F_{t-k_f}]$, we then have from \eqref{eqn:lower bound on hessian of f_pred}-\eqref{eqn:lower bound on E w_hat} that $\nabla^2f_t^{\tt pred}(M)\succeq\frac{\sigma_m(R)\sigma_w^2}{2}I_{d_m}$, where $d_m$ is the dimension of $M$ (i.e., ${\tt Vec}(M)$).$\hfill\blacksquare$

\section{Proof Pertaining to Upper Bounding $R_3$}\label{sec:proof R_3}
\subsection{Proof of Lemma~\ref{lemma:f_pred minus f_counterfactual}}
Suppose the event $\CE$ holds and consider any $t\in\{N_0,\dots,T-1\}$. First, recall from Definitions~\ref{def:counterfactual cost}-\ref{def:true prediction cost} that 
	\begin{align}\nonumber
		f_t^{\tt pred}(M_{\tt apx})-f_t(\tilde{M}_{\star}|\Phi,w_{0:t-1})&=x_t^{\tt pred}(M_{\tt apx})^{\top}Qx_t^{\tt pred}(M_{\tt apx})-x_t(\tilde{M}_{\star}|\Phi,w_{0:t-1})^{\top}Qx_t(\tilde{M}_{\star}|\Phi,w_{0:t-1})\\\nonumber
		&\quad+u_t(M_{\tt apx}|\hat{w}_{0:t-1})^{\top}Ru_t(M_{\tt apx}|\hat{w}_{0:t-1})-u_t(\tilde{M}_{\star}|w_{0:t-1})^{\top}Ru_t(\tilde{M}_{\star}|w_{0:t-1}).
	\end{align}
	Noting that $M_{\tt apx}\in\D$ and $\tilde{M}_{\star}\in\D_0$, one can use similar arguments to those for Lemma~\ref{lemma:bounds on norm of alg and pred input/state} and show that $\norm{x_t^{\tt pred}(M_{\tt apx})}\le R_x$, $x_t(\tilde{M}_{\star}|\Phi,w_{0:t-1})\le R_x$, $u_t({M_{\tt apx}|\hat{w}_{0:t-1}})\le R_u$, and $u_t(\tilde{M}_{\star}|w_{0:t-1})\le R_u$. Using similar arguments to those for Lemma~\ref{lemma:lipschitz c}, one can then show that 
	\begin{align}\nonumber
		f_t^{\tt pred}(M_{\tt apx})-f_t(\tilde{M}_{\star}|\Phi,w_{0:t-1})&\le 2R_u\sigma_1(R)\nm[\big]{u_t(M_{\tt apx}|\hat{w}_{0:t-1})-u_t(\tilde{M}_{\star}|w_{0:t-1})}\\\nonumber
		&\quad+2R_x\sigma_1(Q)\nm[\Big]{\sum_{k=t-h}^{t-1}A^{t-(k+1)}B\big(u_k(M_{\tt apx}|\hat{w}_{0:k-1})-u_k(\tilde{M}_{\star}|w_{0:k-1}))\big)}\\\nonumber
		\le2\big(&R_u\sigma_1(R)+R_x\sigma_1(Q)\big)\frac{\Gamma\kappa}{1-\gamma}\max_{k\in\{t-h,\dots,t\}}\nm[\big]{u_k(M_{\tt apx}|\hat{w}_{0:k-1})-u_k(\tilde{M}_{\star}|w_{0:k-1}))},
	\end{align}
	where the second inequality follows from Assumption~\ref{ass:stable A}.$\hfill\blacksquare$

\subsection{Proof of Lemma~\ref{lemma:u M_apx minus u tilde M_star}}
Suppose the event $\CE$ holds. Consider any $t\in\{N+h,\dots,T-1\}$. First, recall from Definition~\ref{def:counterfactual cost} that $u_t(\tilde{M}_{\star}|w_{0:t-1})=\sum_{s\in\U}\sum_{k=t-h/4}^{t-1}I_{\V,s}\tilde{M}_{\star,s}^{[t-k]}\eta_{k,s}$, where $\eta_{k,s}=\begin{bmatrix}w_{k-l_{vs},j_v}^{\top}\end{bmatrix}^{\top}_{v\in\CL_s}$ with $w_{j_v}\to v$ and $\CL_s$ given by Eq.~\eqref{eqn:set of leaf nodes of s}. We can then write
\begin{align}\nonumber
u_t(\tilde{M}_{\star}|w_{0:t-1})=\sum_{s\in\U}\sum_{k=t-\frac{h}{4}}^{t-1}I_{\V,s}\tilde{M}_{\star,s}^{t-k}(\hat{\eta}_{k,s}+\eta_{k,s}-\hat{\eta}_{k,s}),
\end{align}
where $\hat{\eta}_{k,s}=\begin{bmatrix}\hat{w}_{k-l_{vs},j_v}^{\top}\end{bmatrix}^{\top}_{v\in\CL_s}$. Denoting $\tilde{s}=\{j_v:w_{j_v}\to v,v\in\CL_s\}$, we see from Eq.~\eqref{eqn:overall system} that 
\begin{align}\nonumber
\eta_{k,s}=\sum_{v\in\CL_s}I_{\tilde{s},\{j_v\}}I_{\{j_v\},\V}(x_{k-l_{vs}+1}-Ax_{k-l_{vs}}-Bu_{k-l_{vs}}^{\tt alg}).
\end{align}
Similarly, we see from \eqref{eqn:est w_i t} that 
\begin{align}\nonumber
\hat{\eta}_{k,s}=\sum_{v\in\CL_s}I_{\tilde{s},\{j_v\}}I_{\{j_v\},\V}(x_{k-l_{vs}+1}-\hat{A}x_{k-l_{vs}}-\hat{B}u_{k-l_{vs}}^{\tt alg}).
\end{align}
For notational simplicity in the remaining of this proof, we denote 
\begin{equation*}
\begin{split}
I(k)&=\{k-\frac{h}{4},\dots,k-1\},\\
I_{vs}^0(k)&=\{0,\dots,k-l_{vs}-1\},\\
I_{vs}^1(k)&=\{k-l_{vs}-\frac{h}{4},\dots,k-l_{vs}-1\},\\
I_{vs}^2(k)&=\{0,\dots,k-l_{vs}-\frac{h}{4}-1\},
\end{split}
\end{equation*}
for all $k\ge0$. Moreover, denote $\Delta A=\hat{A}-A$ and $\Delta B=\hat{B}-B$. Now, noting that
\begin{align}\nonumber
x_{k-l_{vs}}=\sum_{k^{\prime}=0}^{k-l_{vs}-1}A^{k-l_{vs}-(k^{\prime}+1)}(Bu_{k^{\prime}}^{\tt alg}+w_{k^{\prime}}),
\end{align}
for all $k\in I(t)$, we have from the above arguments that
\begin{align}\nonumber
u_t(\tilde{M}_{\star}|w_{0:t-1})&=\sum_{s\in\U}\sum_{k\in I(t)}I_{\V,s}\tilde{M}_{\star,s}^{[t-k]}\Big(\hat{\eta}_{k,s}+\sum_{v\in\CL_s}I_{\tilde{s},\{j_v\}}I_{\{j_v\},\V}\big(\Delta Ax_{k-l_{vs}}+\Delta Bu_{k-l_{vs}}^{\tt alg}\big)\Big)\\\nonumber
&=\sum_{s\in\U}\sum_{k\in I(t)}I_{\V,s}\tilde{M}_{\star,s}^{[t-k]}\bigg(\hat{\eta}_{k,s}+\sum_{v\in\CL_s}I_{\tilde{s},\{j_v\}}I_{\{j_v\},\V}\\\nonumber
&\qquad\qquad\qquad\quad\times\Big(\Delta A\big(\sum_{k^{\prime}\in I_{vs}^0(k)}A^{k-l_{vs}-(k^{\prime}+1)}(Bu_{k^{\prime}}^{\tt alg}+w_{k^{\prime}})\big)+\Delta Bu_{k-l_{vs}}^{\tt alg}\Big)\bigg)\\
&=u_t^{\tt tr}+u_t^{\tt ma},\label{eqn:u_t tilde M_star decompose}
\end{align}
where 
\begin{align}\nonumber
u_t^{\tt tr}&=\sum_{s\in\U}\sum_{k\in I(t)}I_{\V,s}\tilde{M}_{\star,s}^{[t-k]}\sum_{v\in\CL_s}I_{\tilde{s},\{j_v\}}I_{\{j_v\},\V}\big(\Delta A\sum_{k^{\prime}\in I_{vs}^2(k)}A^{k-l_{vs}-(k^{\prime}+1)}(Bu_{k^{\prime}}^{\tt alg}+w_{k^{\prime}})\big),\\\nonumber
u_t^{\tt ma}&=\sum_{s\in\U}\sum_{k\in I(t)}I_{\V,s}\tilde{M}_{\star,s}^{[t-k]}\bigg(\hat{\eta}_{k,s}+\sum_{v\in\CL_s}I_{\tilde{s},\{j_v\}}I_{\{j_v\},\V}\\\nonumber
&\qquad\qquad\qquad\qquad\qquad\times\Big(\Delta A\big(\sum_{k^{\prime}\in I_{vs}^1(k)}A^{k-l_{vs}-(k^{\prime}+1)}(Bu_{k^{\prime}}^{\tt alg}+w_{k^{\prime}})\big)+\Delta Bu_{k-l_{vs}}^{\tt alg}\Big)\bigg).
\end{align}

In the following, we analyze the terms $u_t^{\tt tr}$ and $u_t^{\tt ma}$. 
\begin{claim}
\label{claim:upper bound on u tr}
It holds that $\norm{u_t^{\tt tr}}\le\bar{\varepsilon}(\Gamma R_u+R_w)p^2qh\sqrt{n}\frac{\kappa^2\Gamma^{2D_{\max}+1}\gamma^{h/4}}{4(1-\gamma)}$.
\end{claim}
\begin{proof}
Recall from Proposition~\ref{prop:upper bound on est error} that $\norm{\Delta A}\le\bar{\varepsilon}$, and recall from Lemma~\ref{lemma:bounds on norm of alg and pred input/state} that $\norm{u_{k}^{\tt alg}}\le R_u$ and $\norm{w_k}\le R_w$. Also note from Eq.~\eqref{eqn:the class of DFCs 0} that $\norm{\tilde{M}_{\star,s}^{[k]}}\le\sqrt{n}\kappa p\Gamma^{2D_{\max}+1}$ for all $k\in[\frac{h}{4}]$, and note from Assumption~\ref{ass:stable A} that $\norm{A^k}\le\kappa\gamma^k$ for all $k\ge0$. We then have
\begin{align}\nonumber
\norm{u_t^{\tt tr}}&\le\sum_{s\in\U}\sum_{k\in I(t)}\nm[\big]{\tilde{M}_{\star,s}^{[t-k]}}\sum_{v\in\CL_s}\bar{\varepsilon}(\Gamma R_u+R_w)\sum_{k^{\prime}\in I_{vs}^2(k)}\nm[\big]{A^{k-l_{vs}-(k^{\prime}+1)}}\\\nonumber
&\le\bar{\varepsilon}(\Gamma R_u+R_w)\sum_{s\in\U}\sum_{v\in\CL_s}\sum_{k\in I(t)}\nm[\big]{\tilde{M}_{\star,s}^{[t-k]}}\frac{\kappa\gamma^{h/4}}{1-\gamma}\\\nonumber
&\le\bar{\varepsilon}(\Gamma R_u+R_w)\frac{\kappa\gamma^{h/4}}{4(1-\gamma)}qh\sqrt{n}\kappa p^2\Gamma^{2D_{\max}+1}\\\nonumber
&\le\bar{\varepsilon}(\Gamma R_u+R_w)p^2qh\sqrt{n}\frac{\kappa^2\Gamma^{2D_{\max}+1}\gamma^{h/4}}{4(1-\gamma)},\nonumber
\end{align}
where the third inequality uses the fact that $|\CL_s|\le p$ for all $s\in\U$.
\end{proof}

Note from the definition of Algorithm~\ref{algorithm:control design} that for any $k\ge N$,
\begin{align}\nonumber
u_k^{\tt alg}&=\sum_{s\in\U}\sum_{k^{\prime}=t-h}^{k-1}I_{\V,s}M_{t,s}\hat{\eta}_{k^{\prime},s}\\
&=\underbrace{\Big(\sum_{s\in\U}\sum_{k^{\prime}\in I(k)}I_{\V,s}\tilde{M}_{\star,s}^{[k-k^{\prime}]}\hat{\eta}_s(k^{\prime})\Big)}_{\tilde{u}_k}+\underbrace{u_k(M_k-\tilde{M}_{\star}|\hat{w}_{0:k-1})}_{\Delta\tilde{u}_k},\label{eqn:u_k alg}
\end{align}
where we use fact that $\tilde{M}_{\star,s}^{[k^{\prime}]}=0$ for all $k^{\prime}>\frac{h}{4}$. Denoting $\Delta w_k=w_k-\hat{w}_k$ for all $k\ge0$, we can then write
\begin{align}\nonumber
u_t^{\tt ma}&=\sum_{s\in\U}\sum_{k\in I(t)}I_{\V,s}\tilde{M}_{\star,s}^{[t-k]}\bigg(\hat{\eta}_{k,s}+\sum_{v\in\CL_s}I_{\tilde{s},\{j_v\}}I_{\{j_v\},\V}\\\nonumber
&\quad\times\Big(\Delta A\big(\sum_{k^{\prime}\in I_{vs}^1(k)}A^{k-l_{vs}-(k^{\prime}+1)}(B(\tilde{u}_{k^{\prime}}+\Delta\tilde{u}_{k^{\prime}})+\hat{w}_{k^{\prime}}+\Delta w_{k^{\prime}})\big)+\Delta B(\tilde{u}_{k-l_{vs}}+\Delta\tilde{u}_{k-l_{vs}})\Big)\bigg)\\
&=u_t^{\tt apx}+u_t^{\tt er},\label{eqn:u ma decompose}
\end{align}
where 
\begin{align}\nonumber
u_t^{\tt apx}&=\sum_{s\in\U}\sum_{k\in I(t)}I_{\V,s}\tilde{M}_{\star,s}^{[t-k]}\bigg(\hat{\eta}_{k,s}+\sum_{v\in\CL_s}I_{\tilde{s},\{j_v\}}I_{\{j_v\},\V}\\\nonumber
&\qquad\qquad\qquad\qquad\qquad\qquad\qquad\times\Big(\Delta A\big(\sum_{k^{\prime}\in I_{vs}^1(k)}A^{k-l_{vs}-(k^{\prime}+1)}(B\tilde{u}_{k^{\prime}}+\hat{w}_{k^{\prime}})\big)+\Delta B\tilde{u}_{k-l_{vs}}\Big)\bigg),\\\nonumber
u_t^{\tt er}&=\sum_{s\in\U}\sum_{k\in I(t)}I_{\V,s}\tilde{M}_{\star,s}^{[t-k]}\bigg(\sum_{v\in\CL_s}I_{\tilde{s},\{j_v\}}I_{\{j_v\},\V}\\\nonumber
&\qquad\qquad\qquad\qquad\qquad\qquad\times\Big(\Delta A\big(\sum_{k^{\prime}\in I_{vs}^1(k)}A^{k-l_{vs}-(k^{\prime}+1)}(B\Delta\tilde{u}_{k^{\prime}}+\Delta w_{k^{\prime}})\big)+\Delta B\Delta\tilde{u}_{k-l_{vs}}\Big)\bigg).
\end{align}
\begin{claim}
\label{claim:existence of M apx}
There exists $M_{\tt apx}\in\D$ with $\D$ given by Eq.~\eqref{eqn:the class of DFCs} such that $u_{t}^{\tt apx}=u_{t}(M_{\tt apx}|\hat{w}_{0:t-1})$, and $\norm{M_{{\tt apx},s}^{[k]}-\tilde{M}_{\star,s}^{[k]}}\le \frac{qh^2n\kappa^3p^3}{4}\Gamma^{4D_{\max}+3}\bar{\varepsilon}$ for all $k\in[h]$ and all $s\in\U$.
\end{claim}
\begin{proof}
Noting \eqref{eqn:u_k alg}, we can first write
\begin{align}\nonumber
u_t^{\tt apx}=\Big(\sum_{s\in\U}\sum_{k\in I(t)}I_{\V,s}\tilde{M}_{\star,s}^{[t-k]}\hat{\eta}_{k,s}\Big)+u_t^{\tt apx_0}+u_t^{\tt apx_1}+u_t^{\tt apx_2},
\end{align}
where 
\begin{align}\nonumber
u_t^{\tt apx_0}&=\sum_{r,s\in\U}\sum_{v\in\CL_s}\sum_{k\in I(t)}\sum_{k^{\prime}\in I_{vs}^1(k)}\sum_{k^{\prime\prime}\in I(k^{\prime})}I_{\V,s}\tilde{M}_{\star,s}^{[t-k]}I_{\tilde{s},\{j_v\}}I_{\{j_v\},\V}\Delta A A^{k-l_{vs}-(k^{\prime}+1)}BI_{\V,r}\tilde{M}_{\star,r}^{[k^{\prime}-k^{\prime\prime}]}\hat{\eta}_r(k^{\prime\prime}),\\\nonumber
u_t^{\tt apx_1}&=\sum_{r,s\in\U}\sum_{v\in\CL_s}\sum_{k\in I(t)}\sum_{k^{\prime}\in I_{vs}^1(k)}I_{\V,s}\tilde{M}_{\star,s}^{[t-k]}I_{\tilde{s},\{j_v\}}I_{\{j_v\},\V}\Delta A A^{k-l_{vs}-(k^{\prime}+1)}BI_{\V,r}M_{w,r}\hat{\eta}_r(k^{\prime}),\\\nonumber
u_t^{\tt apx_2}&=\sum_{r,s\in\U}\sum_{v\in\CL_s}\sum_{k\in I(t)}\sum_{k^{\prime}\in I_{vs}^1(k)}I_{\V,s}\tilde{M}_{\star,s}^{[t-k]}I_{\tilde{s},\{j_v\}}I_{\{j_v\},\V}\Delta BI_{\V,r}\tilde{M}_{\star,r}^{[k-l_{vs}-k^{\prime}]}\hat{\eta}_r(k^{\prime}),
\end{align}
where $M_{w,r}=I_{n_{j_r}}$ if $r\in\CL$ , and $M_{w,r}=0$ if $r\notin\CL$, with $w_{j_r}\to r$ and $\CL$ defined in Eq.~\eqref{eqn:set of leaf nodes}. Noting from the choice of $h$ that $h\ge4(D_{\max}+1)$, one can check that for any $k\in I(t)$, any $k^{\prime}\in I_{vs}^1(k)$ and any $k^{\prime\prime}\in I(k^{\prime})$, it holds that $k\in\{t-h,\dots,t-1\}$. It then follows that there exists $(\tilde{M}_{{\tt apx},s}^{[k]})_{k\in[h],s\in\U}$ such that  
\begin{align}
u_t^{\tt apx_0}+u_t^{\tt apx_1}+u_t^{\tt apx_2}=\sum_{s\in\U}\sum_{k=t-h}^{t-1}I_{\V,s}\tilde{M}_{{\tt apx},s}^{[t-k]}\hat{\eta}_{k,s}.\label{eqn:expression based on tilde M apx}
\end{align}
Moreover, noting from Proposition~\ref{prop:upper bound on est error} that $\norm{\Delta A}\le\bar{\varepsilon}$ and $\norm{\Delta B}\le\bar{\varepsilon}$, and noting from Eq.~\eqref{eqn:tilde M star} that $\norm{\tilde{M}_{\star,s}^{[k]}}\le\sqrt{n}\kappa p\Gamma^{2D_{\max}+1}$ for all $s\in\U$ and all $k\in[h/4]$, where $0<\gamma<1$, one can show via the above arguments that 
\begin{align}\nonumber
\norm{\tilde{M}_{{\tt apx},s}^{[t-k]}}\le 3qp\Big(\frac{h}{4}\Big)^2n\kappa^3p^2\Gamma^{4D_{\max}+3}\bar{\varepsilon},\ \forall k\in\{t-h,\dots,t-1\}.
\end{align}
It follows that $u_t^{\tt apx}=\sum_{s\in\U}\sum_{k\in I(t)}I_{\V,s}\tilde{M}_{\star,s}^{[t-k]}\hat{\eta}_{k,s}+\sum_{s\in\U}\sum_{t-h}^{t-1}I_{\V,s}\tilde{M}_{{\tt apx},s}^{[t-k]}\hat{\eta}_{k,s}$, which also implies that there exist $M_{\tt apx}=(M_{{\tt apx},s}^{[t-k]})_{k\in[h],s\in\U}$ such that $u_t^{\tt apx}=u_t(M_{\tt apx}|\hat{w}_{0:t-1})=\sum_{s\in\U}\sum_{k=t-h}^{t-1}I_{\V,s}M_{{\tt apx},s}^{[t-k]}\hat{\eta}_{k,s}$, where
\begin{align*}
M_{{\tt apx},s}^{[t-k]}=
\begin{cases}
\tilde{M}_{\star,s}^{[t-k]}+\tilde{M}_{{\tt apx},s}^{[t-k]}\ \text{if}\ k\in\{t-\frac{h}{4},\dots,t-1\},\\
\tilde{M}_{{\tt apx},s}^{[t-k]}\ \text{if}\ k\in\{t-h,\dots,t-\frac{h}{4}-1\},
\end{cases}
\end{align*}
for all $s\in\U$. Noting the fact that $\tilde{M}_{\star,s}^{[t-k]}=0$ for all $k>\frac{h}{4}$, we have $M_{{\tt apx},s}^{[t-k]}-\tilde{M}_{\star,s}^{[t-k]}=\tilde{M}_{{\tt apx},s}^{[t-k]}$ for all $k\in\{t-h,\dots,t-1\}$ and all $s\in\U$. Hence, we have from the above arguments that
\begin{align}\nonumber
\nm[\big]{M_{{\tt apx},s}^{[t-k]}-\tilde{M}_{\star,s}^{[t-k]}}=\nm[\big]{\tilde{M}_{{\tt apx},s}^{[t-k]}}&\le 3q\Big(\frac{h}{4}\Big)^2n\kappa^3p^3\Gamma^{4D_{\max}+3}\bar{\varepsilon}\\\nonumber
&\le\sqrt{n}\kappa p\Gamma^{2D_{\max}+1},
\end{align}
for all $k\in\{t-h,\dots,t-1\}$ and all $s\in\U$, where the second inequality follows from the choice of $\bar{\varepsilon}$ in Eq.~\eqref{eqn:epsilon_N}. It follows that 
\begin{align}\nonumber
\nm[\big]{M_{{\tt apx},s}^{[t-k]}}\le\nm[\big]{\tilde{M}_{\star,s}^{[t-k]}}+\sqrt{n}\kappa p\Gamma^{2D_{\max}+1}\le2\sqrt{n}\kappa p\Gamma^{2D_{\max}+1},
\end{align}
for all $k\in\{t-h,\dots,t-1\}$ and all $s\in\U$, which implies that $M_{\tt apx}\in\D$.
\end{proof}

\begin{claim}
\label{claim:upper bound on u er}
For any $t_1\in\{N_0,\dots,T-1\}$ and any $\mu\in\R_{>0}$, it holds that
\begin{align}\nonumber
\norm{u_{t_1}^{\tt er}}&\le\frac{p^5q^3n^{\frac{3}{2}}}{8}h^5\kappa^4\Gamma^{6D_{\max}+4}(\kappa\Gamma+1)R_{\hat{w}}\bar{\varepsilon}^2+\frac{1}{8\mu}\big(p^2q\sqrt{n}h^2\kappa\Gamma^{2D_{\max}+1}(\kappa\Gamma+1)\big)\bar{\varepsilon}^2\\\nonumber
&\qquad\qquad\qquad\qquad\qquad\qquad\qquad+\frac{\mu}{2}\max_{k\in\{t_1-h,\dots,t_1-1\}} \nm[\big]{u_k(M_k-\tilde{M}_{\star}|\hat{w}_{0:k-1})},
\end{align}
\end{claim}
\begin{proof}
Consider any $t_1\in\{N+2h,\dots,T-1\}$. We first write
\begin{align}\nonumber
u_{t_1}^{\tt er}=u_{t_1}^{\tt er_0}+u_{t_1}^{\tt er_1}+u_{t_1}^{\tt er_2},
\end{align}
where 
\begin{align}\nonumber
u_{t_1}^{\tt er_0}&=\sum_{s\in\U}\sum_{k\in I(t)}\sum_{v\in\CL_s}\sum_{k^{\prime}\in I_{vs}^1(k)}I_{\V,s}\tilde{M}_{\star,s}^{[t_1-k]}I_{\tilde{s},\{j_v\}}I_{\{j_v\},\V}\Delta AA^{k-l_{vs}-(k^{\prime}+1)}B\Delta\tilde{u}_{k^{\prime}},\\\nonumber
u_{t_1}^{\tt er_1}&=\sum_{s\in\U}\sum_{k\in I(t_1)}\sum_{v\in\CL_s}\sum_{k^{\prime}\in I_{vs}^1(k)}I_{\V,s}\tilde{M}_{\star,s}^{[t_1-k]}I_{\tilde{s},\{j_v\}}I_{\{j_v\},\V}\Delta A A^{k-l_{vs}-(k^{\prime}+1)}\Delta w_{k^{\prime}},\\\nonumber
u_{t_1}^{\tt er_2}&=\sum_{s\in\U}\sum_{k\in I(t_1)}\sum_{v\in\CL_s}I_{\V,s}\tilde{M}_{\star,s}^{[t_1-k]}I_{\tilde{s},\{j_v\}}I_{\{j_v\},\V}\Delta B\Delta\tilde{u}_{k-l_{vs}}.
\end{align}
Similarly to the arguments in the proof of Claim~\ref{claim:existence of M apx}, one can bound
\begin{align}\nonumber
\norm{u_{t_1}^{\tt er_0}}&\le q\Big(\frac{h}{4}\Big)^2\sqrt{n}\kappa p^2\Gamma^{2D_{\max}+1}\bar{\varepsilon}\kappa\Gamma\max_{k\in\{t_1-h,\dots,t_1-1\}} \nm[\big]{u_k(M_k-\tilde{M}_{\star}|\hat{w}_{0:k-1})},\\\nonumber
\norm{u_{t_1}^{\tt er_1}}&\le q\Big(\frac{h}{4}\Big)^2\sqrt{n}\kappa p^2\Gamma^{2D_{\max}+1}\bar{\varepsilon}\kappa R_{\hat{w}}\bar{\varepsilon}\\\nonumber
\norm{u_{t_1}^{\tt er_2}}&\le q\Big(\frac{h}{4}\Big)\sqrt{n}\kappa p^2\Gamma^{2D_{\max}+1}\bar{\varepsilon}\max_{k\in\{t_1-h,\dots,t_1-1\}} \nm[\big]{u_k(M_k-\tilde{M}_{\star}|\hat{w}_{0:k-1})},
\end{align}
where the second inequality also uses the fact from Lemma~\ref{lemma:bounds on norm of alg and pred input/state} that $\Delta w_k\le R_{\hat{w}}\bar{\varepsilon}$ for all $k\in\{N,\dots,T-1\}$. It then follows that 
\begin{align}
\norm{u_{t_1}^{\tt er}}\le p^2q\sqrt{n}h^2\frac{\kappa^2\Gamma^{2D_{\max}+1}}{16}R_{\hat{w}}\bar{\varepsilon}^2 + p^2q\sqrt{n}h^2\frac{\kappa\Gamma^{2D_{\max}+1}}{4}\bar{\varepsilon}(\kappa\Gamma+1)\max_{k\in\{t_1-h,\dots,t_1-1\}} \nm[\big]{u_k(M_k-\tilde{M}_{\star}|\hat{w}_{0:k-1})}.\label{eqn:upper bound the norm of u er}
\end{align}

Moreover, we have that
\begin{align}\nonumber
\nm[\big]{u_k(M_k-\tilde{M}_{\star}|\hat{w}_{0:k-1})}\le\nm[\big]{u_k(M_k-M_{\tt apx}|\hat{w}_{0:k-1})}+\nm[\big]{u_k(\tilde{M}_{\star}-M_{\tt apx}|\hat{w}_{0:k-1})},
\end{align}
for all $k\in\{t_1-h,\dots,t_1-1\}$. Noting that $t_1\in\{N+2h,\dots,T-1\}$, we know that $k\in\{N+h,\dots,T-1\}$ for all $k\in\{t_1-h,\dots,T-1\}$. Now, consider any $k\in\{t_1-h,\dots,t_1-1\}$. Using similar arguments to those for \eqref{eqn:u_t lipschitz} in the proof of Lemma~\ref{lemma:lipschitz of f pred}, we have that 
\begin{align}\nonumber
\nm[\big]{u_k(M_{\tt apx}-\tilde{M}_{\star}|\hat{w}_{0:k-1})}\le\sqrt{qh}R_{\hat{w}}\nm[\big]{{\tt Mat}(M_{\tt apx})-{\tt Mat}(\tilde{M}_{\star})},
\end{align}
where ${\tt Mat}(M_{\tt apx})=\begin{bmatrix}I_{\V,s}M_{{\tt apx},s}^{[k^{\prime}]}\end{bmatrix}_{s\in\U,k^{\prime}\in[h]}$ and ${\tt Mat}(\tilde{M}_{\star})=\begin{bmatrix}I_{\V,s}\tilde{M}_{\star,s}^{[k^{\prime}]}\end{bmatrix}_{s\in\U,k^{\prime}\in[h]}$. Recalling from Claim~\ref{claim:existence of M apx} that $\norm{M_{{\tt apx},s}^{[k]}-\tilde{M}_{\star,s}^{[k]}}\le\frac{qh^2n\kappa^3p^3}{4}\Gamma^{4D_{\max}+3}\bar{\varepsilon}$ for all $k\in[h]$ and all $s\in\U$, one can also show that 
\begin{align}\nonumber
\nm[\big]{{\tt Mat}(M_{\tt apx})-{\tt Mat}(\tilde{M}_{\star})}\le\frac{\sqrt{qh}}{4}qh^2n\kappa^3p^3\Gamma^{4D_{\max}+3}\bar{\varepsilon},
\end{align}
which implies that 
\begin{align}\nonumber
\nm[\big]{u_k(M_{\tt apx}-\tilde{M}_{\star}|\hat{w}_{0:k-1})}\le\frac{p^3q^2n\kappa^3}{4}h^3\Gamma^{4D_{\max}+3}R_{\hat{w}}\bar{\varepsilon}.
\end{align}

Thus, we have from \eqref{eqn:upper bound the norm of u er} that 
\begin{align}\nonumber
\norm{u_{t_1}^{\tt er}}&\le\frac{p^5q^3n^{\frac{3}{2}}}{8}h^5\kappa^4\Gamma^{6D_{\max}+4}(\kappa\Gamma+1)R_{\hat{w}}\bar{\varepsilon}^2\\\nonumber
&\qquad\qquad+\frac{p^2q}{4}h^2\sqrt{n}\kappa^2\Gamma^{2D_{\max}+1}(\kappa\Gamma+1)\bar{\varepsilon}\max_{k\in\{t_1-h,\dots,t_1-1\}} \nm[\big]{u_k(M_k-M_{{\tt apx}}|\hat{w}_{0:k-1})}\\\nonumber
&\le\frac{p^5q^3n^{\frac{3}{2}}}{8}h^5\kappa^4\Gamma^{6D_{\max}+4}(\kappa\Gamma+1)R_{\hat{w}}\bar{\varepsilon}^2+\frac{1}{8\mu}\big(p^2qh^2\sqrt{n}\kappa\Gamma^{2D_{\max}+1}(\kappa\Gamma+1)\big)\bar{\varepsilon}^2\\\nonumber
&\qquad\qquad\qquad\qquad\qquad\qquad\qquad+\frac{\mu}{2}\max_{k\in\{t_1-h,\dots,t_1-1\}} \nm[\big]{u_k(M_k-M_{\tt apx}|\hat{w}_{0:k-1})}^2,
\end{align}
where the inequality follows from the fact that $ab\le\frac{a^2}{2\mu}+\frac{\mu}{2}b^2$ for any $a,b\in\R$ and any $\mu\in\R_{>0}$.
\end{proof}

Finally, putting the above arguments (particularly Eq.~\eqref{eqn:u_t tilde M_star decompose} and \eqref{eqn:u ma decompose}, and Claims~\ref{claim:upper bound on u tr}-\ref{claim:upper bound on u er}) together, we have that for any $k\in\{N+2h,\dots,T-1\}$,
\begin{align}\nonumber
\nm[\big]{u_k(\tilde{M}_{\star}|w_{0:t-1})-u_k(M_{\tt apx}|\hat{w}_{0:t-1})}&=\nm[\big]{u_k^{\tt tr}+u_k^{\tt apx}+u_k^{\tt er}-u_k(\tilde{M}_{\star}|\hat{w}_{0:k-1})}\\\nonumber
&\le\norm{u_k^{\tt tr}}+\norm{u_k^{\tt er}}+\nm[\big]{u_k^{\tt apx}-u_k(\tilde{M}_{\star}|\hat{w}_{0:k-1})}\\\nonumber
&=\norm{u_k^{\tt tr}}+\norm{u_k^{\tt er}},
\end{align}
which implies that \eqref{eqn:upper bound on the difference between two u's} holds, completing the proof of the lemma.$\hfill\blacksquare$

\subsection{Proof of Lemma~\ref{lemma:upper bound on R_3}}
First, we have from Lemma~\ref{lemma:f_pred minus f_counterfactual} that 
	\begin{align}\nonumber
		R_3\le2\big(R_u\sigma_1(R)+R_x\sigma_1(Q)\big)\frac{\Gamma\kappa}{1-\gamma}\sum_{t=N_0}^{T-1}\max_{k\in\{t-h,\dots,t\}}\nm[\big]{u_k(M_{\tt apx}|\hat{w}_{0:k-1})-u_k(\tilde{M}_{\star}|w_{0:k-1})}.
	\end{align}
	Noting that $N_0=N+3h+D_{\max}$, we see that for any $t\in\{N_0,\dots,T-1\}$ and any $k\in\{t-h,\dots,t\}$, $k\in\{N+2h,\dots,T-1\}$. It then follows from Lemma~\ref{lemma:u M_apx minus u tilde M_star} that 
	\begin{align}\nonumber
		&R_3\le\big(R_u\sigma_1(R)+R_x\sigma_1(Q)\big)\frac{2\Gamma\kappa}{1-\gamma}\Big(\bar{\varepsilon}(\Gamma R_u+R_w)p^2qh\sqrt{n}\frac{\kappa^2\Gamma^{2D_{\max}+1}\gamma^{\frac{h}{4}}}{4(1-\gamma)}T+\frac{p^2qh^2}{8\mu}\sqrt{n}\kappa\Gamma^{2D_{\max}+1}(\kappa\Gamma+1)\bar{\varepsilon}^2T\\\nonumber
		&+\frac{p^5q^3n^{\frac{3}{2}}h^5\kappa^4}{8}\Gamma^{6D_{\max}+4}(\kappa\Gamma+1)R_{\hat{w}}\bar{\varepsilon}^2T+\frac{qh^2R_{\hat{w}}\mu}{2}\sum_{t=N_0}^{T-1}\max_{\substack{k\in\{t-h,\dots,t\}\\ t_1\in\{k-h,\dots,k-1\}}}\nm[\big]{u_{t_1}(M_{\tt apx}|\hat{w}_{0:t_1-1})-u_k(\tilde{M}_{t_1}|w_{0:t_1-1})}^2\Big).
	\end{align}
	To complete the proof of the lemma, we note that 
	\begin{align}\nonumber
		\sum_{t=N_0}^{T-1}\max_{\substack{k\in\{t-h,\dots,t\}\\ t_1\in\{k-h,\dots,k-1\}}}\nm[\big]{u_{t_1}(M_{\tt apx}|\hat{w}_{0:t_1-1})-u_{t_1}(\tilde{M}_{t_1}|w_{0:t_1-1})}^2&\le 2h\sum_{t=N}^{T-1}\nm[\big]{u_t(M_t-M_{\tt apx}|\hat{w}_{0:t-1})}^2\\\nonumber
		&\le 2qh^2R_{\hat{w}}\sum_{t=N}^{T-1}\nm[\big]{{\tt Vec}(M_t)-{\tt Vec}(M_{\tt apx})}^2,
	\end{align}
	where the second inequality follows from similar arguments to those for \eqref{eqn:u_t lipschitz} in the proof of Lemma~\ref{lemma:lipschitz of f pred}.$\hfill\blacksquare$

\section{Proof Omitted in Section~\ref{sec:upper bound on regret}}\label{sec:proof of failure regret}
First, we recall line~10 in Algorithm~\ref{algorithm:control design} and define $N_a=\min_{i\in\V}N_{a,i}$, where
\begin{align}
	N_{a,i}\triangleq\min\big\{t\ge N+D_{\max}:\norm{x_{t,i}^{\tt alg}}>R_x\ \text{or}\ \norm{u_{t,i}(M_t|\hat{w}_{0:t-1})}>R_u\big\},\ \forall i\in\V.
\end{align}
Noting from Lemma~\ref{lemma:bounds on norm of alg and pred input/state} that $N_a\ge T$ on the event $\CE$, and noting that $\Prob(\CE)\ge1-1/T$, we see that $\Prob(\ind\{N_a\le T-1\})\le 1/T$. We then have the following result; the proof is included in Appendix~\ref{sec:technical lemma}.
\begin{lemma}\label{lemma:expected state norm bound}
	It holds that 
	\begin{align}
		&\E\Big[\nm[\big]{x_{N_a}^{\tt alg}}^2\ind\{N_a\le T-1\}\Big]\le\frac{2\Gamma^2p(R_x+R_u)^2}{T}+\frac{30\kappa^2(1+\Gamma^2)\overline{\sigma}^2n\log\frac{3T}{2}}{(1-\gamma)T},\label{eqn:upper bound on norm of x_Na}\\
		&\E\Big[\nm[\big]{u_{N_a}^{\tt alg}}^2\ind\{N_a\le T-1\}\Big]\le2q^2h^2n\kappa^2p^3\Gamma^{4D_{\max}}\Big(\frac{4(\Gamma+\bar{\varepsilon})^2p}{T}(R_x+R_u)^2+\frac{30\kappa^2(1+\Gamma^2)\overline{\sigma}^2n\log\frac{3T}{2}}{(1-\gamma)T}\Big).\label{eqn:upper bound on norm of u_Na}
	\end{align} 
\end{lemma}
To proceed, denoting $\bar{\CE}=(\CE\cap\CE_{R_2})^c$, we further have the following decomposition:
\begin{multline}
	\E\Big[\ind\{\bar{\CE}\}\sum_{t=0}^{T-1}c(x_t^{\tt alg},u_t^{\tt alg})\Big]=\underbrace{\E\Big[\ind\{\bar{\CE}\}\sum_{t=0}^{N-1}c(x_t^{\tt alg},u_t^{\tt alg})\Big]}_{\bar{R}_0}+\underbrace{\E\Big[\ind\{\bar{\CE}\}\sum_{t=N}^{N+D_{\max}-1}c(x_t^{\tt alg},u_t^{\tt alg})\Big]}_{\bar{R}_1}\\+\underbrace{\E\Big[\ind\{\bar{\CE}\}\sum_{t=N+D_{\max}}^{N_a-1}c(x_t^{\tt alg},u_t^{\tt alg})\Big]}_{\bar{R}_2}+\underbrace{\E\Big[\ind\{\bar{\CE}\}\sum_{t=N_a}^{T-1}c(x_t^{\tt alg},u_t^{\tt alg})\Big]}_{\bar{R}_3}.
\end{multline}
Now, for any $t\in\Z_{\ge0}$, we have 
\begin{align}\nonumber
	\nm[\big]{x_{t}^{\tt alg}}&=\nm[\Big]{\sum_{k=0}^{t-1}A^{t-(k+1)}\big(w_k+Bu_k^{\tt alg}\big)}\\\nonumber
	&\le\big(\max_{0\le k\le t-1}\norm{w_k}+\Gamma\max_{0\le k\le t-1}\nm[\big]{u_k^{\tt alg}}\big)\sum_{k=0}^{t-1}\nm[big]{A^k}\\
	&\le\big(\max_{0\le k\le t-1}\norm{w_k}+\Gamma\max_{0\le k\le t-1}\nm[\big]{u_k^{\tt alg}}\big)\frac{\kappa}{1-\gamma},
\end{align}
which implies that 
\begin{align}
	\nm[\big]{x_t^{\tt alg}}^2\le\big(\max_{0\le k\le t-1}\norm{w_k}^2+\Gamma^2\max_{0\le k\le t-1}\nm[\big]{u_k^{\tt alg}}^2\big)\frac{2\kappa^2}{(1-\gamma)^2}.\label{eqn:upper bound on x_t alg general}
\end{align}
Noting that $\Prob(\bar{\CE})\le2/T$, and recalling from Algorithm~\ref{algorithm:control design} that $u_t^{\tt alg}\overset{\text{i.i.d.}}\sim\CN(0,\sigma_u^2I_m)$ for all $t\in\{0,\dots,N-1\}$, we obtain from \eqref{eqn:upper bound on x_t alg general} and Lemma~\ref{lemma:expected max gaussian} that 
\begin{align}
	\bar{R}_0\le\frac{20N\kappa^2}{(1-\gamma)^2T}\log\frac{3T}{2}\big(\sigma_w^2n+\sigma_u^2m\big).\label{eqn:upper bound on bar R_0}
\end{align}
Similarly, recalling from Algorithm~\ref{algorithm:control design} that $u_t^{\tt alg}=0$ for all $t\in\{N,\dots,N+D_{\max}-1\}$, we have that 
\begin{align}
	\bar{R}_1\le\frac{20D_{\max}\kappa^2}{(1-\gamma)^2T}\log\frac{3T}{2}\big(\sigma_w^2n+\sigma_u^2m\big).\label{eqn:upper bound on bar R_1}
\end{align}
Moreover, we see from line~10 of Algorithm~\ref{algorithm:control design} that $\nm[\big]{x_t^{\tt alg}}\le\sqrt{p}R_x$ and $\nm[\big]{u_t^{\tt alg}}\le\sqrt{p}R_u$ for all $t\in\{N+D_{\max},\dots,N_a-1\}$, which implies that 
\begin{align}
	\bar{R}_2\le\frac{2N_a}{T}(R_x+R_u).\label{eqn:upper bound on bar R_2}
\end{align}
Furthermore, noting line~11 of Algorithm~\ref{algorithm:control design}, we have that for any $t\in\{N_a+1,\dots,T-1\}$,
\begin{align}\nonumber
	\nm[\big]{x_t^{\tt alg}}^2&=\nm[\big]{A^{t-N_a}x_{N_a}^{\tt alg}+\sum_{k=N_a}^{t-1}A^{t-(k+1)}(w_k+Bu_k^{\tt alg})}^2\\
	&\le3\nm[\big]{A^{t-N_a}x_{N_a}^{\tt alg}}^2+3\nm[\Big]{\sum_{k=N_a}^{t-1}A^{t-(k+1)}w_k}^2+3\nm[\big]{A^{t-(N_a+1)}u_{N_a}^{\tt alg}}^2\\\nonumber
	&\le3(\kappa\gamma^{t-N_a})^2\nm[\big]{x_{N_a}^{\tt alg}}^2+\frac{3\kappa^2}{(1-\gamma)^2}\max_{N_a\le k\le t-1}\norm{w_k}+3(\kappa\gamma^{t-(N_a+1)})^2\nm[\big]{u_{N_a}^{\tt alg}}^2.
\end{align}
Denoting the right-hand sides of \eqref{eqn:upper bound on norm of x_Na} and \eqref{eqn:upper bound on norm of u_Na} as $\bar{R}_x$ and $\bar{R}_u$, respectively, and invoking Lemma~\ref{lemma:expected max gaussian}, one can show that 
\begin{align}\label{eqn:upper bound on bar R_3}
	\bar{R}_3\le\sigma_1(Q)\Big(\bar{R}_x+(T-N_a-1)\big((3\kappa^2\gamma^2+1)\bar{R}_x+3\kappa^2\gamma^2\bar{R}_u\big)\Big).
\end{align}
Combining \eqref{eqn:upper bound on bar R_0}-\eqref{eqn:upper bound on bar R_2} and \eqref{eqn:upper bound on bar R_3} together, we prove Lemma~\ref{lemma:failure regret}.$\hfill\blacksquare$

\section{Auxiliary Proof and Lemmas}\label{sec:technical lemma}
\subsection{Proof of Lemma~\ref{lemma:expected state norm bound}}
Based on the definition of Algorithm~\ref{algorithm:control design}, we split our arguments for the proof of \eqref{eqn:upper bound on norm of x_Na} into $N_a=N+D_{\max}$ and $N+D_{\max}<N_a\le T-1$. First, supposing $N_a=N+D_{\max}$, we have
\begin{align}\nonumber
\nm[\big]{x_{N_a}^{\tt alg}}&=\nm[\Big]{\sum_{t=0}^{N_a-1}A^{N_a-(t+1)}\big(w_t+Bu_t^{\tt alg}\big)}\\\nonumber
&\le\max_{0\le t\le N_a-1}\nm[\big]{w_t+Bu_t^{\tt alg}}\sum_{t=0}^{N_a-1}\nm[\big]{A^{N_a-(t+1)}}\\\nonumber
&\le\Big(\max_{0\le t\le N_a-1}\norm{w_t}+\Gamma\max_{0\le t\le N_a-1}\nm[\big]{u_t^{\tt alg}}\Big)\frac{\kappa}{1-\gamma}\\\nonumber
&\le\Big(\max_{0\le t\le T-1}\norm{w_t}+\Gamma\max_{0\le t\le N-1}\nm[\big]{u_t^{\tt alg}}\Big)\frac{\kappa}{1-\gamma},
\end{align}
where the last inequality follows from the definition of Algorithm~\ref{algorithm:control design}, and we note that $u_t^{\tt alg}\overset{\text{i.i.d.}}{\sim}\CN(0,\sigma_u^2I)$. Since $\Prob(\ind\{N_a\le T-1\})\le 2/T$, we see from Lemma~\ref{lemma:expected max gaussian} that 
\begin{align}\nonumber
\E\Big[\ind\{N_a\le T-1\}\max_{0\le t\le T-1}\norm{w_t}^2\Big]&\le \frac{10\sigma_w^2n}{T}\log3T,\\\nonumber
\E\Big[\ind\{N_a\le N-1\}\max_{0\le t\le T-1}\norm{u_t^{\tt alg}}^2\Big]&\le \frac{10\sigma_u^2m}{T}\log 3N.
\end{align}
It then follows that 
\begin{align}\nonumber
\E\Big[\nm[\big]{x_{N_a}^{\tt alg}}^2\ind\{N_a=N+D_{\max}\}\Big]&\le\frac{2\kappa^2}{1-\gamma}\Big(\E\Big[\ind\{N_a\le T-1\}\max_{0\le t\le T-1}\norm{w_t}^2\Big]\\\nonumber
&\qquad\qquad\qquad+\Gamma^2\E\Big[\ind\{N_a\le N-1\}\max_{0\le t\le T-1}\norm{u_t^{\tt alg}}^2\Big]\Big)\\
&\le\frac{20\kappa^2(1+\Gamma^2)\overline{\sigma}^2n\log3T}{(1-\gamma)T}.\label{eqn:upper bound on x_Na 1}
\end{align}
Next, suppose $N+D_{\max}<N_a\le T-1$. Noting that $\nm[\big]{x_{t,i}^{\tt alg}}\le R_x$ and $\nm[\big]{u_{t,i}^{\tt alg}}\le R_u$ for all $t<N_a$ and all $i\in\V$, we have
\begin{align}\nonumber
\nm[\big]{x_{N_a}^{\tt alg}}&=\nm[\big]{Ax_{N_a-1}^{\tt alg}+Bu_{N_a-1}^{\tt alg}+w_{N_a-1}}\\\nonumber
&\le\Gamma\sqrt{p}(R_x+R_u)+\max_{0\le t\le T-1}\norm{w_t},
\end{align}
which implies that 
\begin{align}
\E\Big[\nm[\big]{x_{N_a}^{\tt alg}}^2\ind\{N+D_{\max}<N_a\le T-1\}\Big]\le\frac{2\Gamma^2p(R_x+R_u)^2}{T}+\frac{10\sigma_w^2n}{T}\log3T.\label{eqn:upper bound on x_Na 2}
\end{align}
Combining \eqref{eqn:upper bound on x_Na 1}-\eqref{eqn:upper bound on x_Na 2} together completes the proof of \eqref{eqn:upper bound on norm of x_Na}.

We then prove \eqref{eqn:upper bound on norm of u_Na}. Consider $N+D_{\max}\le N_a\le T-1$. We recall from Algorithm~\ref{algorithm:control design} that $u_{N_a}^{\tt alg}=u_{N_a}(M_{N_a}|\hat{w}_{0:N_a-1})=\sum_{s\in\U}\sum_{k=1}^hI_{\V,s}M_{N_a,s}^{[k]}\hat{\eta}_{N_a-k,s}$, where $\hat{\eta}_{N_a-k,s}=\begin{bmatrix}\hat{w}_{N_a-k-l_{vs},j_v}^{\top}\end{bmatrix}_{v\in\CL_s}^{\top}$ with $\hat{w}_{N_a-k-l_{vs},j_v}$ given by \eqref{eqn:est w_i t}. It follows that 
\begin{align}
\nm[\big]{u_{N_a}^{\tt alg}}^2\le\Big(\sum_{s\in\U}\sum_{k=1}^{h}\nm[\big]{M_{N_a,s}^{[k]}}\nm[\big]{\hat{\eta}_{N_a-k,s}}\Big)^2.\label{eqn:upper bound on u_Na 1}
\end{align}
Moreover, consider any $s\in\U$, any $t\in\{N_a-k-l_{vs}:k\in[h],v\in\CL_s\}$, and any $j_v$ with $w_{j_v}\to v$, we have from \eqref{eqn:est w_i t} that 
\begin{align}\nonumber
\nm[\big]{\hat{w}_{t,j_v}}&\le\nm[\big]{x_{t+1,j_v}^{\tt alg}}+\norm{\hat{A}_j}\nm[\big]{x_{t,\CN_{j_v}}^{\tt alg}}+\norm{\hat{B}_j}\nm[\big]{u_{t,\CN_{j_v}}^{\tt alg}}\\
&\le\nm[\big]{x_{t+1,j_v}^{\tt alg}}+(\Gamma+\bar{\varepsilon})\sqrt{p}(R_x+R_u).\label{eqn:upper bound on w hat}
\end{align}
Now, suppose $t=N_a-1$. We see from \eqref{eqn:upper bound on w hat} that 
\begin{align}\nonumber
\nm[\big]{\hat{w}_{t,j_v}}\le\nm[\big]{x_{t+1,j_v}^{\tt alg}}+(\Gamma+\bar{\varepsilon})\sqrt{p}(R_x+R_u),
\end{align}
which implies via \eqref{eqn:upper bound on norm of x_Na} and the fact $\Prob(\ind\{N_a\le T-1\})\le 1/T$ that 
\begin{align}
\E\Big[\nm[\big]{\hat{w}_{t,j_v}}^2\ind\{N+D_{\max}\le N_a\le T-1\}\Big]\le\frac{2\Gamma^2p+2(\Gamma+\bar{\varepsilon})^2p}{T}(R_x+R_u)^2+\frac{30\kappa^2(1+\Gamma^2)\overline{\sigma}^2n\log3T}{(1-\gamma)T}.\label{eqn:upper bound on w_jv hat 1}
\end{align}
Supposing $t<N_a-1$ and noting from Algorithm~\ref{algorithm:control design} that $\norm{x_{t+1}^{\tt alg}}\le R_x$, we have from \eqref{eqn:upper bound on w hat} that 
\begin{align}\nonumber
\nm[\big]{\hat{w}_{t,j_v}}&\le\nm[\big]{x_{t+1,j_v}^{\tt alg}}+(\Gamma+\bar{\varepsilon})\sqrt{p}(R_x+R_u)\\\nonumber
&\le R_x+(\Gamma+\bar{\varepsilon})\sqrt{p}(R_x+R_u),
\end{align}
which implies that 
\begin{align}
\E\Big[\nm[\big]{\hat{w}_{t,j_v}}^2\ind\{N+D_{\max}\le N_a\le T-1\}\Big]\le\frac{2R_x^2}{T}+\frac{2(\Gamma+\bar{\varepsilon})^2p}{T}(R_x+R_u)^2.\label{eqn:upper bound on w_jv hat 2}
\end{align}
Noting from Eq.~\eqref{eqn:the class of DFCs} that $\nm[\big]{M_{N_a,s}^{[k]}}\le2\sqrt{n}\kappa p\Gamma^{2D_{\max}}$ for all $s\in\U$ and all $k\in[h]$, and recalling that $\hat{\eta}_{N_a-k,s}=\begin{bmatrix}\hat{w}_{N_a-k-l_{vs},j_v}^{\top}\end{bmatrix}_{v\in\CL_s}^{\top}$ (with $|\CL_s|\le p$) for all $v\in\CL_s$ and all $s\in\U$, one can plug \eqref{eqn:upper bound on w_jv hat 1}-\eqref{eqn:upper bound on w_jv hat 2} into \eqref{eqn:upper bound on u_Na 1} and show that \eqref{eqn:upper bound on norm of u_Na} holds.$\hfill\blacksquare$

\subsection{Auxiliary Lemmas}
\begin{lemma}\label{lemma:matrix power perturbation bound}
\cite[Lemma~5]{mania2019certainty}
Consider any matrix $M\in\R^{n\times n}$ and any matrix $\Delta\in\R^{n\times n}$. Let $\kappa_M\in\R_{\ge1}$ and $\gamma_M\in\R_{>0}$ be such that $\gamma_M>\rho(M)$, and $\norm{M^k}\le\kappa_M\gamma_M^k$ for all $k\in\Z_{\ge0}$. Then, for all $k\in\Z_{\ge0}$,
\begin{equation*}
\norm{(M+\Delta)^k-M^k}\le k\kappa_M^2(\kappa_M\norm{\Delta}+\gamma_M)^{k-1}\norm{\Delta}.
\end{equation*}
\end{lemma}

\begin{lemma}\label{lemma:expected max gaussian}
Let $\CE$ be an probabilistic event with $\Prob(\CE)\le\delta$, where $0<\delta<1$, and let $w_t\overset{\text{i.i.d.}}{\sim}\CN(0,\sigma_w^2I_n)$ for all $t\in\{0,\dots,T-1\}$. Then, for any $T>1$, it holds that 
\begin{align}\nonumber
\E\Big[\ind\{\CE\}\max_{0\le t\le T-1}\norm{w_t}^2\Big]\le5\sigma_w^2n\delta\log\frac{3T}{\delta}.
\end{align}
\end{lemma}
\begin{proof}
The proof first relies on the following quadratic form of Gaussian random vector proved in \cite{hsu2012tail}:
\begin{equation*}
\mathbb{P}\Big(\norm{Ax}^2>\text{Tr}(\Sigma)+2\sqrt{\text{Tr}(\Sigma^2)z}+2\norm{\Sigma}z\Big)\le e^{-z}\quad\forall z\in\R_{\ge0},
\end{equation*}	
where $x\sim\CN(0,I_n)$, $A\in\R^{m\times n}$ and $\Sigma=A^{\top}A$. Next, following \cite[Lemma~34]{cassel2020logarithmic}, we know that 
\begin{equation*}
\max_{0\le t\le T-1}\norm{w_t}\le\sigma\sqrt{5n\log\frac{T}{\delta}}.
\end{equation*}
The rest of the proof follows from \cite[Lemma~35]{cassel2020logarithmic}.
\end{proof}

\end{document}